\documentclass[a4paper,12pt]{amsart}
\usepackage{environ}
\usepackage{fancyhdr}
\usepackage{mabliautoref}
\usepackage{amssymb,amsthm,amsmath}
\RequirePackage[dvipsnames,usenames]{xcolor}
\usepackage{hyperref}
\usepackage{cleveref}
\usepackage{mathtools}
\usepackage{tcolorbox}
\usepackage[all]{xy}
\usepackage{tikz}
\usepackage{tikz-cd}
\usetikzlibrary{decorations.markings,shapes,positioning}
\usepackage{enumitem}
\usepackage{stmaryrd}
\usepackage{xfrac}
\usepackage{fix-cm}
\usepackage{faktor}
\usepackage{minibox}
\usepackage{commath}
\usepackage{chngcntr}
\usepackage{setspace}
\usepackage{array}
\usepackage[toc,page]{appendix}
\usepackage{calligra,mathrsfs}
\usepackage{mathtools}

\hypersetup{
	bookmarks,
	bookmarksdepth=3,
	bookmarksopen,
	bookmarksnumbered,
	pdfstartview=FitH,
	colorlinks,backref,hyperindex,
	linkcolor=Sepia,
	anchorcolor=BurntOrange,
	citecolor=Cyan,
	filecolor=BlueViolet,
	menucolor=Yellow,
	urlcolor=OliveGreen
}

\newcommand{\quot}[2]{%
	\raise1ex\hbox{$#1$}\Big/\lower1ex\hbox{$#2$}%
}

\newcommand{\colim}{\varinjlim}
\renewcommand{\lim}{\varprojlim}

\newcommand{\expl}[2]{\underset{\mathclap{\minibox[c]{$\uparrow$\\ \fbox{\footnotesize #2}}}}{#1}}

\newcommand{\ra}{\rightarrow}
\newcommand{\dra}{\dashrightarrow}
\newcommand{\surj}{\twoheadrightarrow}
\newcommand{\inj}{\hookrightarrow}
\newcommand{\ttilde}{\widetilde}

\newcommand{\HHom}{\mathcal{H}om}
\newcommand{\EExt}{\mathcal{E}xt}
\newcommand{\et}{\acute{e}t}

\newcommand{\inc}{\subseteq}

\newcommand{\cni}{\supseteq}

\newcommand{\nor}{\unlhd}
\newcommand{\esp}{\mbox{ }}
\newcommand{\bighat}{\widehat}

\DeclareMathAlphabet{\mathchanc}{OT1}{pzc}%
{m}{it}                      

\renewcommand{\set}[2]{\{ \ #1 \  | \ #2 \ \}}
\newcommand{\bigset}[2]{\big\{ \ #1 \  \big| \ #2 \ \big\}}

\newcommand{\biggset}[2]{\bigg\{ \ #1 \  \bigg| \ #2 \ \bigg\}}

\newcommand{\bA}{\mathbb{A}}

\newcommand{\bC}{\mathbb{C}}
\newcommand{\bD}{\mathbb{D}}

\newcommand{\bF}{\mathbb{F}}

\newcommand{\bP}{\mathbb{P}}
\newcommand{\bQ}{\mathbb{Q}}

\newcommand{\bZ}{\mathbb{Z}}

\newcommand{\scr}{\mathcal}
\newcommand{\cA}{\scr{A}}
\newcommand{\cB}{\scr{B}}
\newcommand{\cC}{\scr{C}}

\newcommand{\cE}{\scr{E}}
\newcommand{\cF}{\scr{F}}
\newcommand{\cG}{\scr{G}}
\newcommand{\cH}{\scr{H}}
\newcommand{\cI}{\scr{I}}
\newcommand{\cJ}{\scr{J}}

\newcommand{\cL}{\scr{L}}
\newcommand{\cM}{\scr{M}}
\newcommand{\cN}{\scr{N}}
\newcommand{\cO}{\scr{O}}

\newcommand{\cR}{\scr{R}}

\newcommand{\cV}{\scr{V}}

\DeclareMathOperator{\nil}{{nil}}

\DeclareMathOperator{\Nil}{{Nil}}

\DeclareMathOperator{\ind}{{ind}}

\DeclareMathOperator{\Loc}{Loc}
\DeclareMathOperator{\Zar}{Zar}

\DeclareMathOperator{\codim}{codim}
\DeclareMathOperator{\coker}{{coker}}

\DeclareMathOperator{\Ext}{Ext}

\DeclareMathOperator{\Frac}{Frac}

\DeclareMathOperator{\Hom}{Hom}

\DeclareMathOperator{\im}{{im}}

\DeclareMathOperator{\Pic}{Pic}

\DeclareMathOperator{\GL}{{GL}}

\DeclareMathOperator{\red}{red}

\DeclareMathOperator{\Spec}{{Spec}}

\DeclareMathOperator{\Supp}{{Supp}}
\DeclareMathOperator{\Stab}{{Stab}}

\DeclareMathOperator{\Perv}{Perv}

\DeclareMathOperator{\Crys}{Crys}
\DeclareMathOperator{\coh}{coh}

\DeclareMathOperator{\RGamma}{R\Gamma}
\DeclareMathOperator{\unit}{unit}

\DeclareMathOperator{\Mod}{Mod}

\DeclareMathOperator{\indcoh}{indcoh}
\DeclareMathOperator{\IndCoh}{IndCoh}
\DeclareMathOperator{\IndCrys}{IndCrys}
\DeclareMathOperator{\indcrys}{indcrys}

\DeclareMathOperator{\qcoh}{qcoh}

\DeclareMathOperator{\Sh}{Sh}
\DeclareMathOperator{\QCoh}{QCoh}
\DeclareMathOperator{\QCrys}{QCrys}
\DeclareMathOperator{\crys}{crys}
\DeclareMathOperator{\LNil}{LNil}
\DeclareMathOperator{\RH}{RH}
\DeclareMathOperator{\Coh}{Coh}
\DeclareMathOperator{\Sol}{Sol}

\DeclareMathOperator{\RHom}{RHom}

\DeclareMathOperator{\ShHom}{\mathscr{H}\text{\kern -3pt {\calligra\large om}}\,}
\DeclareMathOperator{\sep}{sep}  

\DeclareFontFamily{OT1}{pzc}{}
\DeclareFontShape{OT1}{pzc}{m}{it}{<-> s * [1.200] pzcmi7t}{}
\DeclareMathAlphabet{\mathpzc}{OT1}{pzc}{m}{it}

\renewcommand{\phi}{\varphi}

\newcommand{\factor}[2]{\left. \raise 2pt\hbox{\ensuremath{#1}} \right/
	\hskip -2pt\raise -2pt\hbox{\ensuremath{#2}}}

\makeatletter
\renewcommand\subsection{
	\renewcommand{\sfdefault}{pag}
	\@startsection{subsection}%
	{2}{0pt}{.8\baselineskip}{.4\baselineskip}{\raggedright
		\sffamily\itshape\small\bfseries
}}
\renewcommand\section{
	\renewcommand{\sfdefault}{phv}
	\@startsection{section} %
	{1}{0pt}{\baselineskip}{.8\baselineskip}{\centering
		\sffamily
		\scshape
		\bfseries
}}
\makeatother

\definecolor{gr}{rgb}{0,0.5,0}

\newcommand{\ptofpf}[1]{{\ \\[-9pt] \noindent\color{gr}\fbox{\textit{\it\color{black} #1}}}}

\newcommand{\fm}{\mathfrak{m}}

\newcommand{\fp}{\mathfrak{p}}

\newcommand{\Addresses}{{
		\bigskip
		\footnotesize
		
		\textsc{\'Ecole Polytechnique F\'ed\'erale de Lausanne, SB MATH CAG, MA C3 615 (B\^atiment MA), Station 8, CH-1015 Lausanne, Switzerland}\par\nopagebreak
		\textit{E-mail address}: \texttt{jefferson.baudin@epfl.ch}

}}

\author[J.~Baudin]{Jefferson Baudin} 
\date{}

\usepackage[margin=1.0in]{geometry}
\setlist{  
	listparindent=\parindent,
	parsep=0pt,
}

\newcommand{\Acal}{\mathcal{A}}
\newcommand{\Bcal}{\mathcal{B}}

\newcommand{\Pcal}{\mathcal{P}}
\newcommand{\Mcal}{\mathcal{M}}
\newcommand{\Ncal}{\mathcal{N}}
\newcommand{\Fcal}{\mathcal{F}}
\newcommand{\Scal}{\mathcal{S}}

\newcommand{\Jcal}{\mathcal{J}}
\newcommand{\Ical}{\mathcal{I}}
\newcommand{\Rcal}{\mathcal{R}}
\newcommand{\Ocal}{\mathcal{O}}

\newcommand{\Lcal}{\mathcal{L}}

\newcommand{\Ff}{\mathbb{F}}

\newcommand{\Zz}{\mathbb{Z}}

\newcommand{\Dd}{\mathbb{D}}

\newcommand{\pfr}{\mathfrak{p}}
\newcommand{\qfr}{\mathfrak{q}}

\newcommand{\mfr}{\mathfrak{m}}


\title[Duality between Cartier crystals and perverse $\bF_p$-sheaves]{Duality between Cartier crystals and perverse $\Ff_p$-sheaves, and application to generic vanishing}
\begin{document}
	\maketitle
	
	\begin{abstract}
		We show that on any separated, Noetherian and $F$-finite $\mathbb{F}_p$-scheme, there is an anti-equivalence of categories between Cartier crystals and étale perverse $\mathbb{F}_p$-sheaves, commuting with derived proper pushforwards. We use this duality to construct an upper shriek functor for Cartier crystals, and give new proofs of Kashiwara's equivalence and the finite length of Cartier crystals.
				
		Finally, we deduce a generic vanishing statement for perverse $\overline{\mathbb{F}}_p$-sheaves on abelian varieties of characteristic $p > 0$, reminiscent of the characteristic zero and $l$-adic statements.
	\end{abstract}
	
\tableofcontents

\section{Introduction}
\subsection{Generic vanishing for perverse $\overline{\bF}_p$-sheaves}

%
%

Along the previous decades, several interesting vanishing theorems were proven for perverse sheaves on abelian varieties. In characteristic zero, we have for example \cite{Kramer_Weissauer_Vanishing_theorems_for_constructible_sheaves_on_abelian_varieties}, \cite{Schnell_Holonomic_D_modules_on_abelian_varieties}, \cite{Popa_Schnell_Generic_vanishing_theory_via_mixed_Hodge_modules},  \cite{Bhatt_Schnell_Scholze_Vanishing_theorems_for_perverse_sheaves_on_abelian_varieties_revisited}. In positive characteristic, we have \cite{Gabber_Loeser_Faisceaux_Pervers_l_adiques_sur_un_tore}, \cite{Weissauer_Vanishing_theorems_for_constructible_sheaves_on_abelian_varieties_over_finite_fields} and \cite{Esnault_Kerz_etale_cohomology_of_rank_one_l_adic_local_systems_in_pos_char} that deals with the case of $l$-adic sheaves, where $l$ is prime to the characteristic. However, we have no result for $\bF_p$-sheaves in characteristic $p > 0$. One objective of this paper is to fill this gap. \\

Let us first explain the setup. Let $A$ be an abelian variety over an algebraically closed field of characteristic $p > 0$, let $\bighat{A}$ denote its dual abelian variety, and let $\Loc^1(A_{\et}, \overline{\bF}_p)$ denote the group of étale $\overline{\bF}_p$-local systems of rank one on $A$.

Given a complex of étale $\overline{\Ff}_p$-sheaves $\cF^{\bullet}$, we would like to understand the loci
\[ S^i(\Fcal^{\bullet}) \coloneqq \bigset{\Lcal \in \Loc^1(A_{\et},  \overline{\bF}_p)}{H^i(X_{\et}, \Fcal^{\bullet} \otimes \Lcal) \neq 0}. \] The point is that the group $\Loc^1(A_{\et},  \overline{\bF}_p)$ can be naturally identified with $\bighat{A}^{(p)}$, the subgroup of prime-to-$p$ elements of $\bighat{A}$. Indeed, let $\cO_{A_{\et}}$ denote the extension on the Zariski sheaf $\cO_A$ to the étale site. Then the function 

\[ \begin{tikzcd}[row sep = tiny]
	{\Loc^1(A_{\et}, \overline{\bF}_p)} \arrow[rr] &  & \bighat{A}                                                    \\
	\cL \arrow[rr, maps to]                        &  & \left(\cL \otimes_{\overline{\bF}_p} \cO_{A_{\et}}\right)\Big|_{A_{\Zar}}
\end{tikzcd} \] 
induces a group isomorphism $\Loc^1(A_{\et}, \overline{\bF}_p) \cong \bighat{A}^{(p)}$ (see \autoref{local systems and torsion line bundles}). Thus, for all $i \in \bZ$, we can think of $S^i(\cF^{\bullet})$ as a subset of $\bighat{A}$. In particular, we can consider its closure $\overline{S^i(\cF^{\bullet})} \inc \bighat{A}$, and try to understand it. 

As for the already known generic vanishing theorems, we cannot do this for any complex $\cF^{\bullet}$, but we can prove something when $\cF^{\bullet}$ is \emph{perverse and constructible}. 

\begin{itemize}
	\item The constructibility notion is a finiteness condition, analogous to that of coherent $\cO_X$-modules (see \cite[{\href{https://stacks.math.columbia.edu/tag/03RW}{Tag 03RW}}]{Stacks_Project} for the definition). In particular, constructible $\overline{\bF}_p$-sheaves are Noetherian, and any $\overline{\bF}_p$-sheaf is a union of constructible $\overline{\bF}_p$-subsheaves (see \cite[{\href{https://stacks.math.columbia.edu/tag/09YV}{Tag 09YV}} and {\href{https://stacks.math.columbia.edu/tag/0F0N}{Tag 0F0N}}]{Stacks_Project}). In our case, we ask that each cohomology sheaf of the complex $\cF^{\bullet}$ is constructible.
	\item Being perverse is a combination of two conditions. 
	
	\begin{enumerate}
		\item The first one is about bounding the dimension of the support of $\cF^{\bullet}$: we want that for all $i \in \bZ$, $\dim(\cH^i(\cF^{\bullet})) \leq -i$, where $\cH^i(\cF^{\bullet})$ denotes the $i$'th cohomology sheaf of the complex $\cF^{\bullet}$.
	
		\item The second condition is a vanishing condition on its local cohomology: for all point $i_x \colon x \to A$,  \[ \cH^j(i^!_{\overline{\{x\}}}\cF^{\bullet})_{\overline{x}} = 0 \: \: \forall j > \dim\overline{\{x\}}. \] In particular, it depends on the whole complex $\cF^{\bullet}$ and not only on its cohomology sheaves (unlike the first condition). 
	\end{enumerate}
\end{itemize} 

A motivation for this definition is that if we impose the exact same conditions on complexes of coherent $\cO_X$-modules, then we obtain the complexes whose Grothendieck dual is only supported in degree $0$. By this analogy, the following example of perverse sheaf is not suprising: if $i \colon Z \to A$ is a closed immersion from a smooth subvariety and $\cL$ is a local system on $Z$, then $i_*\cL[\dim Z]$ is perverse (and constructible).

Similarly to the $l$-adic and characteristic zero cases, perverse $\overline{\bF}_p$-sheaves are only supported in degrees $\{-\dim(A), \dots, 0\}$. However, there is no Verdier duality, since constructibility is not preserved under $\EExt^i$ (see \autoref{ex:RHom_does_not_preserve_constructible_sheaves}). Furthermore, if $\cF^{\bullet}$ is a perverse $\overline{\bF}_p$-sheaf, then \[ H^i(A_{\et}, \cF^{\bullet}) = 0 \esp \forall i \notin \{-\dim(A), \dots, 0\}.\] This contrasts again the $l$-adic or characteristic zero cases, since for example for any $l$-adic perverse sheaf $\cG^{\bullet}$, \[ H^i(A_{\et}, \cG^{\bullet}) = 0 \: \forall i \notin \{-\dim(A), \dots, \dim(A)\}. \] The reason is that $\overline{\bF}_p$-sheaves behave much like coherent sheaves, due to Artin-Schreier sequences (and ultimately Riemann-Hilbert correspondences, see \cite{Emerton_Kisin_Riemann-Hilbert_correspondence}, \cite{Bockle_Pink_Cohomological_Theory_of_crystals_over_function_fields}, \cite{Bhatt_Lurie_RH_corr_pos_char}, \cite{Schedlmeier_Cartier_crystals_and_perverse_sheaves}). An other way to think about this different behaviour is that $\bF_p$-cohomology in characteristic $p > 0$ connects only to the slope zero part of the crystalline cohomology (i.e. the right analogue of de Rham cohomology), via $\bQ_p$-cohomology. On the other hand, in characteristic zero, cohomology with coefficients in $\bC$ captures the entire de Rham cohomology. \\

We obtain the following generic vanishing result:

\begin{theorem*}[\autoref{main thm generic vanishing}]
	Let $A$ be an abelian variety of dimension $g$ over an algebraically closed field of characteristic $p > 0$, and let $\Fcal^{\bullet}$ be an \'etale constructible perverse $\overline{\Ff}_p$-sheaf. Then we have
	\begin{enumerate}
		\item $S^i(\Fcal^{\bullet}) = \emptyset$ for all $i \notin \{-g, \dots, 0\}$;
		\item for all $0 \leq i \leq g$, \[ \codim \overline{S^{-i}(\Fcal^{\bullet})} \geq i.\] In particular, for a general element $\cL \in \Loc^1(A, \overline{\bF}_p)$, we have \[ H^i(A_{\et}, \Fcal^{\bullet} \otimes \cL) = 0 \: \: \: \forall i \neq 0, \]  
		and hence \[ \chi(A_{\et}, \cF^{\bullet} \otimes \cL) \coloneqq \sum_{i \in \bZ} (-1)^i\dim H^i(A_{\et}, \cF^{\bullet} \otimes \cL) \geq 0. \] 
	\end{enumerate}
	If $A$ is ordinary, then we also have:
	\begin{enumerate}[start=3]
		\item $S^{-g}(\Fcal^{\bullet}) \inc \dots \inc S^0(\Fcal^{\bullet})$;
		\item for all $0 \leq i \leq g$, the closed subset $\overline{S^i(\Fcal^{\bullet})}$ is a finite union of torsion translates of abelian subvarieties;
		\item $\chi(A_{\et}, \Fcal^{\bullet}) \geq 0$.
	\end{enumerate}
\end{theorem*}

Note that for any constructible perverse $\bF_{p^r}$-sheaf $\cG^{\bullet}$, we can apply the generic vanishing theorem above to its base change to $\overline{\bF}_p$. Hence, for all $r \geq 1$, we have a generic vanishing statement for all constructible perverse $\bF_{p^r}$-sheaves as well.

Our approach is to show that there exists an anti-equivalence of categories between perverse constructible sheaves and \emph{Cartier crystals}, and that this equivalence preserves taking derived pushforwards (in particular, cohomology groups). In \cite{Hacon_Pat_GV_Geom_Theta_Divs} and \cite{Baudin_Generic_vanishing_theory_in_positive_characteristic}, an analogous generic vanishing statement is proven for Cartier crystals, so we can use these results and our duality to deduce our theorem. In fact, most of this paper is about this duality.

\subsection{A duality between Cartier crystals and perverse constructible $\overline{\bF}_p$-sheaves}

A coherent Cartier module is a coherent sheaf $\cM$ on an $\bF_p$-scheme, together with a $p^{-1}$-linear endomorphism. An important example is the the sheaf of top forms $\omega_{X/k}$, for $X$ smooth over a perfect field. The $p^{-1}$-linear endomorphism comes from the Cartier isomorphism (see \cite{Cartier_une_nouvelle_operation_sur_les_formes_differentielles} or \cite[Example 2.0.3]{Blickle_Bockle_Cartier_crystals}).

The category of \emph{Cartier crystals} is the category of coherent Cartier modules up to nilpotence (with respect to their endomorphism). For example, a Cartier module with nilpotent endomorphism is trivial as a Cartier crystal. \\

Here is our main result:

\begin{theorem*}[\autoref{eq cat Cartier crystals and perverse Fp sheaves} and \autoref{thm_upper_shriek_functor}]
	Let $X$ be a Noetherian, $F$-finite and semi-separated $\bF_p$-scheme. Then there is a canonical equivalence of categories \[ D^b(\Crys_X^{C})^{op} \cong D^b_c(X_{\et}, \Ff_p), \]
	where $D^b(\Crys_X^{C})$ denotes the derived category of coherent Cartier crystals, and $D^b_c(X_{\et}, \Ff_p)$ denotes the derived category of \'etale $\Ff_p$-sheaves with constructible cohomology sheaves. \\
	
	Moreover, this equivalence sends the canonical t-structure on $D^b(\Crys_X^{C})$ to the perverse t-structure on $D^b_c(X_{\et}, \Ff_p)$, hence inducing an equivalence \[ (\Crys_X^C)^{op} \cong \Perv_c(X, \Ff_p). \]

	Finally, this duality preserves proper derived pushforwards.
\end{theorem*}

This is analogous to the Riemann-Hilbert correspondence in characteristic zero (\cite{Hotta_Takeuchi_Tanisaki_D_modules_perverse_sheaves_and_representation_theory}, \cite{Kashiwara_the_Riemann_Hilbert_problem_for_hoonomic_systems}).\\

Although we only stated a duality with perverse $\bF_p$-sheaves, we can make a similar duality for $\bF_{p^r}$-sheaves for any $r \geq 1$. Taking direct limits, we then obtain a result for perverse $\overline{\bF}_p$-sheaves, which we then apply for deduce the generic vanishing theorem mentioned before. \\

In fact, any dualizing complex $\omega_X^{\bullet}$ together with an isomorphism $\omega_X^{\bullet} \cong F^!\omega_X^{\bullet}$ naturally induces such a duality. We call such an object a \emph{unit dualizing complex}. Indeed, somewhat surprisingly, there is a unique way of making $\omega_X^{\bullet}$ into a complex of Cartier crystals. Our duality functor can be expressed by \[ \cM^{\bullet} \mapsto \cR\HHom_{\Crys_X^C}(\cM^{\bullet}, \omega_X^{\bullet}) \] (taken in the étale site). However, this is not the way we define it first to be able to show our main theorem (see two paragraphs below). \\

The existence of such a unit dualizing complex is guaranteed by ongoing work of Bhargav Bhatt, Manuel Blickle, Karl Schwede and Kevin Tucker. We only prove this existence in the special case of schemes which are of finite type over a Noetherian, $F$-finite and affine $\bF_p$-scheme (see \autoref{finite type over a unit dc has a unit dc}). \\

The duality from our main theorem was already known to exist by Tobias Schedlmeier if our scheme embeds in some smooth, finite type variety over a perfect field (see \cite{Schedlmeier_Cartier_crystals_and_perverse_sheaves}). However, due to the way it was defined, it was not clear how their correspondence behaves cohomologically. Therefore we chose to define this equivalence differently: instead of twisting by the dualizing complex and then applying the Riemann-Hilbert correspondence from \cite{Emerton_Kisin_Riemann-Hilbert_correspondence}, we first take $\Rcal\HHom$ with the dualizing complex, and then apply the Riemann-Hilbert correspondence from \cite{Bockle_Pink_Cohomological_Theory_of_crystals_over_function_fields} (see also \cite{Bhatt_Lurie_RH_corr_pos_char}). It turns out that Schedlmeier's construction and ours agree, see \autoref{Schedlmeier's construction and ours agree}.

\subsection{Other applications of the duality}

We construct an upper shriek functor on Cartier crystals:

\begin{theorem*}[\autoref{thm_upper_shriek_functor}]
	Let $f \colon X \to Y$ be a morphism of finite type between separated, Noetherian and $F$-finite $\bF_p$-schemes. Then there exists a unique functor \[ f^! \colon D^b(\Crys_Y^C) \to D^b(\Crys_X^C) \] with the following properties:
	\begin{itemize}
		\item under both duality functors on $X$ and $Y$, $f^!$ corresponds to the pullback $f^* \colon D^b_c(Y_{\et}, \bF_p) \to D^b_c(X_{\et}, \bF_p)$;
		\item if $f$ is an open immersion, then $f^! = (\cdot)|_X$;
		\item if $f$ is proper, then $f^!$ is the right adjoint to $Rf_*$;
		\item the functor $f^!$ is perverse $t$-exact.
	\end{itemize}
\end{theorem*}

Other consequences of our result include many new proofs of known results (at least to the experts):

\begin{theorem*}[\autoref{section_finiteness_and_exactness_results}]
	Let $X$ be a Noetherian, $F$-finite $\bF_p$-scheme. 
	\begin{itemize}
		\item\cite{Blickle_Bockle_Cartier_modules_finiteness_results} Any Cartier crystal on $X$ is both Noetherian and Artinian, and hence has finite length.
		\item \cite{Blickle_Bockle_Cartier_crystals} For any closed immersion $i \colon Z \inj X$, the functor $i_*$ induces an equivalence of categories between Cartier crystals on $Z$ and Cartier crystals on $X$, whose restriction to $X \setminus Z$ is trivial. The inverse is given by $i^!$.
		\item \cite{Bhatt_and_co_Applications_of_perv_sheaves_in_commutative_algebra} For any separated morphism $f \colon X \to Y$ of Noetherian $F$-finite $\bF_p$-schemes, the functor $Rf_! \colon D^b_c(X_{\et}, \bF_p) \to D^b_c(Y_{\et}, \bF_p)$ is right perverse $t$-exact. If $f$ is affine, then $Rf_!$ is furthermore perverse $t$-exact.
	\end{itemize}
\end{theorem*} 

In \autoref{section local duality}, we also prove a local duality theorem for both Cartier modules and Frobenius modules. \\

After the first draft of this paper was posted, some of our results were also proved in \cite{Bhatt_and_co_Applications_of_perv_sheaves_in_commutative_algebra}. Some important structural results about $\infty$--categories of Cartier and Frobenius modules are obtained in \cite{Klaus_Weiss_The_derived_infty_category_of_Cartier_modules}. We believe that they could lead to alternative contructions/proofs of upper-schriek functors and our duality.

\subsection{Acknowledgements}\label{sec:aknowledgements}

I would like to thank my supervisor Zsolt Patakfalvi for giving this subject and for his continuous support, it was very interesting to work on this. I also warmly thank Fabio Bernasconi and Léo Navarro Chafloque for giving many comments on earlier drafts of this paper. Finally, I thank Stefano Filipazzi for his support in the writing process, and Lucas Gerth for some interesting input to an earlier draft. \\

Financial support was provided by grant \#200020B/192035 from the Swiss National Science Foundation (FNS/SNF).

\subsection{Notations}
Let $\Acal$ be an abelian category, $R$ a ring and $X$ a scheme.
\begin{itemize}
	\item Throughout the paper, we fix a prime number $p > 0$.
	\item The symbol $F$ always denotes the absolute Frobenius.
	\item Given $R$ an $\bF_p$-algebra and $M$ an $R$-module, elements in $F_*M$ are denoted $F_*m$, with $m \in M$. By definition, we have the formula \[ \lambda F_*m = F_*(\lambda^pm) \] for all $m \in M$ and $\lambda \in R$.
	
	\item The category of $\Ocal_X$-modules (resp. $R$-modules) is denoted by $\Mod(\Ocal_X)$ (resp. $\Mod(R)$).
	
	\item The category of quasi-coherent (resp. coherent) sheaves on $X$ is denoted by $\QCoh_X$ (resp. $\Coh_X$).
	
	\item Given any $r \geq 1$, the category of $\Ff_{p^r}$-sheaves on the \'etale site is denoted $\Sh(X_{\et}, \Ff_{p^r})$, and its derived category is $D(X_{\et}, \Ff_{p^r})$. We use the same notations for $\overline{\bF}_p$. 
	
	\item Given a point $x \in X$, we write $\overline{x} \coloneqq \Spec k(x)^{\sep}$, where $k(x)^{\sep}$ denotes a separable closure of $k(x)$. We also define $\cO_{X, \overline{x}}$ to be the strict Henselization of $\cO_{X, x}$.
	
	\item The symbol $D(\Acal)$ denotes the derived category of $\cA$. Given $G \colon \cA \to \cB$ a left exact functor, its right derived functor is denoted $RG$ (when it exists).
	
	\item Let $\cR$ be a sheaf of rings on $X$, and let $\Acal = \Mod(\cR)$ be the category of left $\cR$-modules. We will write $D(\cR)$ instead of $D(\Mod(\cR))$.
	
	\item For a cochain complex $A^\bullet = (A^n)_{n \in \Zz}$ in $\Acal$, the $i$'th cohomology object of this cochain is denoted by $\cH^i(A^\bullet)$. 
	
	\item Given a weak Serre subcategory $\Bcal \inc \Acal$ (see \cite[{\href{https://stacks.math.columbia.edu/tag/02MO}{Tag 02MO}}]{Stacks_Project}), the triangulated category $D_\Bcal(\Acal)$ denotes the full subcategory of $D(\Acal)$ of complexes $A^\bullet$ such that $\cH^i(A^\bullet) \in \Bcal$ for all $i \in \Zz$ (see \cite[{\href{https://stacks.math.columbia.edu/tag/06UQ}{Tag 06UQ}}]{Stacks_Project}). 
	
	\item Given an abelian category $\cA$ and a Serre subcategory $\cB$ (see \cite[{\href{https://stacks.math.columbia.edu/tag/02MO}{Tag 02MO}}]{Stacks_Project}), the quotient will be denoted $\cA/\cB$ (see \cite[{\href{https://stacks.math.columbia.edu/tag/02MS}{Tag 02MS}}]{Stacks_Project}).
	
	\item We loosely denote $D_{coh}(\QCoh_X) \coloneqq D_{\Coh_X}(\QCoh_X)$. Throughout the paper, we will make comparable abuse of notations.
	
	\item The full subcategories $D^+(\Acal)$, $D^-(\Acal)$ $D^b(\Acal)$ denote those complexes $A^\bullet$ for which $\cH^i(A^\bullet) = 0$ for $i \ll 0$ (resp. $i \gg 0$ and $\abs{i} \gg 0$). 
			
	\item We can combine the two notations above.
	
	\item In order to lighten notations, whenever the context is clear, we shall remove the bullets in the notation of cochain complexes. However, the $\Hom$-complex between two complexes $A$ and $B$ will always be denoted $\Hom^{\bullet}(A, B)$.
	
	\item Let $M \in D^+(\Acal)$ and let $I$ be a complex of injectives which is quasi-isomorphic to $M$. We call  $I$ an \emph{injective resolution} of $M$.
\end{itemize}

\section{Preliminaries}\label{section preliminaries}

\subsection{Categorical preparations}

\begin{defn}\label{def of qcoh A-modules}
	Let $X$ be a scheme, and let $\Acal$ be a quasi-coherent sheaf of (non-necessarily commutative) $\Ocal_X$-algebras. This means that $\cA$ is a quasi-coherent sheaf of $\Ocal_X$-modules, together with a ring structure such that the map $\Ocal_X \to \Acal$ is a ring morphism. 
	
	A \emph{quasi-coherent $\Acal$-module} is the datum of a quasi-coherent $\Ocal_X$-module, together with a left action of $\Acal$ compatible with $\cO_X \to \cA$. The category of quasi-coherent $\Acal$-modules is denoted by $\QCoh_X^{\Acal}$. When $\Acal = \Ocal_X$, we simply write $\QCoh_X$.
\end{defn}

\begin{lem}\label{qcoh A-modules is a Grothendieck category}
	Let $X$ be a scheme and $\Acal$ a quasi-coherent sheaf of $\Ocal_X$-algebras. Then there is a set $S(\Acal)$ such that any quasi-coherent $\Acal$-module is the union of elements in $S(\Acal)$. 
	
	In particular, the category $\QCoh_X^{\Acal}$ is a Grothendieck category, and hence has enough injectives.
\end{lem}
\begin{proof}
	The first statement after ``In particular'' follows by definition, and the second one by \cite[{\href{https://stacks.math.columbia.edu/tag/079H}{Tag 079H}}]{Stacks_Project}. 
	
	In the case $\Acal = \Ocal_X$, this is \cite[Section 2 and Corollary 3.5]{Enochs_Estrada_Relative_homological_algebra_in_the_category_of_qcoh_sheaves}. For the general case, let us freely use their language and notations. Given the quiver they consider (a vertex for each open, and an edge for each inclusion of opens), we can define the representation $R_{\Acal}$ given by $R_{\Acal}(U) = \Acal(U)$ (and the restriction maps for the edges). Note that an object in $\QCoh(R_{\Acal})$ (see \cite[Section 2]{Enochs_Estrada_Relative_homological_algebra_in_the_category_of_qcoh_sheaves}) is also in $\QCoh(R_{\Ocal_X})$, since $\Acal$ is quasi-coherent. 
	
	From the equivalence $\QCoh(R_{\cO_X}) \cong \QCoh_X$, we deduce that $\QCoh(R_{\Acal}) \cong \QCoh_X^{\Acal}$. Thus, we conclude by \cite[Corollary 3.5]{Enochs_Estrada_Relative_homological_algebra_in_the_category_of_qcoh_sheaves}.
\end{proof}

\begin{lem}{\cite[Theorem 2.9.4]{Bockle_Pink_Cohomological_Theory_of_crystals_over_function_fields}}\label{D_B(A) = D(B) if B objects are B-acyclic}
	Let $\Acal$ be an abelian category, let $\Bcal$ be a weak Serre subcategory of $\Acal$ (see \cite[{\href{https://stacks.math.columbia.edu/tag/02MO}{Tag 02MO}}]{Stacks_Project}) such that the inclusion $\iota \colon \Bcal \ra \Acal$ admits a right adjoint $G$. Suppose furthermore that for all $B \in \Bcal$, $R^iG(B) = 0$ for all $i > 0$. 
	
	Then for $* \in \{+, b\}$, the natural triangulated functor $D^*(\Bcal) \ra D^*_{\Bcal}(\Acal)$ is an equivalence of triangulated categories with quasi-inverse $RG$.
\end{lem}
\begin{proof}
	For $* = +$, we have to show that for all $B \in D^+(\cB)$ and $A \in D^+_{\cB}(\cA)$, the natural morphisms $B \to RG(\iota(B))$ and $\iota(RG(A)) \to A$ are isomorphism. In both cases, this holds by assumption if the complex is supported in a single degree ($G \circ \iota = id$ since $\iota$ is fully faithful). The case of general complexes follows from a hypercohomology spectral sequence argument.
	
	The case $* = b$ is now deduced from the fact that for any $A \in D(\Bcal)$, $A \in D^b(\Bcal)$ if and only if $\iota(A) \in D^b(\Acal)$ (by assumption, $\iota$ is fully faithful and exact).
\end{proof}

\begin{lem}\label{derived_functors_and_localization}
	Let $\cA$ be a Grothendieck category, let $\cB \inc \cA$ be a Serre subcategory which is closed under filtered colimits, and let $G \colon \cA \to \cC$ be a left exact functor. Assume that $RG \colon D^+(\cA) \to D^+(\cC)$ annihilates the subcategory $D^+_{\cB}(\cA)$. Then there exist factorizations 
	
	\[
		\begin{tikzcd}
			\cA \arrow[rr] \arrow[rrrr, "G"', bend right] &  & \cA/\cB \arrow[rr, "\overline{G}"] &  & \cC
		\end{tikzcd}
	\]
	 and
	\[
		\begin{tikzcd}
			D^+(\cA) \arrow[rr] \arrow[rrrr, "RG"', bend right] &  & D^+(\cA/\cB) \arrow[rr, "R\overline{G}"] &  & D^+(\cC).
		\end{tikzcd}
	\]
\end{lem}
\begin{proof}
	Since for all $B \in \cB$, $RG(B) = 0$, we obtain by the universal property of localization that there exists functor $\overline{G} \colon \cA/\cB \to \cC$ giving the first factorization. 
	
	Let us show that $\overline{G}$ is left exact. Since $D^+(\cA) \to D^+(\cA/\cB)$ is essentially surjective (see \cite[Theorem 2.6.2]{Bockle_Pink_Cohomological_Theory_of_crystals_over_function_fields}), we have to show that if \[ 0 \to A_1 \to A_2 \to A_3 \] is a complex which becomes right exact in $\cA/\cB$, then \[ 0 \to G(A_1) \to G(A_2) \to G(A_3) \] is right exact. This is again straightforward from the fact that $RG$ annihilates elements in $\cB$.
	
	Let $f \colon A_1 \to A_2$ be a morphism whose cone is in $D^+_{\cB}(\cA)$. Then by assumption, $RG(f)$ is an isomorphism. Hence, we conclude the proof by \cite[Theorems 2.7.1 and 2.6.2]{Bockle_Pink_Cohomological_Theory_of_crystals_over_function_fields}.
\end{proof}

\begin{lem}\label{injective qcoh is injective O_X}
	Let $X$ be a locally Noetherian scheme, and let $\Ical$ be an injective object in the category of quasi-coherent sheaves. Then $\Ical$ is also an injective object in the category of $\Ocal_X$-modules.
\end{lem}
\begin{proof}
	By \cite[Theorem 7.18]{Hartshorne_Residues_and_Duality}, there exists a quasi-coherent sheaf $\Jcal$, injective as an $\Ocal_X$-module, such that $\Ical \inc \Jcal$. Since $\Jcal$ is quasi-coherent and $\Ical$ is injective as a quasi-coherent module, $\Ical$ is a direct summand of $\Jcal$. Thus, $\Ical$ is a direct summand of an injective $\Ocal_X$-module, so it is injective as an $\Ocal_X$-module as well.
\end{proof}

\subsection{Results about Frobenius modules}
Here, we collect mostly well-known results on Frobenius modules that will be useful later. We mostly prove basic categorical properties and recall important structural results on Frobenius modules, such as the Riemann-Hilbert correspondence. We point out that using \cite{Bhatt_Lurie_RH_corr_pos_char}, many results we state can be adapted to work in a non-Noetherian setting. Since we do not need such generality, we restrict to Noetherian schemes and use results from \cite{Bockle_Pink_Cohomological_Theory_of_crystals_over_function_fields}. \\

Fix $r \geq 1$, and let $q \coloneqq p^r$.

\begin{defn}\label{def F-module}
	Let $X$ be an $\Ff_p$-scheme, and $\Mcal$ an $\Ocal_X$-module over $X$.
	\begin{itemize}
		\item An \emph{$r$-Frobenius module} (or simply Frobenius module when the context is clear) structure on $\Mcal$ is a morphism of $\Ocal_X$-modules $\tau_{\cM} \colon \Mcal \ra F^r_*\Mcal$. A morphism of Frobenius modules is a morphism $\theta \colon \Mcal \ra \Ncal$ between two $\Ocal_X$-modules such that the square \[ \begin{tikzcd}
			\Mcal \arrow[r, "\theta"] \arrow[d, "\tau_{\cM}"'] & \Ncal \arrow[d, "\tau_{\cN}"] \\
			F^r_*\Mcal \arrow[r, "F^r_*\theta"'] & F^r_*\Ncal			
		\end{tikzcd} \] commutes. This forms the category of Frobenius modules, denoted $\Mod(\Ocal_X[F^r])$.
		\item The full subcategory of Frobenius modules which are quasi-coherent as $\cO_X$-modules is denoted by $\QCoh_X^{F^r}$.
		\item By adjunction, a morphism $\tau_{\cM} \colon \cM \to F^r_*\cM$ is equivalent to a morphism $\tau_{\cM}^* \colon F^{r, *}\cM \to \cM$. This morphism is called the \emph{adjoint structural morphism}. a Frobenius module is said to be \emph{unit} if its adjoint structural morphism is an isomorphism.
	\end{itemize}
\end{defn}
\begin{rem}\label{operations_on_F_modules}
	\begin{itemize}
		\item Unless stated otherwise, the Frobenius module structure on $\cM$ is always denoted $\tau_{\cM}$, and its adjoint structural morphism is denoted $\tau_{\cM}^*$.
		
		\item We can take pushforwards, pullbacks and tensor products of Frobenius modules. For pushforwards, this is immediate. 
		
		For pullbacks, this follows from using the adjoint structural morphism. Explicitly, if $f \colon X \to Y$ is a morphism, and $\cM$ is a Frobenius module on $Y$, then for all $x \in X$, the Frobenius module structure on $(f^*\cM)_x = \cM_{f(x)} \otimes_{\cO_{Y, f(x)}} \cO_{X, x}$ is given by $m \otimes \lambda \mapsto \tau_{\cM}(m) \otimes \lambda^q$. 
		
		Finally, given two Frobenius modules $\cM$ and $\cN$, the $\cO_X$-module $\cM \otimes \cN$ naturally acquires the structure of a Frobenius module via the formula $m \otimes n \mapsto \tau_{\cM}(m) \otimes \tau_{\cN}(n)$.
		\item It is immediate form the definitions that the projection formula holds for Frobenius modules.
	\end{itemize}
\end{rem}

Now we want to prove that the categories we have defined so far are Grothendieck categories. 

\begin{constr}
	For and $\bF_p$-algebra $R$, define $R[F^r]$ to be the set of polynomials of the form $\sum_{i = 1}^na_iF^{ri}$, where we impose the non-commutative relation $F^ra = a^qF^r$. 
	
	For and $\bF_p$-scheme $X$, define a sheaf $\Ocal_X[F^r]$ by $\Gamma(U, \Ocal_X[F^r]) \coloneqq \Ocal_X(U)[F^r]$ for all open $U \inc X$. It is immediate to this defines a quasi-coherent sheaf of $\cO_X$-algebras (see \autoref{def of qcoh A-modules}).
\end{constr}
\begin{obs*}
	There is an equivalence of categories between left $\Ocal_X[F^r]$-modules and Frobenius modules (multiplying by $F^r$ on the left corresponds to applying the structural morphism), inducing an equivalence of categories between $\QCoh_X^{\Ocal_X[F^r]}$ (see \autoref{def of qcoh A-modules}) and $\QCoh_X^{F^r}$.
\end{obs*}
\begin{cor}\label{qcoh F-modules is a Groth cat}
	For any scheme $X$ over $\Ff_p$, the category $\QCoh_X^{F^r}$ is a Grothendieck category.
\end{cor}
\begin{proof}
	This follows from the above observation and \autoref{qcoh A-modules is a Grothendieck category}.
\end{proof}

\begin{defn}
	Let $X$ be a Noetherian $\bF_p$-scheme, and let $\cM \in \QCoh_X^{F^r}$. Then $\cM$ is said to be \emph{ind-coherent} if $\cM$ is a fitered colimit of coherent Frobenius modules. The full category of these objects is denoted by $\IndCoh_X^{F^r}$.
\end{defn}
\begin{rem}\label{rem def ind-coherence}
	Since a filtered colimit of Frobenius modules is simply a filtered colimit of the underlying $\Ocal_X$-modules (with the natural Frobenius module structure), we see that ind-coherence is equivalent to being the union of its coherent Frobenius submodules.
\end{rem}
\begin{example}
	The quasi-coherent Frobenius module $\cO_X[F^r]$ is never ind-coherent. In addition, for $R$ a domain over $\bF_p$ which is not a field, $\Frac(R)$ (with its canonical Frobenius module structure) is never ind-coherent. We will see in \autoref{direct sum of residue fields is ind-coh} that Cartier modules behave very differently.
\end{example}

\begin{defn}
	Let $X$ be a Noetherian $\bF_p$-scheme, and let $\Mcal \in \QCoh_X^{F^r}$.
	\begin{itemize}
		\item The Frobenius module $\cM$ is said to be \emph{nilpotent} if $\tau_{\cM}^n = 0$ for some $n \geq 0$, where we make the abuse of notations \[ \tau_{\cM}^n \coloneqq  F^{r(n - 1)}_*\tau_{\cM} \circ \dots \circ F^r_*\tau_{\cM} \circ \tau_{\cM}. \]
		\item It is said to be \emph{locally nilpotent} if for all open affine $U \inc X$, $\Mcal|_U$ is a union of nilpotent Frobenius modules. The full subcategory of locally nilpotent modules is called $\LNil$ (note that $\LNil \subseteq \IndCoh_X^{F^r}$).
		\item We set $\Nil \coloneqq \LNil \cap \Coh_X^{F^r}$.
		\item Define the category of \emph{$r$-Frobenius quasi-crystals} by $\QCrys_X^{F^r} \coloneqq \QCoh_X^{F^r}/\LNil$, the category of \emph{ind-$r$-Frobenius quasi-crystals} to be $\IndCrys_X^{F^r} \coloneqq \IndCoh_X^{F^r}/\LNil$, and the category of \emph{$r$-Frobenius crystals} to be $\Crys_X^{F^r} \coloneqq \Coh_X^{F^r}/\Nil$ (as for Frobenius modules, we will usually omit the $r$ in discussions).
	\end{itemize}
\end{defn}
\begin{rem}
	\begin{itemize}
		\item For example, a Frobenius crystal is zero if and only if the underlying Frobenius module is nilpotent.
		\item The category $\LNil$ is a Serre subcategory by \cite[Proposition 3.3.4]{Bockle_Pink_Cohomological_Theory_of_crystals_over_function_fields}.
		\item If $\cM$, $\cN \in \QCoh_X^{F^r}$ are objects which become isomorphic in $\QCrys_X^{F^r}$, we write $\cM \sim_F \cN$. We use the same notation for complexes in the derived category.
	\end{itemize}
\end{rem}

\begin{lem}\label{all cats of F-mods are Grothendieck}
	Let $X$ be a Noetherian $\bF_p$-scheme, and let $\cC$ be one of the following categories:  \[\Mod(\cO_X), \QCoh_X, \Mod(\cO_X[F^r]), \QCoh_X^{F^r}, \IndCoh_X^{F^r}, \QCrys_X^{F^r}, \IndCrys_X^{F^r}.\]
	Then $\cC$ is a Grothendieck category. In particular, it has enough injectives.
\end{lem}
\begin{proof}
	If $\cC = \QCoh_X^{F^r}$, this is  \autoref{qcoh F-modules is a Groth cat}. If $\cC = \QCoh_X$, this follows from \autoref{qcoh A-modules is a Grothendieck category}. If $\cC = \Mod(\cO_X[F^r])$ (resp. $\Mod(\cO_X)$), then the collection $\{j_!\cO_U[F^r]\}_{j \colon U \inj X}$ (resp. $\{j_!\cO_U\}_{j \colon U \inj X}$) forms a set of generators, where we range over all open inclusions $j \colon U \inj X$. If $\cC = \IndCoh_X^{F^r}$, then coherent Frobenius modules (up to isomorphism) form a set of generators by definition. Finally, for $\cC \in \{\QCrys_X^{F^r}, \IndCrys_X^{F^r}\}$, this is a consequence of \cite[Theorem 3.4.5]{Bockle_Pink_Cohomological_Theory_of_crystals_over_function_fields}. 
	
	The fact that these categories have enough injectives now follows from \cite[{\href{https://stacks.math.columbia.edu/tag/079H}{Tag 079H}}]{Stacks_Project}.
\end{proof}

\begin{prop}\label{all derived pushforwards are the same for F-modules}
	Let $f \colon X \to Y$ be a proper morphism of Noetherian $\bF_p$-schemes. Then the diagram 
		\[ \begin{tikzcd}
		D^+(\IndCoh_X^{F^r}) \arrow[d] \arrow[rr, "Rf_*"] &  & D^+(\IndCoh_Y^{F^r}) \arrow[d] \\
		D^+(\cC_X) \arrow[rr, "Rf_*"]                     &  & D^+(\cC_Y)                 
	\end{tikzcd} \] commutes for any \[ \cC \in \left\{\Mod(\cO), \Mod(\cO[F^r]), \QCoh, \IndCrys^{F^r}\right\}.\]
\end{prop}

\begin{proof} 
	Let us show that the diagram
	\[ \begin{tikzcd}
		D^+(\IndCoh_X^{F^r}) \arrow[rr, "Rf_*"] \arrow[d] &  & D^+(\IndCoh_Y^{F^r}) \arrow[d] \\
		D^+(\QCoh_X) \arrow[rr, "Rf_*"]                 &  & D^+(\QCoh_Y)                
	\end{tikzcd} \] commutes. Note that the restriction of an injective in both $\IndCoh_X^{F^r}$ and $\QCoh_X$ to an open subset is still injective (see \cite[Proposition 5.2.7]{Bockle_Pink_Cohomological_Theory_of_crystals_over_function_fields} for $\IndCoh_X^{F^r}$, and apply the same proof for $\QCoh_X$). Thus, their respective higher pushforward commute with restrictions to opens, so we may assume that $Y$ is affine. In particular, $X$ is then semi-separated. By the same argument as in \cite[Proposition B.8]{Thomason_Trobaugh_Higher_Algebraic_K_Theory_of_schemes_and_of_derived_categories}, we deduce in this case that the functor $Rf_*$ in both $\QCoh_X^{F^r}$ and $\QCoh_X$ are given by \v{C}ech cohomology, so they agree. Thus, it is enough to show that in this case, the diagram \[ \begin{tikzcd}
		D^+(\IndCoh_X^{F^r}) \arrow[rr, "Rf_*"] \arrow[d] &  & D^+(\IndCoh_Y^{F^r}) \arrow[d] \\
		D^+(\QCoh_X^{F^r}) \arrow[rr, "Rf_*"]                 &  & D^+(\QCoh_Y^{F^r})                
	\end{tikzcd} \] commutes. This is an immediate consequence of the non-crystal version of \cite[Theorem 6.5.2]{Bockle_Pink_Cohomological_Theory_of_crystals_over_function_fields} (which their proof shows), and the fact that given $\cM \in D(\IndCoh_X^{F^r})$, we have $R\ind(\cM) = \cM$ (use \cite[Theorem 5.2.10.(d)]{Bockle_Pink_Cohomological_Theory_of_crystals_over_function_fields} and a hypercohomology spectral sequence argument).

	We automatically deduce the case of $\cC_X = \Mod(\cO_X)$ by \autoref{injective qcoh is injective O_X}. 
	
	Since injectives in $\Mod(\cO_X[F^r])$ are flasque (use the extension by zero functors), taking derived pushforwards in $\Mod(\cO_X)$ or in $\Mod(\cO_X[F^r])$ give the same object (as $\cO_X$-modules). Thus, we deduce the case $\cC_X = \Mod(\cO_X[F^r])$ from the case $\cC_X = \Mod(\cO_X)$.
	
	Finally, for $\cC = \IndCrys^{F^r}$, note that the composition \[ \begin{tikzcd}
		D^+(\IndCoh_X^{F^r}) \arrow[rr, "Rf_*"] &  & D^+(\IndCoh_Y^{F^r}) \arrow[rr] &  & D^+(\IndCrys_X^{F^r})
	\end{tikzcd} \] annihilates $D^+_{\LNil}(\IndCoh_X^{F^r})$. Indeed, by a hypercohomology spectral sequence argument, we have to show that if $\cM \in \LNil$, then $R^if_*(\cM)$ for all $i \geq 0$. Since $R^if_*$ commutes with filtered colimits (this follows from our work above, since the $Rf_*$ on $\cO_X$-modules does), we have to show that the functors $R^if_*$ preserve nilpotent Frobenius modules. This holds by functoriality. 

	Thus, we conclude by \autoref{derived_functors_and_localization} and essential surjectivity of $D^+(\IndCoh_X^{F^r}) \to D^+(\IndCrys_X^{F^r})$, see \cite[Theorem 2.6.2]{Bockle_Pink_Cohomological_Theory_of_crystals_over_function_fields}.
\end{proof}
\begin{rem}\label{rem_ok_if_semisep}
	Unfortunately, we do not know in general the result for $\QCoh_X^{F^r}$. However, the same proof as above shows this result when $Y$ is semi-separated. Indeed, we then obtain that higher direct images for both $\QCoh_X^{F^r}$ and $\QCoh_X$ are given by \v{C}ech cohomology, hence they agree.
\end{rem}

\begin{defn-constr}[\cite{Bockle_Pink_Cohomological_Theory_of_crystals_over_function_fields}, \cite{Bhatt_Lurie_RH_corr_pos_char}]\label{extending F-modules to the etale site}
	Let $\cM \in \QCoh_X^{F^r}$. We can extend both $\Mcal$ and $\tau_{\cM}$ to the étale site, and obtain a morphism $\tau_{{\cM}_{\et}} \colon \Mcal_{\et} \ra F^r_*\Mcal_{\et}$ (recall that we can pullback Frobenius modules, see \autoref{operations_on_F_modules}).
	
	Define $\Sol(\Mcal) \coloneqq \ker(\tau_{{\cM}_{\et}} - 1)$, taken in the abelian category $\Sh(X_{\et}, \Ff_q)$ of \'etale $\Ff_q$-sheaves on $X$. This construction defines a functor $\Sol \colon \QCoh_X^{F^r} \ra \Sh(X_{\et}, \Ff_q)$.
\end{defn-constr}

\begin{theorem}\label{main thm Bhatt Lurie}
	Let $X$ be a Noetherian $\bF_q$-scheme. The functor $\Sol \colon \QCoh_X^{F^r} \to \Sh(X_{\et}, \bF_q)$ vanishes on elements in $\LNil$, and induces equivalences of categories 
	\[ \begin{tikzcd}[row sep=small]
		\IndCrys_X^{F^r} \arrow[rr, "\cong"] &  & {\Sh_{\et}(X, \bF_q)}, \mbox{ and }  \\
		\Crys_X^{F^r} \arrow[rr, "\cong"]    &  & {\Sh^c_{\et}(X, \bF_q)},
	\end{tikzcd} \] where $\Sh^c_{\et}(X, \bF_q)$ denotes the category of constructible étale $\bF_q$-sheaves.

	This functor preserves tensor products, pullbacks and proper pushforwards.
\end{theorem}
\begin{proof}
	See \cite[Theorem 10.3.6 and Proposition 10.2.1]{Bockle_Pink_Cohomological_Theory_of_crystals_over_function_fields}. Although they assume that their schemes are separated and of finite type over $\bF_q$, this does not appear in their proof. The reader can also look at \cite[Theorem 10.2.7]{Bhatt_Lurie_RH_corr_pos_char} in the general case.
\end{proof}

\begin{prop}\label{D^b(Coh) = D^b_{coh}(QCoh) for F-modules}
	Let $X$ be a Noetherian $\Ff_p$-scheme. Then for $* \in \{-, b\}$, the natural functors 
	\[ \begin{tikzcd}[row sep=small]
			D^*(\Coh_X^{F^r}) \arrow[rr]  &  & D^*_{\coh}(\QCoh_X^{F^r}),  \mbox{ and } \\
			D^*(\Crys_X^{F^r}) \arrow[rr] &  & D^*_{\crys}(\QCrys_X^{F^r}) 
		\end{tikzcd}  \] 
	are equivalences of categories, where by $D^*_{\crys}(\QCrys_X^{F^r})$ we mean complexes whose cohomology sheaves are in the essential image of $\Crys_X^{F^r} \to \QCrys_X^{F^r}$. If $X$ is in addition semi-separated or regular, then also the functor \[ \begin{tikzcd}	D^b(\Coh_{X}^{F^r}) \arrow[rr] & & D^b_{\coh}(\cO_X[F^r])
	\end{tikzcd} \] is an equivalence.
\end{prop}
\begin{proof}
	The first two statements are part of \cite[Theorem 5.3.1]{Bockle_Pink_Cohomological_Theory_of_crystals_over_function_fields} (their proof only uses Noetherianness of $X$). Let us show the last one. By our work above, we only have to show that $D(\QCoh_X^{F^r}) \to D_{\qcoh}(\cO_X[F^r])$ is an equivalence of categories. By \cite[{\href{https://stacks.math.columbia.edu/tag/077P}{Tag 077P}}]{Stacks_Project}, there exists a right adjoint $Q$ of the inclusion $\QCoh_X \ra \Mod(\Ocal_X)$. Furthermore, we have that $Q \circ F^r_* = F^r_* \circ Q$ (this equality hold for their left adjoints). Hence, if $\Mcal$ is a Frobenius module, then $Q(\Mcal)$ is naturally a (quasi-coherent) Frobenius module. This defines a functor $Q \colon \Mod(\Ocal_X[F^r]) \ra \QCoh_X^{F^r}$, right adjoint of the inclusion $\QCoh_X^{F^r} \ra \Mod(\Ocal_X[F^r])$. 
	
	By definition, the square \[ \begin{tikzcd}
		{\Mod(\Ocal_X[F^r])} \arrow[rr, "Q"] \arrow[d] &  & \QCoh_X^{F^r} \arrow[d] \\
		\Mod(\Ocal_X) \arrow[rr, "Q"]                  &  & \QCoh_X          
	\end{tikzcd} \] commutes. Assume for now that we know that its derived version	\[ \begin{tikzcd}
		{D^+(\Ocal_X[F^r])} \arrow[rr, "RQ"] \arrow[d] &  & D^+(\QCoh_X^{F^r}) \arrow[d] \\
		D^+(\Ocal_X) \arrow[rr, "RQ"]                  &  & D^+(\QCoh_X)          
	\end{tikzcd} \] also commutes. Then to show that $R^iQ(\Mcal) = 0$ for all $\Mcal \in \QCoh_X^{F^r}$ (which concludes the proof by \autoref{D_B(A) = D(B) if B objects are B-acyclic}), it suffices to show that $R^iQ(\Mcal) = 0$ for all $\Mcal \in \QCoh_X$. Given that injectives in $\QCoh_X$ are also injectives in $\Mod(\Ocal_X)$, in order to compute $R^iQ(\Mcal)$, we may resolve $\Mcal$ by injectives quasi-coherent modules. Since $Q$ is the identity on quasi-coherent modules (it is the right adjoint of a fully faithful functor), we deduce that also $RQ$ is the identity on quasi-coherent modules. 

	Thus, we are left to show that the square above commutes. In other words, we have to show that an injective object in $\Mod(\cO_X[F^r])$ is $Q$-acyclic. If $X$ is regular, this is immediate, since an injective in $\Mod(\cO_X[F^r])$ is also injective in $\Mod(\cO_X)$ by \cite[1.7.2]{Emerton_Kisin_Riemann-Hilbert_correspondence} (see also \autoref{remark injectives F-modules are also injectives O_X for X regular}).
	
	Let us now assume that $X$ is semi-separated, pick an injective object $\cI$ in $\Mod(\cO_X[F^r])$, and let $\{U_i\}$ be a semi-separating affine cover of $X$ (i.e. each inclusion $j_{U_i} \colon U_i \to X$ is affine). Then by the sheaf condition, we have an injection $\cI \inj \bigoplus_i j_{U_i, *}\cI|_{U_i}$ of Frobenius modules, so it splits by injectivity of $\cI$. It is then enough to show that $\bigoplus_i j_{U_i, *}\cI|_{U_i}$ is $Q$-acyclic, and hence that each $j_{U_i, *}\cI|_{U_i}$ is $Q$-acyclic. Since $\cI$ is flasque, so is $\cI_{U_i}$, which implies that this latter sheaf is $j_{U_i, *}$-acyclic. Thus, \[ RQ(j_{U_i, *}\cI|_{U_i}) = RQ(Rj_{U_i, *}\cI|_{U_i}) \expl{\cong}{\cite[{\href{https://stacks.math.columbia.edu/tag/08D8}{Tag 08D8}}]{Stacks_Project}} Rj_{U_i, *}RQ(\cI|_{U_i}). \]
	Given $j \geq 0$, we know by \cite[{\href{https://stacks.math.columbia.edu/tag/08D9}{Tag 08D9}}]{Stacks_Project} that $R^jQ(\cI_{U_i})$ is the quasi--coherent module corresponding to $H^j(U_i, \cI|_{U_i})$. By flasqueness of $\cI|_{U_i}$, this vanishes for $j > 0$, so $RQ(\cI|_{U_i}) = Q(\cI|_{U_i})$. We are thus left to show that $Q(\cI|_{U_i})$ is $j_{U_i, *}$-acyclic. This is immediate, since $j_{U_i}$ is affine.
\end{proof}

\begin{defn}
	Let $X$ be a Noetherian $\bF_p$-scheme, and let $\cM \in D^b(\QCrys_X^{F^r})$. Define \[ \Supp_{\crys}(\cM) \coloneqq \set{x \in X}{\cM_x \not\sim_F 0}. \]
\end{defn}

\begin{lemma}\label{support F-crystals}
	Let $X$ be a Noetherian $\bF_q$-scheme, and let $\cM \in D^b(\Crys_X^{F^r})$.
	\begin{enumerate}
		\item\label{support_Sol} We have an equality \[ \set{x \in X}{\cM \otimes k(x) \not\sim_F 0} = \Supp(\Sol(\cM)), \] and this subset is constructible. In particular, $\cM \sim_F 0$ if and only if $\cM \otimes k(x) \sim_F 0$ for all $x \in X$.
		\item\label{supp_crys} We have that \[ \Supp_{\crys}(\cM) = \overline{\Supp(\Sol(\cM))}, \] so in particular it is closed.
		\item\label{restricts_and_comes_back} Let $Z \coloneqq \Supp_{\crys}(\cM)$. Then $\cM \sim_F i_{Z, *}i_Z^*\cM$, where $i_Z \colon Z \to X$ denotes the inclusion (the functor $i_Z^*$ is exact by \autoref{main thm Bhatt Lurie} and the analogous statement for étale $\bF_q$-sheaves).
	\end{enumerate}
\end{lemma}
\begin{proof}
	\begin{enumerate}
		\item Note that $\cM \otimes k(x) \sim_F 0$ if and only if $\cM \otimes k(x)^{\sep} \sim_F 0$, so the first result follows since by \autoref{main thm Bhatt Lurie}, we have $\Sol(\cM \otimes k(x)^{\sep}) = \Sol(\cM)_{\overline{x}}$ (the Riemann-Hilbert correspondence commutes with pullbacks). The constructibility result then follows from \cite[{\href{https://stacks.math.columbia.edu/tag/09YS}{Tag 09YS}}]{Stacks_Project}, and the statement after ``In particular'' is now immediate from \autoref{main thm Bhatt Lurie}.
		\item By \autoref{support_Sol}, we have that $\Supp(\Sol(\cM)) \inc \Supp_{\crys}(\cM)$. Let $Z$ denote the closure of $\Supp(\Sol(\cM))$. Then $\Sol(\cM) \cong i_{Z, *}i_Z^*\Sol(\cM)$, so by \autoref{main thm Bhatt Lurie}, we have $\cM \sim_F i_{Z, *}i_Z^*\cM$. Thus, $\Supp_{\crys}(\cM) \inc Z$. We are left to show that $\Supp_{\crys}(\cM)$ is closed. We may then assume that $\cM$ is a sheaf. Let $\cM' \in \Coh_X^{F^r}$ be such that $\cM' \sim_F \cM$ and $\tau_{\cM'}$ is injective (e.g. take $\cM'$ to be the image of $\cM \to \cM^{1/p^{\infty}}$). Then we have \[ \Supp_{\crys}(\cM) = \Supp_{\crys}(\cM') = \Supp(\cM'), \] so we are done.
		\item This follows from the proof of \autoref{supp_crys}.
	\end{enumerate}
\end{proof}

\section{Cartier modules}
Our goal is to show some structural results on Cartier modules, which will be crucial for the duality to come. The most important result is an analogue of \autoref{D^b(Coh) = D^b_{coh}(QCoh) for F-modules} for Cartier modules (see \autoref{D^b(Coh) = D^b_{coh}(QCoh) for Cartier modules}). The two main ingredients to prove this are Gabber's result saying that a Noetherian F-finite ring is a quotient of a regular ring admitting a $p$-basis (see \cite[Remark 13.6]{Gabber_notes_on_some_t_structures}), and the fact that injective objects in the category of ind-coherent Cartier modules are also injective as $\cO_X$-modules (\autoref{inj qalg is inj O_X}). This latter fact heavily relies on finiteness properties of Cartier modules, together with the classification of injective quasi-coherent modules (\cite{Matlis_Injective_modules_over_Noetherian_rings}, \cite{Hartshorne_Residues_and_Duality}). \\

As before, fix $r \geq 1$ and set $q = p^r$.

\subsection{The category of Cartier modules}
\begin{defn}\label{def Cart module}
	Let $X$ be an $\bF_p$-scheme, and let $\Mcal$ be an $\Ocal_X$-module. An \emph{$r$-Cartier module} (or simply \emph{Cartier module}) structure on $\Mcal$ is the datum of a morphism $\kappa_{\cM} \colon F^r_*\Mcal \ra \Mcal$. A morphism of Cartier modules is a morphism of underlying $\cO_X$--modules commuting with the respective Cartier module structures (i.e. satisfying the analogue of the diagram in \autoref{def F-module}). We denote the category of Cartier modules by $\Mod(\Ocal_X[F^r]^{op})$ or $\Mod(\Ocal_X[C^r])$. 
	
	The full subcategory of quasi-coherent Cartier modules is denoted by $\QCoh_X^{C^r}$. 
\end{defn}

\begin{rem}\label{rem def Cartier module}
	\begin{enumerate}
		\item Unless stated otherwise, the structural morphism of a Cartier module $\cM$ will always be denoted $\kappa_{\cM}$.
		\item The notation $\Mod(\Ocal_X[F^r]^{op})$ makes more sense because of the next observation. However it is quite cumbersome, so we will usually write it $\Mod(\Ocal_X[C^r])$.
		\item We can take pushforward of Cartier modules, and as we will see in \autoref{eq def of Cartier mod} we can also take the right adjoint of the pushforward for finite morphisms (on quasi-coherent Cartier modules). 
		\item\label{tensor product of Cartier mod and unit F-mod} Let $\cM \in \QCoh_X^{C^r}$ and let $\cN$ be a unit Frobenius module (see \autoref{def F-module}). Then $\cM \otimes \cN$ can be given the structure of a Cartier module as follows: \[ \begin{tikzcd}
			F^r_*(\cM \otimes \cN) \arrow[rr, "F^r_*\left(id \otimes (\tau_{\cN}^*)^{-1}\right)"] && {F^r_*(\cM \otimes F^{r, *}\cN)} \arrow[rr, "\cong"] && F^r_*\cM \otimes \cN \arrow[rr, "\kappa_{\cM} \otimes id"] && \cM \otimes \cN.
		\end{tikzcd} \]
	\end{enumerate}
\end{rem}
\begin{obs*}
	Let $X$ be an $\bF_p$-scheme. The category of (quasi-coherent) Cartier modules is equivalent to the category of (quasi-coherent) sheaves of right $\Ocal_X[F^r]$-modules (or equivalently of left $\Ocal_X[F^r]^{op}$-modules).
\end{obs*}
\begin{cor}\label{Cartier modules is a Groth cat}
	The category $\QCoh_X^{C^r}$ is a Grothendieck category. In particular it has enough injectives.
\end{cor}

\begin{proof}
	This follows from  \autoref{qcoh A-modules is a Grothendieck category}, once we know that $\Ocal_X[F^r]^{op}$ is quasi-coherent. This is immediate, since as $\cO_X$-modules, \[ \Ocal_X[F^r]^{op} \cong \bigoplus_{n \geq 0} F^{rn}_*\cO_X. \]
\end{proof}

\begin{lem}\label{Mod(O_X[C]) has enough injectives}
	The category $\Mod(\Ocal_X[C^r])$ is a Grothendieck category.
\end{lem}
\begin{proof}
	Same proof as in \autoref{all cats of F-mods are Grothendieck}.
\end{proof}

Now we want to compare injective quasi-coherent Cartier modules and injective quasi-coherent $\Ocal_X$-modules.
\begin{prop}\label{inj qcoh Cartier is inj O_X}
	Let $X$ be an $\Ff_p$-scheme, and let $\Ical$ be an injective object in $\QCoh_X^{C^r}$ (resp. $\Mod(\Ocal_X[C^r])$). Then $\Ical$ is also injective in $\QCoh_X$ (resp. $\Mod(\Ocal_X)$). 
	
	In particular, if $X$ is locally Noetherian and $\Ical$ is injective in $\QCoh_X^{C^r}$, then it is also injective in $\Mod(\Ocal_X)$.
\end{prop}
\begin{proof}
	By \autoref{injective qcoh is injective O_X}, it is enough to show the first statement. We only show the quasi-coherent case, as the non-necessarily quasi-coherent case is identical. The main point is that any injective $\cO_X$-module containing a Cartier module also acquires a Cartier structure.
	
	Let $\Mcal \inj \Ncal$ be an inclusion of quasi-coherent modules, and $f \colon \Mcal \ra \Ical$ a morphism of $\Ocal_X$-modules. Let $\Mcal'$ be the smallest quasi-coherent Cartier submodule of $\Ical$ containing $\im(f)$. 
	
	Let $\Jcal$ be an injective in $\QCoh_X$ containing $\Mcal'$. We have a diagram \[ \begin{tikzcd}
		F^r_*\Mcal' \arrow[r, "\kappa_{\cM'}"] \arrow[d, hook] & \Mcal' \arrow[d, hook] \\
		F^r_*\Jcal                             & {\Jcal.}               
	\end{tikzcd} \]
	By injectivity of $\Jcal$, there exists a Cartier structure on $\Jcal$ such that the injection $\Mcal' \inj \Jcal$ becomes an injection of Cartier modules. We have a diagram \[ \begin{tikzcd}
		\Mcal' \arrow[d] \arrow[r, hook] & \Jcal \\
		\Ical                       &      
	\end{tikzcd} \] of Cartier modules, so given that $\Jcal$ is a quasi-coherent Cartier module and $\Ical$ is injective in $\QCoh_X^{C^r}$, there exists a morphism $\mu \colon \Jcal \ra \Ical$ completing the above diagram. Thus, we obtain
	\[ \begin{tikzcd}
		\cN                            &                          &                        &     \\
		\cM  \arrow[u, hook] \arrow[r] & \cM' \arrow[r, "\inc"'] & \cJ \arrow[r, "\mu"'] & \cI.
	\end{tikzcd} \]
	Since $\Jcal$ is injective in $\QCoh_X$, we are done.
\end{proof}

\begin{rem}\label{remark injectives F-modules are also injectives O_X for X regular}
	One could try to mimic the above proof in the case of Frobenius modules, but there is a problem: if $\Jcal$ is injective in $\QCoh_X$, is $F_*\Jcal$ also injective in $\QCoh_X$? 
	
	This holds when $X$ is regular. Indeed, $F$ is then flat by Kunz' theorem (see \cite[Theorem 2.1]{Kunz_Characterizations_of_regular_local_rings_of_char_p}), so $F_*$ is the right adjoint of an exact functor, and hence preserves injectives.
\end{rem}

\begin{prop}\label{D(QCoh) = D_qcoh(O_X[C])}
	Let $X$ be a Noetherian $\bF_p$-scheme. Then for $* \in \{+, b\}$, the natural functor $D^*(\QCoh_X^{C^r}) \ra D^*_{\qcoh}(\Ocal_X[C^r])$ is an equivalence of categories.
\end{prop}
\begin{proof}
	Since an injective object in $\Mod(\cO_X[C^r])$ is also injective in $\Mod(\cO_X)$ (\autoref{inj qcoh Cartier is inj O_X}), we can apply the exact same proof as in the last point of \autoref{D^b(Coh) = D^b_{coh}(QCoh) for F-modules} in the regular case.
\end{proof}
\begin{rem}
	The above result can be generalized to quasi-compact and semi-separated schemes (mimic the proof of \cite[{\href{https://stacks.math.columbia.edu/tag/09T6}{Tag 09T6}}]{Stacks_Project}). 
\end{rem}

\subsection{Ind-coherent Cartier modules}

\begin{defn} 
	Let $X$ be a Noetherian $\bF_p$-scheme. A quasi-coherent Cartier module $\Mcal$ is said to be \emph{ind-coherent} if it is a filtered colimit of coherent Cartier modules. The full subcategory of ind-coherent Cartier modules is denoted by $\IndCoh_X^{C^r}$.
\end{defn}

\begin{lem}\label{basic props of ind-coherent objects}
	Let $X$ be a Noetherian $\bF_p$-scheme.
	\begin{itemize}
		\item Ind-coherence is stable under taking quotients, subobjects and colimits.
		\item A Cartier module $\Mcal$ is ind-coherent if and only if it is locally ind-coherent (i.e. there exists an open cover $\{U_i\}_{i \in I}$ of $X$ such that $\Mcal|_{U_i}$ is ind-coherent over $U_i$ for all $i$).
	\end{itemize}
\end{lem}
\begin{proof}
	\begin{itemize}
		\item This is immediate from the analogue of \autoref{rem def ind-coherence} for Cartier modules.
		\item If $\Mcal$ is ind-coherent, then it is automatically locally ind-coherent. Conversely, since $X$ is Noetherian, we know by \cite[{\href{https://stacks.math.columbia.edu/tag/01PG}{Tag 01PG}}]{Stacks_Project} that $\Mcal$ is the union of its coherent $\Ocal_X$-submodules. Therefore, it is enough to show that for all coherent $\Ncal \inc \Mcal$, there exists a coherent Cartier module $\Pcal$ containing $\Ncal$. 
		
		 Define $\Pcal'$ to be the image of the following composition 
		\begin{gather}
			\bigoplus_{l = 1}^\infty F^{rl}_*\Ncal \xrightarrow{\inc} \bigoplus_{l = 1}^\infty F^{rl}_*\Mcal \ra \Mcal,
		\end{gather} where the last arrow is induced by the Cartier structure on $\cM$ (applied the adequate number of times). Then $\Pcal \coloneqq \Pcal' + \Ncal$ is the smallest quasi-coherent Cartier module containing $\Ncal$, and the same holds over any open subset. 
	
		Now, let $i \in I$, and let $\Mcal|_{U_i}$ be the union of the coherent Cartier modules $\{\Scal_j\}_j$. Then \[ \Ncal|_{U_i} = \bigcup_j(\Scal_j \cap \Ncal|_{U_i}). \] Since $\Ncal|_{U_i}$ is coherent, it is a Noetherian object in the category of quasi-coherent sheaves. Hence, $\Ncal|_{U_i} \inc \Scal_j$ for some $j$, so \[ \Ncal|_{U_i} \inc \Pcal|_{U_i} \inc \Scal_j. \] Thus, $\Pcal|_{U_i}$ is a subobject of a coherent module, so it is coherent.
	\end{itemize}
\end{proof}

\begin{cor}\label{indcoh Cartier is a Groth cat}
	Let $X$ be a Noetherian $\Ff_p$-scheme. Then the category $\IndCoh_X^{C^r}$ is a Grothendieck category.
\end{cor}
\begin{proof}
	By \autoref{basic props of ind-coherent objects}, this category is closed under kernels, cokernels and colimits, so it is an abelian category satisfying Grothendieck's axiom AB5. By definition, coherent objects form a set of generators so it is a Grothendieck category.
\end{proof}

\begin{notation}
	Recall that a Cartier module $M$ on an $\bF_p$-algebra $R$ has a right action of $R[F^r]$. For $m \in M$ and $s \in R[F^r]$, we denote this action by $m \cdot s$.
\end{notation}

\begin{lem}\label{locally ind-coh iff pseudo qalg}
	Let $R$ be a Noetherian $\bF_p$-algebra, let $M$ be a Cartier module on $R$, and let $m \in M$. For all $j \geq 1$, let $\{F^{rj}_*\lambda_{i, j}\}_{1 \leq i \leq n_j}$ be a set of generators of $F^{rj}_*R$. Then $m$ is included in a finitely generated Cartier submodule of $M$ if and only if there exists $j \geq 0$ such that for all $1 \leq i \leq n_j$, there exists $s_i \in R[F^r]$ of degree $rj$ with leading term $\lambda_{i, j}$ satisfying $m \cdot s_i = 0$.
\end{lem}
\begin{proof}
	In order to make the proof nicer to read, we will omit the $F_*$ notations.
	
	The Cartier submodule generated by $m$ is exactly the submodule generated by the elements \[ \kappa_M^j(\lambda_{i, j}m) \] with $j \geq 1$, and $1 \leq i \leq n_j$. By Noetherianness of $R$, we deduce the ``left-to-right'' statement. 
	
	Conversely, assume the existence of $\{s_i\}_{1 \leq i \leq n_j}$ as in the statement, and write $s_i = \lambda_{i, j}F^{rj} + \sum_{k < j}\alpha_{i, k}F^{rk}$. Then for any $l > j$ and $a \in R$ (which we write $a = \sum_{1 \leq i \leq n_j} a_i^{p^{rj}}\lambda_{i, j}$), 
	\begin{align*}
		 \kappa_M^l(am) & = \kappa_M^{l - j}\left(\kappa_M^j(am)\right) \\
		 & = \sum_{1 \leq i \leq n_j}\kappa_M^{l - j}\left(a_i\kappa_M^j(\lambda_{i, j}m)\right) \\
		 & = - \sum_{1 \leq i \leq n_j}\kappa_M^{l - j}\left(a_i\sum_{k < j}\kappa_M^k(\alpha_{i, k}m)\right) \\
		 & \eqqcolon m\cdot t,
	\end{align*}	 
	 where $\deg(t) < rl$. Hence, we have proven by induction that for any $t_1 \in R[F^r]$, there exists $t_2 \in R[F^r]$ with $\deg(t_2) < rj$ such that $m \cdot t_1 = m \cdot t_2$. Thus, we conclude the proof.
\end{proof}

\begin{rem}\label{nilp is ind-coh}
	As an immediate corollary, a nilpotent Cartier module (i.e. a Cartier module $\Mcal$ such that $\kappa_{\cM}^n = 0$ for some $n > 0$) is ind-coherent.
\end{rem}

\begin{cor}[{\cite[Proposition 2.0.8]{Blickle_Bockle_Cartier_crystals}\label{ind-coh closed under extensions}}] Let $X$ be a Noetherian $F$-finite $\bF_p$-scheme. Then the category of ind-coherent modules is closed under extensions.
\end{cor}
\begin{proof}
	By the local nature of ind-coherence (see \autoref{basic props of ind-coherent objects}), we may assume that $X = \Spec R$ is affine. For all $j \geq 1$, fix generators $F^{rj}_*\lambda_{1, j}, \dots, F^{rj}_*\lambda_{n_j, j}$ of $F^{rj}_*R$. Let \[ 0 \ra M \ra N \ra P \ra 0\] be a short exact sequence of Cartier modules with $M$ and $P$ ind-coherent, and let $n \in N$. By \autoref{locally ind-coh iff pseudo qalg}, there exists $j \geq 0$, $s_1, \dots, s_{n_{j}} \in R[F^r]$ such that each $s_i$ has degree $rj$, leading term $\lambda_{i, j}$ and annihilates the image of $n$ in $P$. By exactness, $n\cdot s_i \in M$ for all $i$.
	
	Since $M$ is ind-coherent, for all $i$, there exist $j_i'$ and $s_{1, i}, \dots, s_{j_i', i}$ such that each $s_{i', i}$ has degree $rj_i'$, leading coefficient $\lambda_{i', j_i'}$ and $n \cdot s_is_{i', i} = 0$. 
	
	By $F$-finiteness of $R$, we may assume that all $j'_i$ are equal (see the proof of \autoref{locally ind-coh iff pseudo qalg}). Denote this number by $j'$. We then have that for all $i$, $i'$, the element $s_is_{i', i} \in R[F^r]$ has degree $r(j + j')$ and has leading term $\lambda_{i, j}\lambda_{i', j'}^{p^{rj}}$. Since the elements $\lambda_{i, j}\lambda_{i', j'}^{p^{rj}}$ generate $F^{r(j + j')}_*R$, we conclude by \autoref{locally ind-coh iff pseudo qalg}.
\end{proof}

\vspace{0.5 cm}
\noindent\textit{\textbf{Comparison between injective ind-coherent Cartier modules and injective} $\Ocal_X$\textbf{-modules}.} \\

\noindent Our goal is to show the following result.

\begin{prop}\label{inj qalg is inj O_X}
	Let $X$ be an $F$-finite Noetherian $\Ff_p$-scheme. An injective object of $\IndCoh_X^{C^r}$ is also injective as an object of $\Mod(\Ocal_X)$.
\end{prop}

The proof of this statement requires a few preliminary results.

\begin{notation}\label{notation inj hull}
	Given a Noetherian ring $R$ and a $R$-module $M$, $E(M)$ denotes an injective hull of $M$. Furthermore, given a Noetherian $F$-finite scheme $X$ and $x \in X$, set \[ \Jcal(x) \coloneqq (i_x^X)_*\ttilde{E(k(x))}, \] where $i_x^X$ denotes the natural map $\Spec \Ocal_{X, x} \to X$.
\end{notation}

\begin{lem}\label{E(Rp/pRp)|R = E(R/p)}
	Let $R$ be a Noetherian ring, and let $\pfr \nor R$ be a prime ideal. Then \[ E(R_\pfr/\pfr R_\pfr)|_R \cong E(R/\pfr), \] where the first injective hull is taken over $R_\pfr$.
\end{lem}
\begin{proof}
	By uniqueness of injective hulls, it is enough to show that $E(R_\pfr/\pfr R_\pfr)|_R$ is an injective hull of $R/\pfr$ over $R$. First, the forgetful functor $(\cdot)|_R \colon \Mod(R_\pfr) \ra \Mod(R)$ is the right adjoint of the pullback via $R \to R_{\pfr}$ (which is exact). Therefore it preserves injectives, whence $E(R_\pfr/\pfr R_\pfr)|_R$ is injective. 
	
	We have a chain of inclusions \[ R/\pfr \inc R_\pfr/\pfr R_\pfr  \inc E(R_\pfr/\pfr R_\pfr)|_R. \] The second inclusion is an essential extension over $R_\pfr$, so also over $R$. and the first one is also an essential extension (we have 
	$R_\pfr/\pfr R_\pfr \cong \Frac(R/\pfr)$). Thus, $R/\pfr \inc E(R_\pfr/\pfr R_\pfr)|_R$ is an essential extension as well, which proves the lemma. 
\end{proof}

\begin{lem}\label{there is no non-zero map between too different inj hulls}
	Let $X$ be a Noetherian scheme, and suppose that there exists a non-zero morphism $F^r_*\Jcal(x) \to \Jcal(y)$, then $y \in \overline{\{x\}}$. 
\end{lem}
\begin{proof} We decompose the argument in two steps. \\
	\ptofpf{Step 1:} \esp For all $x \in X$, $\Supp(\Jcal(x)) = \overline{\{x\}}$.
	\begin{itemize}
		\item Since $\cJ(x)$ is quasi-coherent, its support is stable under specialization. Thus, $\overline{\{x\}} \inc \Supp(\Jcal(x))$. 
		\item Let $V$ be an open subset such that $V \cap \overline{\{x\}} = \emptyset$. Then by definition, \[ \Gamma(V, \Jcal(x)) = \Gamma\left(W, \: \ttilde{E(k(x))}\right) \] where $W \coloneqq (i_x^X)^{-1}(V) \inc \Spec\Ocal_{X, x} \setminus \mfr_{X, x}$. The latter group is trivial because for all $m \in E(k(x))$ and $r \in \mfr_{X, x}$, $r^Nm = 0$ for $N \gg 0$ (see \cite[Theorem 3.4]{Matlis_Injective_modules_over_Noetherian_rings}).
	\end{itemize}		
	\ptofpf{Step 2:} \esp Conclude the argument. 
	
	Let $\mu \colon F^r_*\Jcal(x) \ra \Jcal(y)$ be a non-zero morphism. By the previous step, there exists $z \in \overline{\{x\}} \cap \overline{\{y\}}$ such that $\mu_z \neq 0$. Let $U = \Spec R$ be an open affine containing $z$. Since open subsets are closed under generization, we also obtain that $x, y \in U$. Let $x$ (resp. $y$) correspond to $\pfr \nor R$ (resp. $\qfr \nor R$). By \autoref{E(Rp/pRp)|R = E(R/p)}, there exists a non-trivial morphism \[ \mu \colon F^r_*E(R/\pfr) \ra E(R/\qfr). \] Suppose by contradiction that $\pfr \not\inc \qfr$, and let $\lambda \in \pfr \setminus \qfr$. Since $\mu$ is non-zero, there exists $F^r_*m \in F^r_*E(R/\pfr)$ such that $\mu(F^r_*m) \neq 0$. By \cite[Theorem 3.4]{Matlis_Injective_modules_over_Noetherian_rings}, there exists $N \geq 1$ such that $\pfr^Nm = 0$, so in particular $\lambda^Nm = 0$. On the other hand, by \cite[Lemma 3.2]{Matlis_Injective_modules_over_Noetherian_rings}, multiplying by $\lambda^N$ in $E(R/\qfr)$ is an isomorphism, so we have \[ 0 \neq \lambda^N\mu(F^r_*m) = \mu(F^r_*{\lambda}^{Np^r}m) = 0. \] This is a contradiction. 
\end{proof}

\begin{lem}\label{direct sum of residue fields is ind-coh}
	Let $R$ be an $F$-finite Noetherian integral domain over $\bF_p$, and let $L = \Frac(R)$. Then for any $n \geq 1$ and any Cartier module structure $\kappa_{L^{\oplus n}}$ on $L^{\oplus n}$, the Cartier module $(L^{\oplus n}, \kappa_{L^{\oplus n}})$ is ind-coherent.
\end{lem}
\begin{proof}
	Let $F^r_*a_1, \dots, F^r_*a_m$ be generators of $F^r_*R$ as an $R$-module, and let $e_j$ denote the standard $j$'th coordinate vector. Write \[\kappa_{L^{\oplus n}}\left(F^r_*(a_ie_j)\right) = \left(\frac{s_{i, j, 1}}{b_{i, j, 1}}, \dots, \frac{s_{i, j, n}}{b_{i, j, n}}\right), \] and set $b \coloneqq \prod_{i, j, l}b_{i, j, l}$. We then have \[ \kappa_{L^{\oplus n}}(F^r_*R^{\oplus n}) \inc \frac{R^{\oplus n}}{b}, \] where for any $0 \neq \lambda \in R$, we set \[ \frac{R^{\oplus n}}{\lambda} \coloneqq \biggset{(x_1, \dots, x_n) \in L^{\oplus n}}{\forall i \esp \exists a_i \in R, \: x_i = \frac{a_i}{\lambda}}.\] Therefore, \[ \kappa_{L^{\oplus n}}\left(F^r_*\frac{R^{\oplus n}}{b^q}\right) = \frac{\kappa_{L^{\oplus n}}(F^r_*R^{\oplus n})}{b} \inc \frac{R^{\oplus n}}{b^2} \inc \frac{R^{\oplus n}}{b^q}, \] so $M \coloneqq \frac{R^{\oplus n}}{b^q}$ is a coherent Cartier submodule of $L^{\oplus n}$. 
	
	Note that for all $\lambda \neq 0$, $\frac{M}{\lambda}$ is a finitely generated submodule of $L^{\oplus n}$. Furthermore it is stable under $\kappa_{L^{\oplus n}}$, since we have \[ \kappa_{L^{\oplus n}}\left(F^r_*\frac{M}{\lambda}\right) = \kappa_{L^{\oplus n}}\left(F^r_*\frac{\lambda^{q - 1}M}{\lambda^q}\right) \inc \frac{\kappa_{L^{\oplus n}}(F^r_*M)}{\lambda} \inc \frac{M}{\lambda}  \] Thus, $\frac{M}{\lambda}$ is a coherent Cartier submodule of $L^{\oplus n}$. The fact that \[ L^{\oplus n} = \bigcup_{\lambda \neq 0} \frac{M}{\lambda} \] concludes the proof.
\end{proof}

\begin{lemma}\label{inj hull of coh is indcoh}
	Let $X$ be an $F$-finite Noetherian $\bF_p$-scheme, let $\Mcal \in \Coh_X$ and let $\Ical$ be an injective hull of $\Mcal$ in $\QCoh_X$. For any Cartier structure $\kappa_{\cI}$ on $\Ical$, $(\Ical, \kappa_{\cI})$ is ind-coherent.
\end{lemma}
\begin{proof}
	By the proof of \cite[Theorem 7.18]{Hartshorne_Residues_and_Duality}, we can write \[ \Ical = \bigoplus_{i \in I}\Jcal(x_i) \] for some points $x_i \in X$. Since $\Mcal$ is a Noetherian object and $\Ical$ is an injective hull, there are only finitely many summands. Let us write $\Ical = \bigoplus_{i = 1}^n\Jcal(x_i)$. 
	
	Let $\kappa_{\cI}$ denote any Cartier structure on $\cI$. Since the direct sum is finite, we can see it as a collection of morphisms \[ F^r_*\Jcal(x_i) \xrightarrow{\kappa_{i, j}} \Jcal(x_j). \]
	By \autoref{there is no non-zero map between too different inj hulls}, we obtain that if $x_1, \dots, x_l$ are maximal under specialization (among the $x_i$'s), then each $\Jcal(x_i)$ is a Cartier submodule of $\Ical$. More generally, we obtain a filtration of Cartier submodules \[ 0 = \Jcal_0 \inc \Jcal_1 \inc \dots \inc \Jcal_s = \Ical, \] where for all $l$, there exists $y_l \in \{x_1, \dots, x_n\}$, $j_l \geq 1$ and a Cartier structure on $\Jcal(y_l)^{\oplus j_i}$ such that \[\Jcal_{l + 1}/\Jcal_l \cong \Jcal(y_l)^{\oplus j_l}. \]
	Since an extension of ind-coherent Cartier modules is ind-coherent (see \autoref{ind-coh closed under extensions}), we reduced our problem to showing that for all $x \in X$ and $n \geq 1$, the quasi-coherent sheaf $\Jcal(x)^{\oplus n}$, together with any Cartier structure, is ind-coherent. 
	
	By \autoref{basic props of ind-coherent objects}, we may do this computation locally, i.e. we may assume that $X = \Spec R$ with $R$ Noetherian and $F$-finite. Let $\pfr \nor R$ correspond to a point $x$. Given that $E(k(x))|_R = E(R/\pfr)$ by \autoref{E(Rp/pRp)|R = E(R/p)}, we have to show $(E(R/\pfr)^{\oplus n}, \kappa)$ is ind-coherent for any Cartier structure $\kappa$. 
	
	By \cite[Theorem 3.4]{Matlis_Injective_modules_over_Noetherian_rings}, we can write $E \coloneqq E(R/\pfr) = \bigcup_{i = 1}^\infty A_i$ where $A_i \coloneqq \set{m \in E}{\pfr^im = 0}$. Moreover, this result says that $A_0 = \Frac(R/\pfr)$, $A_i \inc A_{i + 1}$ for all $i$ and the quotient $A_{i + 1}/A_i$ is a finite dimensional $\Frac(R/\pfr)$-vector space. Let us write \[ E(R/\pfr)^{\oplus n} = \bigcup_{i = 1}^\infty A_i^{\oplus n}. \] First, we prove that all the modules $A_i^{\oplus n}$ are Cartier submodules. 
	
	Let $a = (a_1, \dots, a_n) \in A_i^{\oplus n}$. We have to show that $\kappa(F^r_*a) \in A_i^{\oplus n}$ (in other words, that $\pfr^i\kappa(F^r_*(ae_j)) = 0$). This holds, because for all $\lambda \in \pfr^i$, \[\lambda\kappa(F^r_*a) = \kappa(F^r_*(\lambda^qa)) = 0. \]
	Hence, all $A_i^{\oplus n}$ are Cartier submodules, so we are left to show that these are ind-coherent. For all $i$, there is an exact sequence of Cartier modules \[ 0 \ra A_i^{\oplus n} \ra A_{i + 1}^{\oplus n} \ra A_{i + 1}^{\oplus n}/A_i^{\oplus n} \ra 0, \] so by \autoref{ind-coh closed under extensions}, it is enough to show that $A_1^{\oplus n}$ is ind-coherent and $A_{i + 1}^{\oplus n}/A_i^{\oplus n}$ is ind-coherent for any $i$. 
	
	The first statement follows from \autoref{inj hull of coh is indcoh}. For the second one, recall that $A_{i + 1}/A_i$ is a finite dimensional $\Frac(R/\pfr)$-vector spaces, so there is an isomorphism \[ A_{i + 1}^{\oplus n}/A_i^{\oplus n} \cong  \Frac(R/\pfr)^{\oplus l_n}\] of $\Frac(R/\fp)$-vector spaces. We then conclude the proof by \autoref{inj hull of coh is indcoh}.
\end{proof}
\begin{proof}[Proof of \autoref{inj qalg is inj O_X}]
	The proof is now almost identical to that of \autoref{inj qcoh Cartier is inj O_X}, but let us repeat the argument. 
	
	Let $\Ical$ be an injective of $\IndCoh_X^{C^r}$. By \autoref{injective qcoh is injective O_X}, it is enough to show that it is injective as an object of $\QCoh_X$. 
	
	By \cite[{\href{https://stacks.math.columbia.edu/tag/01PG}{Tag 01PG}}]{Stacks_Project}, $\QCoh_X$ is a locally Noetherian category generated by coherent modules. Therefore, all we have to show that for any injection $\Mcal \inj \Ncal$ of coherent modules and any morphism $f \colon \Mcal \ra \Ical$, there exists an extension $\Ncal \ra \Ical$, i.e. the diagram \[ \begin{tikzcd}
		\Mcal \arrow[r, , hook] \arrow[d, "f"'] & \Ncal \arrow[ld] \\
		\Ical                             &                 
	\end{tikzcd} \] commutes.
	Let $\Mcal'$ be the smallest Cartier submodule of $\Ical$ containing $\im(f)$. Since $\Ical$ is ind-coherent and $\Mcal$ is a Noetherian object, $\Mcal'$ is a coherent Cartier module. 

	Let $\Jcal$ be an injective hull of $\Mcal'$ in $\QCoh_X$. We have a diagram \[ \begin{tikzcd}
		F^r_*\Mcal' \arrow[r, "\kappa_{\cM'}"] \arrow[d, hook] & \Mcal' \arrow[d, hook] \\
		F^r_*\Jcal                               & \Jcal.                   
	\end{tikzcd} \]
	By injectivity of $\Jcal$, there exists a Cartier structure on $\Jcal$ such that the injection $\Mcal' \inj \Jcal$ becomes an injection of Cartier modules. By \autoref{inj hull of coh is indcoh}, $\Jcal$ is an ind-coherent Cartier module. We have a diagram \[ \begin{tikzcd}
		\Mcal' \arrow[d] \arrow[r, hook] & \Jcal \\
		\Ical                       &      
	\end{tikzcd} \] of Cartier modules, so given that $\Jcal$ is ind-coherent and $\Ical$ is injective in $\IndCoh_X^{C^r}$, there exists an morphism $\mu \colon \Jcal \ra \Ical$ completing the above diagram. We then obtain a diagram 
	\[ \begin{tikzcd}
		\cN                            &                          &                        &     \\
	\cM  \arrow[u, hook] \arrow[r] & \cM' \arrow[r, "\inc"'] & \cJ \arrow[r, "\mu"'] & \cI.
	\end{tikzcd} \]
	Since $\Jcal$ is injective in $\QCoh_X$, the proof is complete.
\end{proof}

\vspace{0.5 cm}
\noindent\textit{\textbf{Comparison between some derived categories of ind-coherent Cartier modules.}}

\vspace{0.1 cm}
Throughout, fix a Noetherian $F$-finite $\bF_p$-scheme $X$. Our goal here is to show that the natural functor $D^b(\IndCoh_X^{C^r}) \to D^b_{\indcoh}(\QCoh_X^{C^r})$ is an equivalence of categories. We will use \autoref{inj qalg is inj O_X} in an essential way.

\begin{lem}\label{pushforward of qalg is qalg affine version}
	Let $R$ be a Noetherian $F$-finite ring, $S$ a multiplicative system over $R$ and $M$ an ind-coherent Cartier module over $S^{-1}R$. Then $M|_R$ is ind-coherent. 
\end{lem}
\begin{proof}
	We immediately reduce to the case where $M$ is coherent, and then the proof is the same as that of \autoref{direct sum of residue fields is ind-coh}.
\end{proof}
\begin{rem}
	Note that this is particular to Cartier modules, and fails for Frobenius modules ($S^{-1}R$ with the standard Frobenius structure is almost never ind-coherent, because powers in the denominator will only increase).
\end{rem}
\begin{cor}\label{pushforward of qalg is qalg}
	Let $j \colon U \ra X$ be the inclusion of an open subscheme, and let $\Mcal \in \IndCoh_U^{C^r}$. Then $j_*\Mcal \in \IndCoh_X^{C^r}$.
\end{cor}
\begin{proof}
	The affine case follows from \autoref{pushforward of qalg is qalg affine version}, and the general case follows from the affine one and \autoref{basic props of ind-coherent objects}.
\end{proof}

\begin{defn}
	Define the functor $\ind_X \colon \QCoh_X^{C^r} \ra \IndCoh_X^{C^r}$ by sending a Cartier module to its maximal ind-coherent Cartier submodule. This defines a functor, which is right adjoint to the inclusion $\IndCoh_X^{C^r} \ra \QCoh_X^{C^r}$. 
\end{defn}

We want to show that its higher derived functors vanish on ind-coherent Cartier modules. We first focus on the affine case, so assume that $X = \Spec R$ for now.

\begin{lem}\label{ind exact if first module is qalg}
	Let \[ 0 \ra M \ra N \ra P \ra 0 \] be an exact sequence of Cartier modules, with $M$ ind-coherent. Then the sequence \[ 0 \ra M \ra \ind_X(N) \ra \ind_X(P) \ra 0 \] is exact.
\end{lem}
\begin{proof}
	As a right adjoint, $\ind_X$ is left exact, so it is enough to show that $\ind_X(N) \ra \ind_X(P)$ is surjective. Let $\phi$ denote the morphism $N \ra P$, let $x \in \ind_X(P)$, and let $L$ be a finitely generated Cartier submodule of $P$ containing $x$. By surjectivity of $\phi$, the sequence \[ 0 \ra \ker(\phi|_L) \ra \phi^{-1}(L) \ra L \ra 0 \] is exact, so since $L$ and $\ker(\phi|_L)$ are ind-coherent, $\phi^{-1}(L)$ is also ind-coherent by \autoref{ind-coh closed under extensions}. In other words, $\phi^{-1}(L) \inc \ind_X(N)$, so we are done.
\end{proof}

\begin{lem}\label{strongly ind-acyclic is ind-acyclic for regular}
	Assume that $X = \Spec R$ is connected, that $F^{rj}_*R$ is free for all $j \geq 1$ (with basis $\{F^{rj}_*\lambda_{i, j}\}_{1 \leq i \leq n_j}$), and let $M$ be a Cartier module on $R$ with the following property: for all $j \geq 1$, $1 \leq i \leq n_j$, $m_i \in M$ and $s_i \in R[F^r]$ of degree $rj$ with leading term $\lambda_{i, j}$, there exists $m \in M$ such that \[ m_i = m \cdot s_i\] for all $i$. 
	
	Then $R^l\ind_X(M) = 0$ for all $l > 0$. 
\end{lem}
\begin{proof}
	Let $J$ be an injective Cartier module containing $M$, and let $N \coloneqq J/M$. We have a short exact sequence \[ 0 \ra M \ra J \ra N \ra 0, \] so it is enough to show the two following facts: \\
	
	\begin{enumerate}	
		\item the following sequence is exact \[ 0 \ra \ind_X(M) \ra \ind_X(J) \ra \ind_X(N) \ra 0; \]
		\item the Cartier module $N$ also satisfies the property of the lemma.  \\
	\end{enumerate}

	We begin with the first statement. By left exactness of $\ind_X$, it is enough to show that $\ind_X(J) \to \ind_X(N)$ is surjective.
	
	Let $x \in \ind_X(N)$, and let $y \in J$ be a lift of $x$. By \autoref{locally ind-coh iff pseudo qalg}, there exists $j \geq 0$, $s_1, \dots, s_{n_j} \in R[F^r]$ of degree $rj$ such that for all each $s_i$ has leading term $\lambda_{i, j}$ and $x \cdot s_i = 0$. Hence, $y \cdot s_i \in M$, so by assumption there exists $m \in M$ with the property that \[ m\cdot s_i = y \cdot s_i \] for all $i$.
	
	Then $(y - m) \cdot s_i = 0$ for all $i$, so $y - m \in \ind_X(J)$ by \autoref{locally ind-coh iff pseudo qalg}. Since $y - m$ is also a lift of $x$, we deduce that $\ind_X(J) \ra \ind_X(C)$ is surjective. \\
	
	Let us show the second statement now. It is enough to show that $J$ satisfies the property of the lemma, as this property is closed under taking quotients. Therefore, let $j \geq 0$, $x_1, \dots, x_{n_j} \in J$ and $s_1, \dots, s_{n_j}$ of degree $rj$, such that the leading coefficients of each $s_j$ is $\lambda_{i, j}$. Since $F_*R$ is free, we know by Kunz' theorem (see \cite[{\href{https://stacks.math.columbia.edu/tag/0EC0}{Tag 0EC0}}]{Stacks_Project}) that $R$ is a regular ring. Since $\Spec R$ is connected, we deduce that $R$ is in particular integral, so all the Cartier modules $s_iR[F^r]$ are free of rank one as right $R[F^r]$-modules. Therefore there exist (unique) Cartier module morphisms $s_iR[F^r] \ra J$ sending $s_i$ to $x_i$. Consider the diagram \[ \begin{tikzcd}
		{R[F^r]}                                                    &  & J \\
		{\bigoplus_i s_iR[F^r]} \arrow[rru] \arrow[u] &  &  
	\end{tikzcd} \] The vertical arrow is an injection because the elements $F^{rj}_*\lambda_{i, j}$ form a basis of $F^{rj}_*R$, and the elements $s_i$ are of degree $rj$ with leading term $\lambda_{i, j}$. By injectivity of $J$, there exists a morphism $R[F^r] \ra J$ completing the above diagram. The image of $1 \in R[F^r]$ is then the element of $J$ we were looking for.
\end{proof}

\begin{cor}\label{ind-coh are ind-acyclic over regulars}
	Assume that $\Spec R$ is connected and that $F^{rj}_*R$ is free for all $j \geq 1$. Then for any ind-coherent Cartier module $M$, $R^i\ind_X(M) = 0$ for all $i > 0$.
\end{cor}
\begin{proof}
	Embed $M$ in an injective $J \in \QCoh_X^{C^r}$. By the proof of \autoref{strongly ind-acyclic is ind-acyclic for regular}, $J/M$ satisfies the property of \autoref{strongly ind-acyclic is ind-acyclic for regular}, so it is $\ind_X$-acyclic. Since $J$ is injective, it is also $\ind_X$-acyclic, so we conclude using \autoref{ind exact if first module is qalg}.
\end{proof}
\begin{rem}
	In fact, the same proof shows that for all $i > 1$ and any (non-necessarily ind-coherent) Cartier module $M$, $R^i\ind_X(M) = 0$.
\end{rem}

The idea now is to deduce the case of a general Noetherian $F$-finite ring, using the special case above.

\begin{lem}\label{ind-coh is ind-acyclic affine case}
	Assume that $X = \Spec R$ is connected, and let $M$ be an ind-coherent Cartier module. Then $R^i\ind_X(M) = 0$ for all $i > 0$.
\end{lem}
\begin{proof}
	By \cite[Remark 13.6]{Gabber_notes_on_some_t_structures}, there exists a Noetherian $F$-finite $\bF_p$-algebra $S$ such that $F_*S$ is free and $R$ is a quotient of $S$. Set $R \eqqcolon S/J$, $Y \coloneqq \Spec S$ and let $f \colon X \inj Y$ be the closed immersion corresponding to the surjection $S \surj R$. Since $X$ is connected, the image of $f$ lives in a connected component of $Y$. Hence, we may assume $Y$ is also connected. 
	
	Let $f^\flat$ denote the right adjoint of $f_*$, i.e. the functor given by taking an $S$-module $M$ to its $I$-torsion elements $M[I]$ (see \autoref{affine Grothendieck duality} for a more general setup). We can extend this functor to quasi-coherent Cartier modules. Indeed, if $M$ is a Cartier module on $S$, then the $J$-torsion elements $f^{\flat}M = M[J]$ of $M$ is stable under the Cartier structure of $M$. By construction, this functor preserves ind-coherent objects, and for any Cartier module $M$ on $R$, the natural map $M \ra f^\flat f_*M$ is an isomorphism. 
	
	Given that $f^\flat  \circ \ind_Y = \ind_X \circ f^\flat $ (this is true for their left adjoints), and all these functors preserve injectives (their left adjoints are exact), we obtain that $Rf^\flat  \circ R\ind_Y = R\ind_X \circ Rf^\flat $. 
	
	Let $M$ be an ind-coherent Cartier module on $R$, and let $I \in \IndCoh_S^{C^r}$ be an injective containing $f_*M$. We then have \begin{equation}\label{eq_lemma_ind_acyclic}
		Rf^\flat (R\ind_Y(I)) = R\ind_X(Rf^\flat (I)). 
	\end{equation}
	By \autoref{strongly ind-acyclic is ind-acyclic for regular}, the left-hand side is $Rf^\flat I$, which is exactly $f^\flat I$ since $I$ is injective. 
	
	For the right-hand side, note that the functor $f^\flat $ is taken in the category of quasi-coherent Cartier modules, and $I$ needs not be injective in this category, so a priori it is not clear what $Rf^\flat I$ is. However, since injectives in $\QCoh_X^{C^r}$ are also injective in $\QCoh_X$, this $Rf^\flat I$ is the same as the $Rf^\flat $ of quasi-coherent modules applied to $I$. Given that $I$ is injective in $\IndCoh_X^{C^r}$, it is also injective in $\QCoh_X$ by \autoref{inj qalg is inj O_X}, so $R^if^\flat I = 0$ for all $i > 0$. Thus, \autoref{eq_lemma_ind_acyclic} reads as \[ f^\flat I = R\ind_X(f^\flat I). \] In other words, $f^\flat I$ is $\ind_X$-acyclic too. Since $f_*M$ embeds in $I$ and $f^\flat$ is given by taking the $I$-torsion, $M$ embeds in $f^\flat I$.
	
	Doing the same for $f^\flat (I)/M$ and so on, we find a resolution of $M$ by $\ind_X$-acyclic Cartier modules which are ind-coherent. We can then use this resolution to compute $R\ind_X(M)$, and this gives $R\ind_X(M) = M$.
\end{proof}

Now we are ready to give the statement for general $X$.
\begin{cor}\label{ind-coh is ind-acyclic}
	Let $X$ be a Noetherian $F$-finite $\bF_p$-scheme, and let $\Mcal \in \IndCoh_X^{C^r}$. Then for all $i > 0$, $R^i\ind_X(\Mcal) = 0$.
\end{cor}
\begin{proof}
	Since $X$ is Noetherian, its connected components are open, so we can cover $X$ by finitely many connected affine open subschemes $U_i$. Let $j_i \colon U_i \inj X$ denote the inclusion. By \autoref{pushforward of qalg is qalg}, the functors $(j_i)_*$ preserve ind-coherent modules. For all $i$, embed $\Mcal|_{U_i}$ in an injective $\Ical_i$ in $\IndCoh_{U_i}^{C^r}$. Then $\Mcal$ embeds in the ind-coherent module \[ \Ical \coloneqq \bigoplus_i(j_i)_*\Ical_i. \]
	Let us show that $\Ical$ is $\ind_X$-acyclic. 
	
	Since $(j_i)_* \circ \ind_{U_i} = \ind_X \circ (j_i)_*$ (their left adjoints commute) and both $\ind_{U_i}$ and $(j_i)_*$ preserve injectives (their left adjoints are exact), we get that \[ R\ind_X \circ R(j_i)_* = R(j_i)_* \circ R\ind_{U_i}.\] Given that $\Ical_i$ is injective in the category of ind-coherent modules, it is also injective in $\Mod(\Ocal_X)$ by \autoref{inj qalg is inj O_X}, so in particular it is acyclic for $(j_i)_* \colon \QCoh_{U_i}^{C^r} \ra \QCoh_X^{C^r}$. Indeed, an injective quasi-coherent Cartier module is also injective as a quasi-coherent module by \autoref{inj qcoh Cartier is inj O_X}, so their respective derived pushforwards agree.
	
	Combining the above discussion and \autoref{ind-coh is ind-acyclic affine case}, we obtain that \[ R\ind_X(\Ical) = \bigoplus_i R(j_i)_*R\ind_{U_i}(\Ical_i) = \bigoplus_i R(j_i)_*\Ical_i = \bigoplus_i (j_i)_*(\Ical_i) = \Ical. \]
	Repeating this for $\Ical/\Mcal$ and so on, we can find a resolution of $\Mcal$ by $\ind_X$-acyclic and ind-coherent Cartier modules. Hence, we can compute $R\ind_X$ using this resolution, which concludes the proof.
\end{proof}

\begin{cor}\label{D^b(Coh) = D^b_{coh}(QCoh) for Cartier modules}
	Let $X$ be a Noetherian $F$-finite scheme. Then for any $* \in \{+, b\}$, the natural functor \[ D^*(\IndCoh_X^{C^r}) \ra D^*_{\indcoh}(\QCoh_X^{C^r}) \] is an equivalence of categories. In particular, the natural functor \[ D^b(\Coh_X^{C^r}) \ra D^b_{\coh}(\QCoh_X^{C^r}) \] is also an equivalence of categories.
\end{cor}
\begin{proof}
	The first statement follows from \autoref{ind-coh is ind-acyclic} and \autoref{D_B(A) = D(B) if B objects are B-acyclic}. The second statement is then a consequence of \cite[Theorem 2.9.1]{Bockle_Pink_Cohomological_Theory_of_crystals_over_function_fields}.
\end{proof}

\subsection{Cartier crystals}
Fix a Noetherian $F$-finite $\Ff_p$-scheme $X$.
\begin{defn}
	Let $\Mcal$ be a quasi-coherent Cartier module on $X$.
	\begin{itemize}
		\item It is said to be \emph{nilpotent} if $\kappa_{\cM}^n = 0$ for some $n \geq 0$, where we do the abuse of notations \[ \kappa_{\cM}^n \coloneqq \kappa_{\cM} \circ F^r_*\kappa_{\cM} \circ \dots \circ F^{r(n - 1)}_*\kappa_{\cM}. \] A morphism of Cartier modules whose kernel and cokernel are nilpotent is called a \emph{nil-isomorphism}.
		\item The Cartier module $\cM$ is said to be \emph{locally nilpotent} if for any open affine $U \inc X$, $\Mcal|_U$ is a union of nilpotent Cartier modules. The full subcategory of locally nilpotent modules is called $\LNil$. A morphism of Cartier modules whose kernel and cokernel are locally nilpotent is called an \emph{lnil-isomorphism}.
		\item We set $\Nil = \LNil \cap \Coh_X^{C^r}$.
	\end{itemize}
\end{defn}
\begin{rem}\label{locally nilp iff union of nilp}
	By the same argument as in the proof of \autoref{basic props of ind-coherent objects}, a quasi-coherent Cartier module is locally nilpotent if and only if it is a union of nilpotent Cartier submodules.
\end{rem}

\begin{lem}\label{locally nilp is a Serre subcategory}
	The category $\LNil$ is a Serre subcategory of $\IndCoh_X^{C^r}$.
\end{lem}
\begin{proof}
	The fact that $\LNil \inc \IndCoh_X^{C^r}$ follows from \autoref{nilp is ind-coh}. To show that it is a Serre subcategory, we therefore reduce to the coherent case, where the statement is immediate.
\end{proof}

\begin{defn}
	\begin{itemize}
		\item We define the category of \emph{$r$-Cartier quasi-crystals} (resp. \emph{$r$-ind-crystals}) to be $\QCrys_X^{C^r} \coloneqq \QCoh_X^{C^r}/\LNil$ (resp. $\IndCrys_X^{C^r} \coloneqq \IndCoh_X^{C^r}/\LNil$). We also define the category of \emph{$r$-Cartier crystals} to be $\Crys_X^{C^r} \coloneqq \Coh_X^{C^r}/\Nil$ (as before, we shall often omit the $r$ in discussions if no confusion is likely to arise).
		\item A morphism in either $\QCrys_X^{C^r}, \IndCrys_X^{C^r}$ or $\Crys_X^{C^r}$ will be denoted by dashed arrows $\cM \dra \cN$. 
		\item If $\cM$, $\cN \in \QCoh_X^{C^r}$ are objects which become isomorphic in $\QCrys_X^{C^r}$, we write $\cM \sim_C \cN$. We use the same notation for complexes in the derived category.
	\end{itemize}
\end{defn}

\begin{lem}\label{qcrys and incrys are Groth cats}
	Both $\IndCrys_X^{C^r}$ and $\QCrys_X^{C^r}$ are Grothendieck categories.
\end{lem}
\begin{proof}
	This follows from combining \cite[Proposition 2.4.9]{Bockle_Pink_Cohomological_Theory_of_crystals_over_function_fields}, \autoref{indcoh Cartier is a Groth cat} and \autoref{Cartier modules is a Groth cat}. 
\end{proof}

\begin{lem}\label{Cartier crystals form a Serre subcategory}
	The natural functors $\Crys_X^{C^r} \ra \IndCrys_X^{C^r}$ and $\IndCrys_X^{C^r} \to \QCrys_X^{C^r}$ are fully faithful, and their respective essential images are Serre subcategories of their respective targets. In particular, so does their composition $\Crys_X^{C^r} \to \QCrys_X^{C^r}$.
\end{lem}
\begin{proof}
	The proof is identical to that of \cite[Proposition 3.4.2]{Bockle_Pink_Cohomological_Theory_of_crystals_over_function_fields} (see also the proof of \cite[Proposition 3.3.5, case (c)]{Bockle_Pink_Cohomological_Theory_of_crystals_over_function_fields}).
\end{proof}

\begin{cor}\label{qcrys qalg with crys cohom is crys}
	For $* \in \{-, b\}$, the natural triangulated functor \[ D^*(\Crys_X^{C^r}) \ra D^*_{\crys}(\IndCrys_X^{C^r})\] is an equivalence of categories, where the subscript crys in $D^*_{\crys}(\QCrys_X^{C^r})$ means that cohomology sheaves of the complex must lie in the essential image of $\Crys_X^{C^r} \ra \IndCrys_X^{C^r}$.
\end{cor}
\begin{proof}
	This follows from \cite[{\href{https://stacks.math.columbia.edu/tag/0FCL}{Tag 0FCL}}]{Stacks_Project}.
\end{proof}

\begin{cor}\label{qcrys with crys cohom is crys}
	For $* \in \{+, b\}$ the natural functor \[ D^*(\IndCrys_X^{C^r}) \ra D^*_{\indcrys}(\QCrys_X^{C^r}) \] is an equivalence of categories. Thus, the natural functor \[ D^b(\Crys_X^{C^r}) \ra D^b_{\crys}(\QCrys_X^{C^r}) \] is also an equivalence.
\end{cor}

\begin{proof}
	By \autoref{D_B(A) = D(B) if B objects are B-acyclic}, \autoref{derived_functors_and_localization} and essential surjectivity of $D^+(\IndCoh_X^{C^r}) \to D^+(\IndCrys_X^{C^r})$ (see \cite[Theorem 2.6.2]{Bockle_Pink_Cohomological_Theory_of_crystals_over_function_fields}), all we have to show is that if $\cM \in D^+_{\LNil}(\QCoh_X^{F^r})$, then $R\ind(\cM) \in D^+_{\LNil}(\IndCoh_X^{F^r})$. 
	
	Since $\LNil \inc \IndCoh_X^{C^r}$ (see \autoref{locally nilp is a Serre subcategory}), since follows from the fact that $R\ind(\cM) \cong \cM$ by \autoref{ind-coh is ind-acyclic}.
\end{proof}

\begin{cor}\label{S_nil is a system and eq cat for Cartier modules}
	Let $S$ denote the collection of morphisms $E_1 \to E_2$ in $D^b_{\coh}(\QCoh_X^{C^r})$ whose cone is in $D_{\nil}(\QCoh_X^{C^r})$. Then $S$ is a saturated multiplicative system, and the natural functor \[ S^{-1}D^b_{\coh}(\QCoh_X^{C^r}) \to D^b_{\crys}(\QCrys_X^{C^r}) \] is an equivalence of categories.
\end{cor}
\begin{proof}
	The proof is now identical to that of the analogous statement for Frobenius modules (see the proof of \cite[Theorem 5.3.1]{Bockle_Pink_Cohomological_Theory_of_crystals_over_function_fields}).
\end{proof}

Recall that pushforward via open immersions preserve ind-coherent Cartier modules (\autoref{pushforward of qalg is qalg}). We have an even stronger property with Cartier crystals.

\begin{lemma}\label{pushforward preserves crystals}
	Let $j \colon U \inj X$ be an open immersion between Noetherian $F$-finite $\bF_p$-schemes. Then the following holds:
	\begin{enumerate}
		\item\label{push_crystals_non_derived} the functor $j_* \colon \Crys_U^{C^r} \to \IndCrys_X^{C^r}$ factors through $\Crys_X^{C^r}$; 
		\item\label{push_crystals_derived} the functor $Rj_*$ maps $D^b(\Crys_U^{C^r})$ to $D^b(\Crys_X^{C^r})$. In particular, for any $\cM \in \Coh_U^{C^r}$ and $i \geq 0$, there exists $\cM_i \in \Coh_X^{C^r}$ such that $R^ij_*\cM \sim_C \cM_i$.
		\item\label{push_crystals_any_map} The two points above hold for any separated morphism of finite type.
	\end{enumerate}
\end{lemma}
\begin{proof}
	Although this was already contained in \cite[Theorem 3.2.14]{Blickle_Bockle_Cartier_crystals}, we believe it makes sense to reprove this statement for the sake of the reader.
	
	Let us start with \autoref{push_crystals_non_derived}. By a standard argument, we reduce to showing the following: let $A$ be a Noetherian $F$-finite $\bF_p$-algebra, let $f \in A$, and let $A \to A_f$ be the localization morphism. Then for any $M \in \Coh_{A_f}^{C^r}$, $M|_A \sim_C N$ for some $N \in \Coh_A^{C^r}$. 
	
	Let $M' \inc M$ be a finitely generated sub-$A$-module of $M|_A$ such that $(M')_f = M$, and let $M''$ be the Cartier submodule generated by $M'$ (it is coherent by \autoref{pushforward of qalg is qalg}). It is enough to show that \[ \frac{M''}{f} \sim_C (M'')_f, \] where we use the notations from the proof of \autoref{direct sum of residue fields is ind-coh}. The result follows since for any $n \geq 2$, \[ \kappa_{M''}\left(\frac{M}{f^n}\right) = \kappa_{M''}\left(\frac{f^{rs - n}M}{f^{rs}}\right) \inc \frac{M}{f^r}, \] where $r \geq 1$ is the smallest integer such that $p^{rs} \geq n$ (in particular, $r < n$).
	
	Let us now show the first statement of \autoref{push_crystals_derived}, so let $\cM \in D^b(\Crys_U^{C^r})$. Since $Rj_*\cM$ is computed via taking $j_*$ of an associated \v{C}ech complex and we know by \autoref{push_crystals_non_derived} that the objects defining this \v{C}ech complex live in $\Crys_U^{C^s}$, the result again follows from \autoref{push_crystals_non_derived}. The statement after ``In particular'' is now immediate.

	Finally, \autoref{push_crystals_any_map} is now a consequence of the fact that any separated morphism of finite type is a composition of an open immersion and a proper morphism, see \cite[{\href{https://stacks.math.columbia.edu/tag/0F41}{Tag 0F41}}]{Stacks_Project}.
\end{proof}

\begin{defn}[{\cite[Section 4.3]{Blickle_Bockle_Cartier_crystals}}]
	For $\cM \in D^b(\QCrys_X^{C^r})$, we define \[ \Supp_{\crys}(\cM) \coloneqq \set{x \in X}{\cM_x \not\sim_C 0}. \]
\end{defn}

\begin{lemma}\label{support Cartier crystals}
	For any $\cM \in D^b(\Crys_X^{C^r})$, the subset $\Supp_{\crys}(\cM)$ is closed.
\end{lemma}
\begin{proof}
	For this proof only, we allow ourselves to use results which we will see later. 
	
	This is a local statement, so we may assume that $X$ is semi-separated. By \autoref{main thm duality for crystals} and \autoref{duality preserves supports}, there exists $\cN \in D^b(\Crys_X^{F^r})$ such that \[ \Supp_{\crys}(\cM) = \Supp_{\crys}(\cN).\] Thus, we conclude the proof using \autoref{support F-crystals}.\autoref{supp_crys}.
\end{proof}

\begin{rem}
	We will see in \autoref{proper base change Cartier crystals} that the natural analogue of \autoref{support F-crystals}.\autoref{restricts_and_comes_back} also holds for Cartier crystals. In particular, for any $\cM \in D^b(\Crys_X^{C^r})$, we have that $\cM \sim_C 0$ if and only if $\cM_x \sim_C 0$ for all $x \in X$. The proof of this statement in \cite{Blickle_Bockle_Cartier_crystals} (see Theorem 4.3.2 in \emph{loc.cit}) is direct, while ours will use our duality and the analogous statement for Frobenius crystals (which we recall is an immediate consequence of the Riemann-Hilbert correspondence).
\end{rem}

\subsection{Unit Cartier modules}
Before studying unit Cartier modules, let us recall Grothendieck duality theory for finite morphisms. 

\begin{lemma}\label{affine Grothendieck duality}
	Let $f \colon X \to Y$ be a finite morphism of Noetherian schemes. 
	
	\begin{itemize}
		\item The functor $f_* \colon \QCoh_X \to \QCoh_Y$ has a right adjoint $f^{\flat}$, with the property that for all $\cN \in \QCoh_Y$, \[ f_*f^\flat\Ncal = \HHom(f_*\Ocal_X, \Ncal). \] In addition, for all $\Mcal \in \QCoh_X$ and $\Ncal \in \QCoh_Y$, we have a natural isomorphism 
		\begin{equation}\label{eq iso f^flat}
			f_*\HHom(\Mcal, f^\flat\Ncal) \cong \HHom(f_*\Mcal, \Ncal). 
		\end{equation}
		\item Assume that $X = \Spec R$ and $Y = \Spec S$. For any $R$-module $M$ and $S$-module $N$, the isomorphism \autoref{eq iso f^flat} is given by 
		\[  
		\begin{tikzcd}[row sep=tiny]
			{\Hom_R(M, \Hom_S(R, N))} \arrow[rr, leftrightarrow, "\cong"] &  & {\Hom_S(M, N)}           \\
			\theta \arrow[rr, maps to]                    &  & (m \mapsto \theta(m)(1)), \\
			(m \mapsto (s \mapsto \psi(sm)))              &  & \psi. \arrow[ll, maps to]         
		\end{tikzcd}
		\]
	\end{itemize}
\end{lemma}

\begin{notation}
	Same assumptions as in \autoref{affine Grothendieck duality}. In order to agree with the literature, the derived functor of $f^{\flat}$ will be denoted $f^!$.
\end{notation}

For the remaining of this section, fix a Noetherian $F$-finite scheme $X$ over $\bF_p$. In particular, we can apply \autoref{affine Grothendieck duality} to the absolute Frobenius $F \colon X \to X$.

\begin{cor}[{\cite[Proposition 2.18]{Blickle_Bockle_Cartier_modules_finiteness_results}}]\label{eq def of Cartier mod}
	The datum of a quasi-coherent Cartier module on $X$ is equivalent to that of a quasi-coherent $\cO_X$-module $\cM$, together with a morphism $\kappa_{\cM}^{\flat} \colon \cM \to F^{r, \flat}\cM$.
	
	In particular, for any quasi-coherent Cartier module $\cM$, $F^{r, \flat}\cM$ is also naturally a Cartier module, and $\kappa_{\cM}^{\flat} \colon \cM \to F^{r, \flat}\cM$ is a morphism of Cartier modules.
\end{cor}

Throughout this paper, we will use \autoref{eq def of Cartier mod} without further mention. The morphism $\cM \to F^{r, \flat}\cM$ will be called the \emph{adjoint structural morphism} of $\cM$. Unless stated otherwise, it will always be denoted $\kappa_{\cM}^{\flat}$.

\begin{rem}\label{explicit_adjoint_structural_morphism}
	Thanks to the explicit formula in \autoref{affine Grothendieck duality}, we have that if $M$ is a Cartier module on a Noetherian $F$-finite $\bF_p$-algebra $R$, $m \in M$ and $F^r_*\lambda \in F^r_*R$, then \[ \kappa_M^{\flat}(m)(F^r_*\lambda) = \kappa_M(F^r_*(\lambda m)). \]
\end{rem}

\begin{defn}
	A quasi-coherent Cartier module $\Mcal$ is said to be \emph{unit} if its adjoint structural morphism $\Mcal \ra F^{r, \flat}\Mcal$ is an isomorphism. The full subcategory of unit Cartier modules is denoted by $\QCoh_X^{C^r, \unit}$, and the full subcategory of ind-coherent unit Cartier modules is denoted by	 $\IndCoh_X^{C^r, \unit}$.
\end{defn}

We want to show that $\IndCoh_X^{C^r, \unit}$ is a Grothendieck category. We will do this by showing that it is equivalent to $\IndCrys_X^{C^r}$, via the \emph{unitalization functor}. This will require a few intermediate results.

\begin{defn}
	Define the \emph{unitalization functor} $(\cdot)^u \colon \QCoh_X^{C^r} \to \QCoh_X^{C^r, \unit}$ by sending a quasi-coherent Cartier module $\cM$ to \[ \cM^u \coloneqq \colim F^{rn, \flat}\cM, \] where the colimit is taken in the category of Cartier modules (see \autoref{eq def of Cartier mod}).
\end{defn}

\begin{lem}\label{unitalization is 0 iff locally nilpotent}
	A quasi-coherent Cartier module $\Mcal$ is locally nilpotent if and only if $\Mcal^u = 0$. Furthermore, in this case, also $R^i(\cdot)^u(\cM) = 0$ for all $i > 0$.
\end{lem}
\begin{proof}
	Let us begin with the part before ``Furthermore''. We immediately reduce to the affine case by \autoref{locally nilp iff union of nilp}. Let $X = \Spec R$, and let $M \in \QCoh_R^{C^r}$. By \autoref{explicit_adjoint_structural_morphism}, 
	\[ \ker(\kappa_M^{\flat, n}) = \set{m \in M}{\kappa_{M}^n(F^n_*(\lambda m)) = 0 \esp \forall \lambda \in R} \] so by definition, $M$ is locally nilpotent if and only if $M = \bigcup_n\ker(\kappa_M^{\flat, n})$.
	
	To see the final statement, note that for all $i \geq 0$, we have \[ R^i(\cdot)^u = \colim R^iF^{rn, \flat}. \] Indeed, these functors define a $\delta$-functor (filtered colimits are exact), agreeing with $(\cdot)^u$ for $i = 0$, and the $R^i(\cdot)^u$ vanish on injective objects. 
	
	In particular, this higher derived functor commute with filtered colimits (each $F^j_*\cO_X$ are compact objects), so it is enough to show that if $\cM$ is nilpotent, then each $R^i(\cdot)^u(\cM)$ are trivial. By our work above, \[ R(\cdot)^u(\cM) = \colim F^{rn, !}\cM, \] where the symbol $\colim$ should be understood as the pointwise colimit (or equivalently the homotopy colimit), and the maps $F^{rn, !}\cM \to F^{r(n + 1), !}\cM$ are given by applying $F^{rn, !}$ to the composition \[ \cM \to F^{r, \flat}\cM \to F^{r, !}\cM. \] Thus, the result is immediate.
\end{proof}

\begin{lem}\label{map to unitalization is a nil-iso}
	Let $\cM \in \QCoh_X^{C^r}$. For any $n \geq 1$, the composite of the adjoint structural morphism $\kappa_{\cM}^{\flat, n} \colon \cM \to F^{nr, \flat}\cM$ is a nil-isomorphism. Furthermore, so is the natural map $\cM \to \cM^u$.
\end{lem}
\begin{proof}
	We have an exact diagram 
	\[  
		\begin{tikzcd}
			0 \arrow[r] & {\ker(\kappa_{\cM}^{\flat})} \arrow[d] \arrow[r]    & \cM \arrow[d, "{\kappa_{\cM}^{\flat}}"] \arrow[r] & {F^{r, \flat}\cM} \arrow[d, "{F^{r, \flat}\kappa_{\cM}^{\flat}}"] \\
			0 \arrow[r] & {F^{m, \flat}\ker(\kappa_{\cM}^{\flat})} \arrow[r] & {F^{m, \flat}\cM} \arrow[r]                         & {F^{2r, \flat}\cM},                          
		\end{tikzcd}
	\] which automatically gives that the adjoint structural morphism of $\ker(\kappa_{\cM}^{\flat})$ is trivial. Similarly, that of $\coker(\kappa_{\cM}^{\flat})$ is trivial, so $\kappa_{\cM}^{\flat}$ is indeed a nil-isomorphism. We then deduce the result for each $\kappa_{\cM}^{\flat, n}$. 
	
	Finally, since the kernel (resp. cokernel) of the unitalization morphism is a colimit of the kernel (resp. cokernel) of each $\kappa_{\cM}^{\flat, n}$, we conclude.
\end{proof}

\begin{cor}\label{unit Cartier modules form a Grothendieck category}
	The unitalization functor induces a equivalences of categories \[ \QCrys_X^{C^r} \to \QCoh_X^{C^r, \unit}. \]
	
	In particular, $\QCoh_X^{C^r, \unit}$ is a Grothendieck category, the unitalization functors $\QCoh_X^{C^r} \to \QCoh_X^{C^r, \unit}$ is exact, and the inclusion $\QCoh_X^{C^r, \unit} \inc \QCoh_X^{C^r}$ preserves injectives.
	
	Finally, all the statements above still hold, when replacing $\QCoh_X^{C^r}$ (resp. $\QCrys_X^{C^r}$) by $\IndCoh_X^{C^r}$ (resp. $\IndCrys_X^{C^r}$).
\end{cor}
\begin{proof}
	By \autoref{unitalization is 0 iff locally nilpotent} and \autoref{derived_functors_and_localization}, the unitalization functor factors though a functor $\QCrys_X^{C^r} \to \QCoh_X^{C^r}$. The essential image of this functor is $\QCoh_X^{C^r, \unit}$, so we have to show that it is fully faithful. 
	
	Let $f, g \colon \cM \dra \cN$ be two morphisms in $\QCrys_X^{C^r}$, and suppose that $f^u = g^u$. Note that the diagram 
	\[ 
		\begin{tikzcd}
			\cM \arrow[r, "f", dashed] \arrow[d] & \cN \arrow[d] \\
			\cM^u \arrow[r, "f^u"]               & \cN^u        
		\end{tikzcd}
	\] commutes in $\QCrys_X^{C^r}$. Doing the same for $g$ and using that both vertical arrows are isomorphisms in $\QCrys_X^{C^r}$ by \autoref{map to unitalization is a nil-iso}, we conclude that $f = g$. The proof for fullness is similar.
	
	The category $\QCoh_X^{C^r, \unit}$ is then a Grothendieck category by \autoref{qcrys and incrys are Groth cats}. The unitalization functor is then exact. Indeed, by \autoref{unitalization is 0 iff locally nilpotent}, we can factor it as \[ \QCoh_X^{C^r} \to \QCrys_X^{C^r} \to \QCoh_X^{C^r, \unit}, \] and both these functors are exact (see \cite[{\href{https://stacks.math.columbia.edu/tag/02MS}{Tag 02MS}}]{Stacks_Project} for the first one). Finally, since unitalization is also left adjoint to the inclusion $\QCoh_X^{C^r, \unit} \ra \QCoh_X^{C^r}$, this latter functor preserves injectives.
	
	The proof in the ind-coherent case is identical.
\end{proof}

\begin{cor}\label{injective unit is injective O_X}
	An injective in either $\IndCoh_X^{C^r, \unit}$ or $\QCoh_X^{C^r, \unit}$ is also injective in $\Mod(\Ocal_X)$.
\end{cor}
\begin{proof}
	Combine \autoref{unit Cartier modules form a Grothendieck category}, \autoref{inj qcoh Cartier is inj O_X} and \autoref{inj qalg is inj O_X}.
\end{proof}

\begin{lem}\label{inclusion from unit to O_X is exact is X regular}
	If $X$ is regular, then the inclusions $\IndCoh_X^{C^r, \unit} \to \Mod(\cO_X)$ and $\QCoh_X^{C^r, \unit} \to \Mod(\cO_X)$ are exact.
\end{lem}
\begin{proof}
	Since they is left exact, we only have to show that the cokernel of a morphism of unit Cartier modules is again unit. This is a consequence from the fact that $F^\flat$ is exact by Kunz' theorem (see \cite[Theorem 2.1]{Kunz_Characterizations_of_regular_local_rings_of_char_p}).
\end{proof}
\begin{rem}
	In general, the inclusion $\IndCoh_X^{C^r, \unit} \ra \QCoh_X^{C^r}$ is not exact. Since it is a right adjoint, this inclusion preserves kernels. However, the cokernel in the former category is the unitalization of the cokernel is the latter (the usual cokernel needs not be unit if the scheme is not regular).
\end{rem}

\begin{lem}\label{F^flat acyclic implies iota-acyclic}
	Let $\iota$ denote the inclusion $\IndCoh_X^{C^r, \unit} \ra \IndCoh_X^{C^r}$, and let $\Ncal \in \IndCoh_X^{C^r, \unit}$. If $\Ncal$ is $F^{r, \flat}$-acyclic as a quasi-coherent module, then it is also $\iota$-acyclic.
	
	The same statement also holds if we replace all instances of $\IndCoh$ by $\QCoh$.
\end{lem}
\begin{proof}
	We are going to construct an injective resolution of $\Ncal$ in $\IndCoh_X^{C^r, \unit}$, which will also be exact in $\IndCoh_X^{C^r}$ (this would conclude the proof by definition). 
	
	Let $\Ical$ be an injective in $\IndCoh_X^{C^r, \unit}$ containing $\Ncal$. We then have an exact sequence \[ 0 \ra \Ncal \ra \Ical \ra \Ical/\Ncal \ra 0 \] in $\QCoh_X$, so by assumption, the sequence \[ 0 \ra F^{r, \flat}\Ncal \ra F^{r, \flat}\Ical \ra F^{r, \flat}(\Ical/\Ncal) \ra 0 \] is also exact. Applying the snake lemma on \[ \begin{tikzcd}
		0 \arrow[r] & \Ncal \arrow[r] \arrow[d] & \Ical \arrow[r] \arrow[d] & \Ncal \arrow[r] \arrow[d] & 0 \\
		0 \arrow[r] & F^{r, \flat}\Ncal \arrow[r]     & F^{r, \flat}\Ical \arrow[r]     & F^{r, \flat}(\Ical/\Ncal) \arrow[r]   & 0
	\end{tikzcd} \] shows that $F^{r, \flat}(\Ical/\Ncal)$ is a unit Cartier module. Since $\Ical$ is also injective in $\QCoh_X$ by \autoref{injective unit is injective O_X}, we deduce that $\Ical/\Ncal$ is $F^{r, \flat}$-acyclic. We can then continue this process to obtain our sought injective resolution.
\end{proof}

\begin{lem}\label{morphisms of crystals are easy for Cartier modules}
	Let $f \colon (\Mcal, \phi) \dra (\Ncal, \psi)$ be a morphism in $\Crys_X^{C^r}$. Then there exists $n \geq 0$ and a commutative diagram \[ \begin{tikzcd}
		\Mcal \arrow[rr, dashed, "f"]                                       &  & \Ncal \\
		F^{rn}_*\Mcal \arrow[u, "\kappa_{\cM}^n"] \arrow[rru, "f'"'] &  &      
	\end{tikzcd} \] in $\Crys_X^{C^r}$, where $f' \colon F^{nr}_*\Mcal \ra \Ncal$ is a morphism in $\Coh_X^{C^r}$.
\end{lem}
\begin{proof}
	The proof is identical to that of \cite[Proposition 3.4.6]{Bockle_Pink_Cohomological_Theory_of_crystals_over_function_fields}.
\end{proof}

\begin{cor}\label{morphism of crystal to unit is actual morphism}
	Let $\cM$, $\cN$ be two coherent Cartier modules, with $\cN$ unit, and let $f \colon \cM \dra \cN$ be a morphism of crystals. Then $f$ can be represented by a morphism of Cartier modules.
\end{cor}
\begin{proof}
	By \autoref{morphisms of crystals are easy for Cartier modules} and the adjunction $F^e_* \dashv F^{e,\flat}$, this morphism can be represented by a morphism $\cM \to F^{rn, \flat}\cN$. Composing with $(\kappa_{\cN}^{n, \flat})^{-1}$ gives the result.
\end{proof}

\section{Duality between Frobenius modules and Cartier modules}\label{section duality F-modules and Cartier modules}
Our main goal is to show that there is a duality between Frobenius modules and Cartier modules. The idea is to pair Frobenius modules and Cartier modules via the $\HHom$ functor of $\cO_X$-modules. Our duality will then be defined by pairing with a dualizing complex admitting the structure of a unit Cartier module.

Afterwards, we will show that our duality passes to crystals, and finally we will use our $\HHom$-pairings to prove a version of the local duality theorem for Frobenius modules and Cartier modules. \\

As in the previous sections, fix an integer $r \geq 1$ and set $q = p^r$.

\subsection{The Hom pairings}\label{subsection Hom pairings}
Fix an $\bF_p$-scheme $X$. \\ \\
\noindent\textit{\textbf{Pairing a Frobenius module and a Cartier module to obtain a Cartier module}}
\begin{constr}\label{construction Hom from F to Cartier}
	Let $\Mcal \in \Mod(\Ocal_X[F^r])$ and $\Ncal \in \Mod(\Ocal_X[C^r])$. Define the following Cartier structure on $\HHom_{\cO_X}(\Mcal, \Ncal)$: \[ \left(f \colon \Mcal|_U \ra \Ncal|_U \right) \mapsto \left(\kappa_{\cN} \circ F^r_*f \circ \tau_{\cM} \colon \Mcal|_U \ra \Ncal|_U \right) \] for all opens $U \inc X$. Due to the presence of $F^r_*f$ in the formula, this morphism is indeed $q^{-1}$-linear. We will often drop the symbol $\cO_X$ in $\HHom_{\cO_X}(-, -)$.
\end{constr}

\begin{prop}\label{Hom from F to Cartier}
	\begin{itemize}
		\item The pairing defined above defines a functor \[ \HHom(-, -) \colon \Mod(\Ocal_X[F^r])^{op} \times \Mod(\Ocal_X[C^r]) \ra \Mod(\Ocal_X[C^r]) \] 
		\item The above functor gives rise to \[ \Rcal\HHom(-, -) \colon D(\Ocal_X[F^r])^{op} \times D^+(\Ocal_X[C^r]) \ra D(\Ocal_X[C^r]). \]
		Furthermore, if we fix either variable, the induced functor is triangulated.
		\item The diagram \[ \begin{tikzcd}
			{D(\cO_X[F^r])^{op} \times D^+(\cO_X[C^r])} \arrow[rr, "\Rcal\HHom"] \arrow[d] &  & {D(\cO_X[C^r])} \arrow[d] \\
			D(\Ocal_X)^{op} \times D^+(\Ocal_X) \arrow[rr, "\Rcal\HHom"']                                &  & D(\Ocal_X)                           
		\end{tikzcd} \] commutes.
	\end{itemize}			
\end{prop}
\begin{proof} 
	The first two statements are immediate, and the third one follows from \autoref{inj qcoh Cartier is inj O_X}.
\end{proof}

We can restrict this pairing to coherent Frobenius modules and quasi--coherent Cartier modules. Note that although the image of the functor below lies in $\IndCoh_X^{C^r}$, we will not need it (and in fact, this statement will not hold for the other pairing).

\begin{prop}\label{Hom from F to Cartier for qalg}
	Assume that $X$ is Noetherian and $F$-finite. Then the $\HHom$ pairing defined above induces a functor \[ \HHom(-, -) \colon (\Coh_X^{F^r})^{op} \times \IndCoh_X^{C^r} \ra \QCoh_X^{C^r}. \] The associated diagram \[ \begin{tikzcd}
		{D^-(\Coh_X^{F^r})^{op} \times D^+(\IndCoh_X^{C^r})} \arrow[rr, "\Rcal\HHom"] \arrow[d] &  & {D^+(\QCoh_X^{C^r})} \arrow[d] \\
		D(\Ocal_X)^{op} \times D^+(\Ocal_X) \arrow[rr, "\Rcal\HHom"']                                &  & D(\Ocal_X)                           
	\end{tikzcd} \] commutes, and the top arrow is triangulated if we fix either variable.
\end{prop}
\begin{proof}
	This is immediate from the definitions and \autoref{inj qalg is inj O_X} (note that we restrict to bounded above complexes of coherent Frobenius modules because infinite products of quasi-coherent modules are not quasi-coherent in general).
\end{proof}

\vspace{0.5 cm}
\noindent\textit{\textbf{Pairing a Cartier module and a unit Cartier module to obtain a Frobenius module}}

\vspace{0.1 cm}
\noindent Throughout, we assume that $X$ is Noetherian and $F$-finite.
\begin{constr}\label{construction Hom from Cartier to F}
	Let $\Mcal \in \QCoh_X^{C^r}$ and $\Ncal \in \QCoh_X^{C^r, \unit}$. Define the following Frobenius module structure on $\HHom(\Mcal, \Ncal)$: 
	\[ \left(f \colon \Mcal|_U \ra \Ncal|_U \right) \mapsto \left((\kappa_{\cN}^{\flat})^{-1} \circ F^{r, \flat}f \circ \kappa_{\cM}^{\flat} \colon \Mcal|_U \ra \Ncal|_U \right) \] for all opens $U \inc X$. Due to the presence of $F^{r, \flat}f$ in the formula, this morphism is indeed $q$-linear.
\end{constr}

\begin{rem}\label{explicit_pairing_Cartier_to_F}
	Let $R$ be a Noetherian $F$-finite $\bF_p$-algebra, $M \in \QCoh_R^{C^r}$, $N \in \QCoh_R^{C^r, \unit}$, $f \in H \coloneqq \Hom_R(M, N)$ and $m \in M$. Then $\tau_H(f)(m)$ is the unique element in $N$ such that for all $\lambda \in R$, \[ \kappa_N\left(F^r_*(\lambda\tau_H(f)(m))\right) = f(\kappa_M(F^r_*(\lambda m))). \]
\end{rem}

\begin{prop}\label{Hom from Cartier to F} 
	\begin{itemize}
		\item The pairing defined above gives rise to a functor \[ \HHom(-, -) \colon (\QCoh_X^{C^r})^{op} \times \QCoh_X^{C^r, \unit} \ra \Mod(\Ocal_X[F^r]), \] which in particular induces \[ \HHom(-, -) \colon (\Coh_X^{C^r})^{op} \times \IndCoh_X^{C^r, \unit} \ra \QCoh_X^{F^r}.
		\]
		\item The latter pairing induces
		\[ \Rcal\HHom(-, -) \colon D^-(\Coh_X^{C^r})^{op} \times D^+(\IndCoh_X^{C^r, \unit}) \ra D^+(\QCoh_X^{F^r}). \] This functor is triangulated when we fix the second variable.
		\item The diagram \[ \begin{tikzcd}
			{D^-(\Coh_X^{C^r})^{op} \times D^+(\IndCoh_X^{C^r, \unit})} \arrow[rr, "\Rcal\HHom"] \arrow[d] &  & D^+(\QCoh_X^{F^r}) \arrow[d] \\
			D(\Ocal_X)^{op} \times D^+(\Ocal_X) \arrow[rr, "\Rcal\HHom"']                                       &  & D(\Ocal_X)                  
		\end{tikzcd} \] commutes, where the arrow $D^+(\IndCoh_X^{C^r, \unit}) \ra D(\Ocal_X)$ is the right derived functor of the inclusion $\IndCoh_X^{C^r, \unit} \ra \Mod(\Ocal_X)$. 
	\end{itemize}			
\end{prop}
\begin{proof}
	The two first points are immediate, and the third point follows from \autoref{injective unit is injective O_X}.
\end{proof}

Recall that we observed in \autoref{rem def Cartier module}.\autoref{tensor product of Cartier mod and unit F-mod} that given $\cM \in \QCoh_X^{C^r}$ and $\cL \in \QCoh_X^{F^r}$ unit, then we can give a natural Cartier module structure to the tensor product $\cM \otimes \cL$. We show that our pairing behaves the expected way with respect to this operation.

\begin{lemma}\label{behaviour Hom and tensor product}
	Let $\cM \in D(\QCoh_X^{C^r})$, $\cN \in D^+(\IndCoh_X^{C^r, \unit})$ and let $\cV$ be a unit Frobenius module, which is a locally free of finite rank as an $\cO_X$-module. Then we have an isomorphism of Frobenius modules \[ \Rcal\HHom(\cM \otimes \cV, \cN) \cong \Rcal\HHom(\cM, \cN) \otimes \cV^{\vee}, \] where the adjoint Frobenius module structure on $\cV^{\vee}$ is given by $\left((\tau_{\cV}^*)^\vee)\right)^{-1}$. 
\end{lemma}
\begin{proof}
	It is enough to show that the natural morphism \[ \HHom(\cM \otimes \cV, \cN) \to \HHom(\cM, \cN) \otimes \cV^{\vee} \] is a morphism of Frobenius modules for any $\cM \in \QCoh_X^{C^r}$ and $\cN \in \IndCoh_X^{C^r}$. This can be done locally, so we may assume that $X = \Spec R$ and that $\cV \cong \cO_R^{\oplus n}$ (as $R$-modules). Let $A \in M_{n \times n}(R)$ be such that the induced Frobenius module structure on $R^{\oplus n}$ is given by $s \mapsto F^r_*(As^q)$, and let $\psi \colon R \to F^{r, *}R$ denote the natural isomorphism. Then by construction, the morphism 
	\begin{equation}\label{eq:hom_and_tensor_product} \begin{tikzcd}
		R^{\oplus n} \arrow[r, "\psi^{\oplus n}"] & {F^{r, *}R^{\oplus n}} \arrow[r, "\tau_{R^{\oplus n}}^*"] & R^{\oplus n}
	\end{tikzcd} \end{equation} 
	agrees with multiplication by $A$. In particular, $A$ is invertible. Furthermore, the induced Cartier structure on $M \otimes_R R^{\oplus n}$ (which we identify with $M^{\oplus n}$) is given by \[ \begin{tikzcd}[row sep = small]
	F^r_*M^{\oplus n} \arrow[rr] &  & M^{\oplus n}                                     \\
	F^r_*m \arrow[rr, maps to]   &  & \kappa_M^{\oplus n}\left(F^r_*(A^{-1})^Tm\right).
\end{tikzcd}\] Note that the Frobenius module structure on $\cV^{\vee}$ is identified (as in \autoref{eq:hom_and_tensor_product}) with the matrix $(A^T)^{-1} = (A^{-1})^T$, so the result follows from the explicit definition of the pairings.
\end{proof}
\begin{rem}
	The same result holds with the pairing between Frobenius modules and Cartier modules, but we will not need it.
\end{rem}

\subsection{The duality}\label{subsection duality}

Fix a Noetherian $F$-finite $\bF_p$-scheme $X$.

\begin{defn}\label{def property has a unit dualizing complex}
	A complex $\omega_X^\bullet \in D^+(\IndCoh_X^{C^r, \unit})$ is said to be a \emph{unit dualizing complex} if $R\iota(\omega_X^\bullet)$ is a dualizing complex (see \cite[{\href{https://stacks.math.columbia.edu/tag/0A87}{Tag 0A87}}]{Stacks_Project}), where $\iota$ denotes the inclusion $\IndCoh_X^{C^r, \unit} \ra \Mod(\Ocal_X)$.
\end{defn}
\begin{rem}
	\begin{itemize}
		\item We will see in \autoref{upper shriek functor} that the datum of a unit dualizing complex is (somewhat surprisingly) equivalent to that of a dualizing complex $\cN \in D_{\coh}(\QCoh_X)$, together with an isomorphism $\cN \to F^{r, !}\cN$ in the derived category. 
		\item Although we could have defined a unit dualizing complex to live in $D^+(\QCoh_X^{C^r, \unit})$ (and the definitions would be equivalent), we explicitly use ind-coherence in the proof of \autoref{having a unit dc is equivalent to the naive notion}. That is why we decided to force it to begin with, and hence why we defined the derived pairing functors as they are. 
		\item We will denote by $\iota$ any functor that forgets (at least) the ``unit'' part. For example $\iota$ could mean any of the functors $\IndCoh_X^{C^r} \to \Mod(\cO_X)$, $\IndCoh_X^{C^r, \unit} \to \IndCoh_X^{C^r}$ or $\IndCoh_X^{C^r, \unit} \to \QCoh_X^{C^r}$.
	\end{itemize}
\end{rem}

From now on, fix a unit dualizing complex $\omega_X^\bullet$ on $X$ (we assume there exists at least one). Given any $\iota$ (see the remark above), we make the abuse of notation $\omega_X^\bullet = R\iota(\omega_X^\bullet)$ (this does not cause any trouble, and simplifies notations).

\begin{lem}\label{first properties of RHom(- , omega)}
	For any $* \in \{\emptyset, +, -, b\}$, pairing with $\omega_X^\bullet$ restrict and corestrict to triangulated functors \[ \Rcal\HHom(-, \omega_X^\bullet) \colon D^*_{\coh}(\Ocal_X[F^r])^{op} \ra D^{*'}_{\coh}(\Ocal_X[C^r]) \] and \[ \Rcal\HHom(-, \omega_X^\bullet) \colon D^*_{\coh}(\QCoh_X^{C^r})^{op} \ra D^{*'}_{\coh}(\Ocal_X[F^r]), \] where we set $\emptyset' \coloneqq \emptyset$, $+' \coloneqq -$, $-' \coloneqq +$ and $b' \coloneqq b$. 
	
	Similarly, for $* \in \{-, b\}$, they restrict to triangulated functors \[ \Rcal\HHom(-, \omega_X^\bullet) \colon D^*(\Coh_X^{F^r})^{op} \ra D^{*'}_{\coh}(\QCoh_X^{C^r}), \] and \[ \Rcal\HHom(-, \omega_X^\bullet) \colon D^*(\Coh_X^{C^r})^{op} \ra D^{*'}_{\coh}(\QCoh_X^{F^r}). \]
\end{lem}
\begin{rem}
	In the first and third $\Rcal\HHom$'s, we used the abuse of notations $\omega_X^\bullet = R\iota(\omega_X^\bullet)$.
\end{rem}
\begin{proof}
	Let us only show the statement before ``Similarly'' (the proof of the other statement is identical). 
	
	The only new statement is that the functors $\Rcal\HHom(-, \omega_X^\bullet)$ map to $D^*_{\coh}$ and remain triangulated. The fact that they are triangulated follow from the canonical triangulated structure we can put on $D^*_{\coh}$, since coherent modules from a weak Serre subcategory of $\Mod(\Ocal_X)$ (hence this is also true for Frobenius modules and Cartier modules). 
	
	These restrictions map to $D^*_{\coh}$ because of the last points of \autoref{Hom from F to Cartier} and \autoref{Hom from Cartier to F}, together with the analogous statement for $\Ocal_X$-modules (\cite[{\href{https://stacks.math.columbia.edu/tag/0A89}{Tag 0A89}}]{Stacks_Project}).
\end{proof}
\begin{defn}\label{definition duality}
	Define the duality functors $\Dd$ as the following compositions:
	\[ \begin{tikzcd}
		D^b(\Coh_X^{F^r})^{op} \arrow[rrr, "{\Rcal\HHom(-, \omega_X^\bullet)}"] & & & {D^b_{\coh}(\QCoh_X^{C^r})} \arrow[rrr, "\cong"] & & & D^b(\Coh_X^{C^r})
	\end{tikzcd} \] and \[ \begin{tikzcd}
	D^b(\Coh_X^{C^r})^{op} \arrow[rrr, "{\Rcal\HHom(-, \omega_X^\bullet)}"] & & & {D^b_{\coh}(\QCoh_X^{F^r})} \arrow[rrr, "\cong"] & & & {D^b(\Coh_X^{F^r})},
	\end{tikzcd} \] where the equivalence $D^b_{\coh}(\QCoh_X^{C^r}) \ra D^b(\Coh_X^{C^r})$ comes from \autoref{D^b(Coh) = D^b_{coh}(QCoh) for Cartier modules}, and the equivalence $D^b_{\coh}(\QCoh_X^{F^r}) \ra D^b(\Coh_X^{F^r})$ comes from \autoref{D^b(Coh) = D^b_{coh}(QCoh) for F-modules}.
\end{defn}

We want to show that $\Dd \circ \Dd = id$.

\begin{lem}\label{evaluation map is okay sheaf case}
	Let $\Mcal \in \Coh_X^{F^r}$ (resp. $\Coh_X^{C^r}$) and $\Ncal \in \QCoh_X^{C^r, \unit}$. Then the evaluation morphism \[ ev \colon \Mcal \ra \HHom(\HHom(\Mcal, \Ncal), \Ncal) \] is a morphism of Frobenius modules (resp. Cartier modules).
\end{lem}
\begin{proof}
	During this proof, we will skip the notations $F^r_*$ to make things clearer. This is a local computation, so we may assume that $X = \Spec R$. Let $M$ and $N$ correspond respectively to $\cM$ and $\cN$,  let $H \coloneqq \Hom(M, N)$ and $L \coloneqq \Hom(\Hom(M, N), N)$. Let $m \in M$ and $f \in H$.
	\begin{itemize}
		\item We start with the case $M \in \Coh_R^{F^r}$. We have to show that $\tau_L(ev(m))(f) = ev(\tau_M(m))(f)$. By \autoref{explicit_pairing_Cartier_to_F}, we have that for all $\lambda \in R$, 
		\begin{align*}
			 \kappa_N\left(\lambda\tau_L(ev(m))(f)\right) &= ev(m)\left(\kappa_H(\lambda f)\right) \\
			 &= \kappa_H(\lambda f)(m) \\
			 &= \kappa_N(\lambda f(\tau_M(m))). 
		 \end{align*}  
	 	Since $N$ is unit, we obtain by \autoref{explicit_adjoint_structural_morphism} that $\tau_L(ev(m))(f) = f(\tau_M(m))$,
		which is exactly what we wanted.
		\item Now, let us show the case $M \in \Coh_R^{C^r}$. We have to show that $\kappa_L(ev(m))(f) = ev(\kappa_M(m))(f)$. We have 
		\begin{align*}
			 \kappa_L(ev(m))(f) &= (\kappa_N \circ ev(m) \circ \tau_H)(f) \\
			 &= \kappa_N(\tau_H(f)(m)) \\
			 &= f(\kappa_M(m)) \\
			 &= ev(\kappa_M(m))(f), 
		\end{align*} where the third equality comes from \autoref{explicit_pairing_Cartier_to_F}. Hence, we are also done with this case.
	\end{itemize}
\end{proof}

\begin{theorem}\label{main thm duality}
	Let $X$ be a Noetherian, $F$-finite and semi-separated $\bF_p$-scheme with has a unit dualizing complex $\omega_X^\bullet$. Then the two functors $\Dd$ defined in \autoref{definition duality} are essential inverses to each other, and hence define equivalences of triangulated categories \[ D^b(\Coh_X^{C^r})^{op} \cong D^b(\Coh_X^{F^r}) \] and \[ D^b_{\coh}(\QCoh_X^{C^r})^{op} \cong D^b_{\coh}(\QCoh_X^{F^r}). \]
\end{theorem}
\begin{proof}
	The second statement follows from the first one, \autoref{D^b(Coh) = D^b_{coh}(QCoh) for F-modules} and \autoref{D^b(Coh) = D^b_{coh}(QCoh) for Cartier modules}, so we only need to show the first statement. 
	
	Fix an injective resolution $\Ical$ of $\omega_X^\bullet$ in $\IndCoh_X^{C^r, \unit}$, so that both pairing functors $\Rcal\HHom(-, \omega_X^\bullet)$ are given by $\HHom^{\bullet}(-, \Ical)$. We will find a natural isomorphism $1 \to \bD \circ \bD$. 
	
	Let $A \in \{F, C\}$, let $i \colon D^b(\Coh_X^{A^r}) \to D^b_{\coh}(\QCoh_X^{A^r})$ and $j \colon  D^b(\QCoh_X^{A^r}) \ra D^b_{\qcoh}(\Ocal_X[A^r])$ denote the canonical functors, and let $\Mcal \in D^b(\Coh_X^{A^r})$. By definition, we have a natural isomorphism \[ i\Dd(\Mcal) \ra \Rcal\HHom(\Mcal, \omega_X^\bullet) \] Applying $\Rcal\HHom(-, \omega_X^\bullet)$ gives a natural isomorphism \[ \Rcal\HHom(\Rcal\HHom(\Mcal, \omega_X^\bullet), \omega_X^\bullet) \cong \Rcal\HHom(i\Dd(\Mcal), \omega_X^\bullet). \] Note that we have a natural isomorphism $\Rcal\HHom(i\Dd(\Mcal), \omega_X^\bullet) \cong ji\Dd(\Dd(\Mcal))$. Indeed, \[ ji\Dd(\Dd(\Mcal)) \cong j\Rcal\HHom(\Dd(\Mcal), \omega_X^\bullet) = \HHom^{\bullet}(\Dd(\Mcal), \Ical) \cong \Rcal\HHom(i\Dd(\Mcal), \omega_X^{\bullet}) \] Recall that pairing with a dualizing complex induces an anti-equivalence on $D^b(\cO_X)$, and the natural isomorphism is the evaluation map, see \cite[Lemma V.1.2]{Hartshorne_Residues_and_Duality} (use \cite[Corollary V.2.3]{Hartshorne_Residues_and_Duality} to see that their definition of dualizing complex agrees with ours).
	
	Combining the above discussion, \autoref{evaluation map is okay sheaf case} and the last point of \autoref{Hom from Cartier to F}, we deduce that the natural evaluation map $\Mcal \ra \HHom^{\bullet}(\HHom^{\bullet}(\Mcal, \Ical), \Ical)$ is a quasi-isomorphism of complexes of Frobenius/Cartier modules. Indeed, although \autoref{evaluation map is okay sheaf case} only treats the sheaf case, the case of complexes follows from the fact that the evaluation map of complexes is a product of zero maps and evaluation maps in the sheaf case (see \cite[Lemma V.1.2]{Hartshorne_Residues_and_Duality}).
	
	Thus, we have a natural isomorphism \[ ji\Mcal \cong \Rcal\HHom(\Rcal\HHom(\Mcal, \omega_X^\bullet), \omega_X^\bullet), \] so we conclude the existence of a natural isomorphism \[ ji\Mcal \cong ji\Dd(\Dd(\Mcal)) \] in $D^b_{\coh}(\cO_X[A^r])$. Since both $i$ and $j$ are equivalences (see \autoref{D^b(Coh) = D^b_{coh}(QCoh) for Cartier modules}, \autoref{D(QCoh) = D_qcoh(O_X[C])} for Cartier modules, and \autoref{D^b(Coh) = D^b_{coh}(QCoh) for F-modules} for Frobenius modules), we deduce that we have a natural isomorphism $\Mcal \cong \Dd(\Dd(\Mcal))$. \\ 
%
\end{proof}

\begin{rem}
	The only moment where semi-separatedness was used in the proof was because we applied \autoref{D^b(Coh) = D^b_{coh}(QCoh) for F-modules}.
\end{rem}

\subsection{Passing to crystals}
In order to define the pairing at the level of crystals, we have to show that nil-isomorphisms of coherent Frobenius modules (resp. Cartier modules) are sent to nil-isomorphisms of Cartier modules (resp. Frobenius modules).

\begin{lem}\label{Hom preserves nilpotence}
	Let $X$ be any $\bF_p$-scheme (resp. a Noetherian $F$-finite $\bF_p$-scheme), let $\Mcal \in \Mod(\Ocal_X[F^r])$ (resp. $\cM \in \QCoh_X^{C^r}$) be nilpotent, and let $\Ncal \in \Mod(\Ocal_X[C^r])$ (resp. $\QCoh_X^{C^r, \unit})$. Then $\HHom(\Mcal, \Ncal)$ is nilpotent.
\end{lem}
\begin{proof}
	This it immediate from the definitions of the pairings.
\end{proof}

\begin{defn}
	Let $f \colon \Mcal \ra \Mcal'$ be a morphism of complexes in the derived category of Frobenius modules (resp. Cartier modules). The morphism $f$ is said to be a \emph{nil-quasi-isomorphism} if $\cH^i(f)$ is a nil-isomorphism for all $i \in \bZ$. 
\end{defn}
\begin{lem}\label{nil-isomorphisms are preserved under Hom}
	Let $X$ be a Noetherian (resp. Noetherian and $F$-finite) $\Ff_p$-scheme, let $f \colon \Mcal \ra \Mcal'$ be a nil-quasi-isomorphism in $D^b(\QCoh_X^{F^r})$ (resp. $D^b(\QCoh_X^{C^r})$), and let $\Ncal \in D^+(\QCoh_X^{C^r})$ (resp. $D^+(\IndCoh_X^{C^r, \unit})$). Then the morphism \[ g \colon \Rcal\HHom(\Mcal', \Ncal) \ra \Rcal\HHom(\Mcal, \Ncal)\] is a nil-quasi-isomorphism, where $g \coloneqq \cR\HHom(f, \cN)$.
\end{lem}
\begin{proof}
	Let $C(f)$ denote the cone of $f$, and consider the exact triangle \[ \Mcal \xrightarrow{f} \Mcal' \ra C(f) \xrightarrow{+1} \] Note that by assumption, $C(f)$ consists of nilpotent objects. 
	
	Applying $\Rcal\HHom(-, \Ncal)$ gives an exact triangle \[ \Rcal\HHom(C(f), \Ncal) \ra \Rcal\HHom(\Mcal', \Ncal) \xrightarrow{g} \Rcal\HHom(\Mcal, \Ncal) \xrightarrow{+1} \] By \autoref{Hom preserves nilpotence} and the fact that $C(f) \in D^b$, also $\Rcal\HHom(C(f), \Ncal)$ consists of nilpotent objects. Using the long exact sequence in cohomology, we deduce that $g$ is a nil-quasi-isomorphism.  
\end{proof}

\begin{lem}\label{nil-quasi-isos are preserved under both equiv D(Coh) = Dcoh(QCoh)}
	Both equivalences of triangulated categories $D^b(\Coh_X^{F^r}) \cong D^b_{\coh}(\QCoh_X^{F^r})$ and $D^b(\Coh_X^{C^r}) \cong D^b_{\coh}(\IndCoh_X^{C^r})$ preserve nil-quasi-isomorphisms.
\end{lem}
\begin{proof}
	This follows from the fact that both functors $D^b(\Coh_X^{F^r}) \ra D^b_{\coh}(\QCoh_X^{F^r})$ and $D^b(\Coh_X^{C^r}) \ra D^b_{\coh}(\IndCoh_X^{C^r})$ are equivalences, and have the property that the image of a map is a nil-quasi-isomorphism if and only if the map itself is a nil-quasi-isomorphism.
\end{proof}

We can finally give our duality theorem between Frobenius crystals and Cartier crystals.
\begin{theorem}\label{main thm duality for crystals}
	Let $X$ be a Noetherian, $F$-finite and semi-separated $\bF_p$-scheme with a unit dualizing complex. Then the functors $\Dd$ induce equivalences of triangulated categories \[ D^b(\Crys_X^{C^r})^{op} \cong D^b(\Crys_X^{F^r}) \] and \[ D^b_{\crys}(\QCrys_X^{C^r})^{op} \cong D^b_{\crys}(\QCrys_X^{F^r}), \] where $D^b_{\crys}(\QCrys)$ means the full sucategory of complexes whose cohomology sheaves are in the essential image of $\Crys \ra \QCrys$.
\end{theorem}
\begin{proof}
	By definition of the functors $\Dd$, \autoref{main thm duality}, \autoref{nil-isomorphisms are preserved under Hom} and \autoref{nil-quasi-isos are preserved under both equiv D(Coh) = Dcoh(QCoh)}, the functors $\Dd$ induce an equivalence of categories \[ S^{-1}_{\nil}D^b(\Coh_X^{C^r})^{op} \cong S^{-1}_{\nil}D^b(\Coh_X^{F^r}), \] where $S_{\nil}$ is the collection of morphisms whose cone has nilpotent cohomology sheaves (these localizations exist by \cite[Theorem 2.6.2]{Bockle_Pink_Cohomological_Theory_of_crystals_over_function_fields}). By \emph{loc. cit.}, we obtain the first statement. The second statement then follows from \autoref{qcrys with crys cohom is crys} and \autoref{D^b(Coh) = D^b_{coh}(QCoh) for F-modules}.
\end{proof}

\subsection{Local duality}\label{section local duality}
In this section, we use the pairings we defined earlier to obtain an analogue of local duality for both Frobenius modules and Cartier modules.

\begin{defn}\label{def_local_cohomology}
	Let $R$ be a local $\bF_p$-algebra, and let $I = (a_1, \dots, a_m)$ be a finitely generated ideal of $R$. We denote by $C^{\bullet}$ the usual complex \[ R \to \bigoplus_i R_{a_i} \to \bigoplus_{i < j}R_{a_ia_j} \to \dots \to R_{a_1 \dots a_m}, \] and endow it with its natural Frobenius module structure. For $A \in \{F, C\}$, we define \[ \RGamma_I \colon D(\QCoh_R^{A^r}) \to D(\QCoh_R^{A^r}) \] by the formula \[ M \mapsto M \otimes C^{\bullet}. \]
\end{defn}
\begin{rem}\label{remark def local cohomology}
	\begin{enumerate}
		\item Since $C^{\bullet}$ consists of flat Frobenius modules, we do not need to take derived tensor products.
		\item This definition makes sense for Cartier modules too, since $C^{\bullet}$ consists of unit Frobenius modules (recall \autoref{rem def Cartier module}.\autoref{tensor product of Cartier mod and unit F-mod}). In particular, it preserves (complexes of) unit Cartier modules.
		\item The definition of $\RGamma_I$ only depends on $V(I) \subseteq \Spec R$, since it is an adequate right adjoint (the precise statement and proof are the same as in \cite[{\href{https://stacks.math.columbia.edu/tag/0A6R}{Tag 0A6R}}]{Stacks_Project}).
		\item If $R$ is Noetherian, then this local cohomology functor agrees with the usual definition when seen as $\cO_X$-modules (see \cite[{\href{https://stacks.math.columbia.edu/tag/0955}{Tag 0955}}]{Stacks_Project}, or \cite[Exercise III.2.3.(e)]{Hartshorne_Algebraic_Geometry}).
		\item\label{rem:local cohomology for Cartier modules} In fact, this functor agrees with taking the derived functor of sections supported on a closed subscheme for Cartier modules, since injective Cartier modules are also injective as $\cO_X$-modules. Furthermore, as a corollary of \autoref{inj qalg is inj O_X}, one deduces that local cohomology groups preserve ind-coherence for Cartier modules.
	\end{enumerate}
\end{rem}

From now on, fix a Noetherian $F$-finite local $\bF_p$-algebra $(R, \fm, k)$, together with a unit dualizing complex $\omega_R^{\bullet}$ (it exists by \autoref{finite type over a unit dc has a unit dc}). Up to a shift, we may assume that it is normalized (i.e. that $R\iota(\omega_R^{\bullet})$ is normalized, where $\iota \colon \IndCoh_R^{C^r, \unit} \to \Mod(R)$ is the forgetful functor).

\begin{lemma}\label{local cohomology of normalized unit dc is inj hull of the residue field}
	The complex of Cartier modules $E \coloneqq R\iota(\RGamma_{\fm}(\omega_R^{\bullet}))$ is only supported in degree $0$ and as an $R$-module, it is an injective hull of $k$.
\end{lemma}
\begin{proof}
	Let us show that for any $M \in D^+(\QCoh_R^{C^r, \unit})$, $R\iota(M \otimes C^{\bullet}) \cong R\iota(M) \otimes C^{\bullet}$. This will conclude by \cite[{\href{https://stacks.math.columbia.edu/tag/0A82}{Tag 0A82}}]{Stacks_Project}.
	
	It is enough to show that if $I \in \QCoh_R^{C^r, \unit}$ is injective, then for all $a \in R$, $I_a$ is $\iota$-acyclic. Note that $I$ is injective in $\Mod(R)$ by \autoref{injective unit is injective O_X}, hence so is $I_a$ (a filtered colimit of injectives is injective in a Noetherian category). Thus, we conclude by \autoref{F^flat acyclic implies iota-acyclic}.
\end{proof}

Let $\bighat{R}$ denote the $\fm$-adic completion of $R$.

\begin{lemma}\label{completion commutes with Frobenius pushforward}
	The commutative square 
	\[ \begin{tikzcd}
		R \arrow[r, "F"] \arrow[d] & R \arrow[d] \\
		\bighat{R} \arrow[r, "F"'] & \bighat{R} 
	\end{tikzcd} \] is a pushout square of rings. In particular, for all $M \in \Mod(R)$, $F_*M \otimes_R \bighat{R} \cong F_*(M \otimes_R \bighat{R})$.
\end{lemma}
\begin{proof}
	The first statement follows from \cite[{\href{https://stacks.math.columbia.edu/tag/0394}{Tag 0349}}]{Stacks_Project}, and the second statement is then a consequence of flat base change.
\end{proof}

Thanks to \autoref{completion commutes with Frobenius pushforward}, we can take pullback of (unit) Cartier modules along the morphism $R \to \bighat{R}$. We can now give the local duality theorem for Frobenius modules and Cartier modules.

\begin{prop}\label{local duality Frobenius version}
	Let $(A, B) \in \{(F, C), (C, F)\}$, and let $M \in D_{coh}(\QCoh_R^{A^r})$. Then there is a natural isomorphism \[ \RHom(M, \omega_R^{\bullet}) \otimes_R \bighat{R} \cong \RHom(\RGamma_{\fm}(M), E) \] in $D(\QCoh_R^{B^r})$. In particular, for all $i \in \bZ$, we have \[ \Ext^{-i}(M, \omega_R^{\bullet}) \sim_B 0 \iff H^i_{\fm}(M) \sim_A 0. \]
\end{prop}

\begin{proof}
	First, note that $E$ is a unit Cartier module by \autoref{local cohomology of normalized unit dc is inj hull of the residue field}, so $\RHom(\RGamma_{\fm}(M), E)$ makes also sense when $M$ is a Cartier module. Let $I$ be an injective resolution of $\omega_R^{\bullet}$. Define $\RHom(M, \omega_R^{\bullet}) \to \RHom(\RGamma_{\fm}, E)$ by \[ \Hom^{\bullet}(M, I) \to \Hom^{\bullet}(M \otimes C^{\bullet}, I \otimes C^{\bullet}) \expl{\cong}{\autoref{local cohomology of normalized unit dc is inj hull of the residue field}} \Hom^{\bullet}(M \otimes C^{\bullet}, E) \cong \RHom(\RGamma_{\fm}(M), E), \] where the last isomorphism comes from the fact that $E$ is injective as an $R$-module, hence acyclic for the $\Hom$-pairings by \autoref{Hom from F to Cartier for qalg} and \autoref{Hom from Cartier to F}.
	
	Since $E$ is an $\bighat{R}$-module, so is the right-hand side. Therefore, we have an induced morphism \[ \RHom(M, \omega_R^{\bullet}) \otimes_R \bighat{R} \to \RHom(\RGamma_{\fm}(M), E). \] The fact that this is a quasi-isomorphism can be checked after passing to $D(R)$, and there is follows from the usual local duality theorem, see \cite[{\href{https://stacks.math.columbia.edu/tag/0A84}{Tag 0A84}} and {\href{https://stacks.math.columbia.edu/tag/0A06}{Tag 0A06}}]{Stacks_Project} (the fact that their morphism agrees with ours is a straightforward computation, see the proof of \cite[{\href{https://stacks.math.columbia.edu/tag/0A6Y}{Tag 0A6Y}}]{Stacks_Project}). \\
	
	Now, let $i \in \bZ$. We have an isomorphism of $B$-modules \[ \Ext^{-i}(M, \omega_R^{\bullet}) \otimes_R \bighat{R} \cong \Hom(H^i_{\fm}(M), E), \] so we are left to show the two following statements: for all $N \in \Coh_R^{A^r}$,
	
	\[ \begin{cases*}
		N \sim_A 0 \iff N \otimes_R \bighat{R} \sim_A 0 \\
		N \sim_A 0 \iff \Hom(N, E) \sim_B 0
	\end{cases*} \]

	\noindent The first statement follows immediately from faithful flatness of $R \to \bighat{R}$ and \autoref{completion commutes with Frobenius pushforward}. To show the second statement, note that the evaluation morphism \[ N \to \Hom(\Hom(N, E), E) \] (see \autoref{evaluation map is okay sheaf case}) is, by Matlis duality, an isomorphism after taking completion. Hence, we are done by the first statement and the fact that taking $\Hom(-, E)$ preserves nilpotence (see \autoref{Hom preserves nilpotence}).
\end{proof}

\begin{rem}
	We use the notations from \autoref{local duality Frobenius version}. Suppose that $H^i_{\fm}(M) \sim_A 0$ (and hence that $\Ext^{-i}(M, \omega_R^{\bullet}) \sim_B 0$). Since $\Ext^{-i}(M, \omega_R^{\bullet})$ is coherent, then it has a finite index of nilpotence. In particular, it follows from the proof of \autoref{local duality Frobenius version} that also $H^i_{\fm}(M)$ is nilpotent (and not only locally nilpotent). For Frobenius modules, finiteness results of this kind were already observed for example in \cite{Blickle_Bockle_Cartier_modules_finiteness_results}. For Cartier modules, this can be seen much more directly, thanks to the following lemma.
\end{rem}

\begin{lemma}
	Let $M \in \Coh_R^{C^r}$. Then for all $i \geq 0$, $H^i_{\fm}(M)$ is nil-isomorphic to a coherent Cartier module.
\end{lemma}
\begin{proof}
	By \autoref{remark def local cohomology}.\autoref{rem:local cohomology for Cartier modules}, we have \[ H^i_{\fm}(M) \cong \colim \Ext^i(R/\fm^n, M) \] as Cartier modules. However, the morphisms $R/\fm^n \to R/\fm$ are nil-isomorphisms of Frobenius crystals, so $H^i_{\fm}(M) \sim_C \Ext^i(R/\fm, M)$.
\end{proof}

\section{Some applications of the duality}

\subsection{Compatibilities and the upper shriek functor}\label{upper shriek functor}

Given a finite type morphism $f \colon X \to Y$ between Noetherian schemes, there exists the upper-shriek functor $f^! \colon D^+_{\qcoh}(\cO_Y) \to D^+_{\qcoh}(\cO_X)$ (see \cite[{\href{https://stacks.math.columbia.edu/tag/0AU5}{Tag 0AU5}}]{Stacks_Project}, where it is named $f_{new}^!$). For example, for $f$ finite, we have $f^! = Rf^{\flat}$ by construction. \\

As an application of our duality, we will construct a similar functor $f^! \colon D^b(\Crys_Y^{C^r}) \to D^b(\Crys_X^{C^r})$.

We point out that we restrict our attention to Cartier crystals (and not Cartier modules) for simplicity. We leave the interested reader to work out the details in the case of Cartier modules (when the morphism $f$ has finite Tor-dimension). 

\begin{prop}\label{having a unit dc is equivalent to the naive notion}
	Let $X$ be a Noetherian $F$-finite $\bF_p$-scheme, and let $\omega_X^\bullet \in D^b_{\coh}(\Ocal_X)$ be a dualizing complex on $X$, together with an isomorphism \[ \omega_X^\bullet \to F^{r, !}\omega_X^\bullet \] in $D(\cO_X)$. Then there exists a unit dualizing complex $\Mcal$ such that $R\iota\Mcal \cong \omega_X^\bullet$ (where $\iota \colon \IndCoh_X^{C^r, \unit} \to \Mod(\Ocal_X)$ is the inclusion), and such that the diagram \[ \begin{tikzcd}
		R\iota\Mcal \arrow[rr] \arrow[d] &  & F^{r, !}R\iota\Mcal \arrow[d] \\
		\omega_X^\bullet \arrow[rr]                 &  & F^{r, !}\omega_X^\bullet                
	\end{tikzcd} \] commutes, where the top arrow is induced by the Cartier structure on $\Mcal$.
\end{prop}

\begin{rem}
	\begin{itemize}
		\item In such a situation, we say that $\omega_X^\bullet \cong F^{r, !}\omega_X^\bullet$ \emph{comes from the unit dualizing complex} $\Mcal$.
	\end{itemize}
\end{rem}

\begin{proof}
	By assumption, there exists a dualizing complex $\omega_X^\bullet \in D^b_{\coh}(\Ocal_X)$ admitting an isomorphism \[ \omega_X^\bullet \to RF^{r, \flat}\omega_X^\bullet  = F^{r, !}\omega_X^\bullet. \] 
	Let $\Ncal \coloneqq RQ(\omega_X^\bullet)$, where $Q \colon \Mod(\Ocal_X) \to \QCoh_X$ is the right adjoint of the inclusion (it exists by \cite[{\href{https://stacks.math.columbia.edu/tag/077P}{Tag 077P}}]{Stacks_Project}). By \cite[{\href{https://stacks.math.columbia.edu/tag/09T4}{Tag 09T4}}]{Stacks_Project}, $RQ$ is an inverse of the equivalence $D^+(\QCoh_X) \to D^+_{\qcoh}(\Ocal_X)$, so we obtain a commutative square \[ \begin{tikzcd}
		\omega_X^\bullet \arrow[rr, "\cong"]          &  & F^{r, !}\omega_X^\bullet          \\
		\Ncal \arrow[u, "\cong"] \arrow[rr, "\cong"'] &  & F^{r, !}\Ncal \arrow[u, "\cong"']
	\end{tikzcd} \] in $D(\Ocal_X)$. Since $Q$ preserves injectives (its left adjoint is exact), we see that by construction of right derived functors, $\Ncal$ is a complex of injective quasi-coherent modules, so we have a homotopy equivalence $\Ncal \to F^{r, \flat}\Ncal$ of complexes of quasi-coherent sheaves on $X$. In particular, each sheaf $\Ncal^j$ is a quasi-coherent Cartier module.
	
	Consider $\Mcal \coloneqq R\ind_X(\Ncal)$. Since $\Mcal$ has coherent cohomology sheaves, we know from the proof of \autoref{D^b(Coh) = D^b_{coh}(QCoh) for Cartier modules} that $\Mcal$ is quasi-isomorphic to $\Ncal$ in $D(\QCoh_X^{C^r})$, and in particular it has coherent cohomology sheaves. Since $\ind_X$ preserves injectives, $\Mcal$ is a complex of injectives in $D(\IndCoh_X^{C^r})$. By the above and \autoref{injective unit is injective O_X}, the Cartier structure on $\Mcal$ induces an isomorphism \[ \Mcal \to F^{r, \flat}\Mcal = F^{r, !}\Mcal \] in $D^+(\IndCoh_X^{C^r})$. 
	
	Note that $\Mcal^u$ consists on injective objects in $\IndCoh_X^{C^r, \unit}$. Indeed, it is enough to show that they are injective in $\IndCoh_X^{C^r}$. This follows from the facts that $F^{r, \flat}$ preserves injectives, and that a filtered colimit of injectives in $\IndCoh_X^{C^r}$ is injective (this category is Noetherian by \autoref{unit Cartier modules form a Grothendieck category}).
	
	Thus, $R\iota\Mcal^u = \Mcal^u$ so by the above discussion, we would be done if we knew that the natural map
	\[ \Mcal \to \Mcal^u \]
	was a quasi-isomorphism. 
	
	This morphism is by construction a morphism of complexes, given by the colimit of the morphisms \[ (F^{nr})^\flat\Mcal \to (F^{(n + 1)r})^\flat\Mcal. \] Hence, it is enough that all these maps are quasi-isomorphisms (taking filtered colimits is exact). For $e = 0$ this holds by construction. For $e \geq 1$, this is a consequence of the fact that $F^{rn, !}$ preserves isomorphisms in the derived category, and $F^{nr, !}\Mcal = F^{nr, \flat}\Mcal$ ($\Mcal$ is a complex of injective objects).
\end{proof}

In the proof of \autoref{having a unit dc is equivalent to the naive notion}, it may look like there could be many different choices of unit dualizing complexes inducing the same isomorphism $\omega_X^\bullet \cong F^{r, !}\omega_X^\bullet$. In fact, this is not the case.

\begin{prop}\label{naturality of induced unit dc}
	Let $X$, $\omega_X^\bullet$ be as in \autoref{having a unit dc is equivalent to the naive notion}, and assume $\omega_X^\bullet$ comes from two unit dualizing complexes $\Mcal_1$ and $\Mcal_2$. If $X$ is semi-separated, then there is a unique isomorphism $\Mcal_1 \cong \Mcal_2$ in $D^b_{\coh}(\IndCoh_X^{C^r, \unit})$ such that the diagram \[ \begin{tikzcd}
		R\iota(\cM_1) \arrow[rr, "\cong"] \arrow[rd, "\cong"'] &                    & R\iota(\cM_2) \arrow[ld, "\cong"] \\
		& \omega_X^{\bullet} &                                  
	\end{tikzcd} \] commutes.
\end{prop}
\begin{proof}
	For the existence, it is enough to show that $R\iota\Mcal_1 \cong R\iota\Mcal_2$ in $D(\QCoh_X^{C^r})$, since $(R\iota\Ncal)^u \cong \Ncal$ for all $\Ncal \in D^+(\IndCoh_X^{C^r, \unit})$. 
	
	Let $\Omega$ be any unit dualizing complex such that $R\iota(\Omega) \cong \omega_X^\bullet$ (we do not mind matching the Cartier structure on $\Omega$ and the isomorphism $\omega_X^\bullet \cong F^{r, !}\omega_X^\bullet$), and let $\Dd$ be the duality functor induced by $\Omega$. In addition, let $\theta \colon F^r_*\omega_X^\bullet \to \omega_X^\bullet$ be the morphism corresponding to $\omega_X^\bullet \cong F^{r, !}\omega_X^\bullet$, and define $\theta' \coloneqq \Rcal\HHom(\theta, \omega_X^\bullet) \colon \Ocal_X \to F^r_*\Ocal_X$.
	
	Let $j = 1, 2$. Since the diagram \[ \begin{tikzcd}
		F^r_*R\iota\Mcal_j \arrow[rr] \arrow[d, "\cong"'] &  & R\iota\Mcal_j \arrow[d, "\cong"] \\
		F^r_*\omega_X^\bullet \arrow[rr, "\theta"']                 &  & \omega_X^\bullet                
	\end{tikzcd} \] commutes, we obtain that 
	\[ \begin{tikzcd}
		\Dd(R\iota\Mcal_j) \arrow[rr] \arrow[d, "\cong"'] &  & F^r_*\Dd(R\iota\Mcal_j) \arrow[d, "\cong"] \\
		\Ocal_X \arrow[rr, "\theta'"']                 &  & F^r_*\Ocal_X              
	\end{tikzcd} \] also commutes. Given that $\Ocal_X$ is a sheaf, we must have an isomorphism\[ \Dd(R\iota\Mcal_j) \cong (\Ocal_X, \theta') \] as Frobenius modules. In particular $\Dd(R\iota\Mcal_1) \cong \Dd(R\iota\Mcal_2)$. Applying $\Dd$ again concludes the existence of an isomorphism (the fact that the diagram in the statement commutes is automatic by construction). The uniqueness is immediate, since $R\iota(\cN)^u \cong \cN$ for all $\cN \in D^+(\IndCoh_X^{C^r, \unit})$.
\end{proof}

Let $X$ be a Noetherian, $F$-finite and semi-separated $\bF_p$-scheme. Because of \autoref{naturality of induced unit dc}, we will now see a unit dualizing complex and an element $\omega_X^{\bullet} \in D^b_{\coh}(\cO_X)$ with an isomorphism $\omega_X^{\bullet} \cong F^!\omega_X^{\bullet}$ as the same objects.

\begin{cor}\label{induced unit dualizing complex by a morphism}
	Let $f \colon X \to Y$ be a finite type morphism between Noetherian, $F$-finite and semi-separated $\bF_p$-schemes. Assume that $Y$ has a unit dualizing complex $\omega_Y^{\bullet}$. Then $X$ has a unique induced unit dualizing complex $\omega_X^{\bullet}$, which agrees with $f^!\omega_Y^{\bullet}$ when seen as a complex of quasi-coherent modules, and which matches the Cartier structures.
\end{cor}
\begin{proof}
	Consider the induced functor $f^! \colon D^+_{\coh}(\cO_Y) \to D^+_{\coh}(\cO_X)$ constructed in \cite[{\href{https://stacks.math.columbia.edu/tag/0AU5}{Tag 0AU5}}]{Stacks_Project} (they name it $f^!_{new}$), and let $\omega_X^{\bullet} \coloneqq f^!\omega_Y^{\bullet}$. Since $\omega_Y^{\bullet}$ is a unit dualizing complex, there exists an isomorphism $\omega_Y^{\bullet} \cong F^{r, !}\omega_Y^{\bullet}$. Applying $f^!$ to this isomorphism, there is an isomorphism $\omega_X^{\bullet} \cong F^{r, !}\omega_X^{\bullet}$. We then conclude by \autoref{having a unit dc is equivalent to the naive notion} and \autoref{naturality of induced unit dc}.
\end{proof}

Thanks to \autoref{induced unit dualizing complex by a morphism}, we can now construct upper shriek functors for Cartier crystals.

\begin{defn}\label{definition upper shriek functor}
	Let $f \colon X \to Y$ be a finite type morphism between Noetherian, $F$-finite and semi-separated $\bF_p$-schemes, and assume that $Y$ has a unit dualizing complex $\omega_Y^{\bullet}$. Let $\omega_X^{\bullet}$ be the induced unit dualizing complex on $X$ (see \autoref{induced unit dualizing complex by a morphism}). In addition, let $\bD_X$ and $\bD_Y$ denote, respectively, the induced duality functors from $\omega_X^{\bullet}$ and $\omega_Y^{\bullet}$. 
	
	We define $f^! \colon D^b(\Crys_Y^{C^r}) \to D^b(\Crys_X^{C^r})$ by the formula \[ f^! \coloneqq \bD_X \circ f^* \circ \bD_Y \] (note that $f^* \colon \Crys_Y^{F^r} \to \Crys_X^{C^r}$ is exact by \autoref{main thm Bhatt Lurie}, and the analogous fact for étale $\bF_q$-sheaves).
\end{defn}

\begin{lemma}\label{counit of adjuntion for upper shriek}
	Same notations as in \autoref{definition upper shriek functor}. If $f$ is proper, there exists a unique morphism $Rf_*\omega_X^{\bullet} \to \omega_Y^{\bullet}$ in $D^b(\Coh_Y^{C^r})$, which corresponds to the usual counit of the adjunction $Rf_* \dashv f^!$ at the level of $\cO_X$-modules.
\end{lemma}
\begin{proof}
	Let $\theta \colon Rf_*\omega_X^{\bullet} \to \omega_Y^{\bullet}$ be the natural trace morphism in $D^b(\cO_Y)$ coming from the adjunction $Rf_* \dashv f^!$. By naturality, the diagram \[ \begin{tikzcd}
		F^r_*Rf_*\omega_X^{\bullet} \arrow[d] \arrow[r, "F^r_*\theta"] & F^r_*\omega_Y^{\bullet} \arrow[d] \\
		Rf_*\omega_X^{\bullet} \arrow[r, "\theta"']                    & \omega_Y^{\bullet}               
	\end{tikzcd} \] commutes in $D^b(\cO_X)$. Let $D_X$ (resp. $D_Y$) denote the duality functor with respect to $\omega_X^{\bullet}$ (resp. $\omega_Y^{\bullet}$) at the level of coherent $\cO_X$-modules.

 	Applying $D_Y$ and letting $\theta' \coloneqq D_Y(\theta)$, the diagram 
	\[ \begin{tikzcd}
		F^r_*D_Y(Rf_*\omega_X^{\bullet})      &  & F^r_*\cO_Y \arrow[ll, "F^r_*\theta'"'] \\
		D_Y(Rf_*\omega_X^{\bullet}) \arrow[u] &  & \cO_Y \arrow[u] \arrow[ll, "\theta'"] 
	\end{tikzcd} \] also commutes. By \cite[{\href{https://stacks.math.columbia.edu/tag/0AUE}{Tag 0AUE}}]{Stacks_Project}, we have $Rf_* \circ D_Y = D_X \circ Rf_*$, so the complexes on the left are supported in degree $\geq 0$. Since $\cO_Y$ is supported in degree $0$, $\theta'$ factors as \[ \cO_Y \xrightarrow{\theta''} \cH^0(D_Y(Rf_*\omega_X^{\bullet})) \to D_Y(Rf_*\omega_X^{\bullet}). \] By the commutativity above, $\theta''$ is a morphism of Frobenius modules.

	Let $\bD_Y$ denote the duality with respect to the unit dualizing complex $\omega_Y^{\bullet}$. As Frobenius modules, $\cH^0(D_Y(Rf_*\omega_X^{\bullet})) = \cH^0(\bD_Y(Rf_*\omega_X^{\bullet}))$, so we obtain a morphism \[ \cO_Y \to \bD_Y(Rf_*\omega_X^{\bullet}) \] in $D^b(\Coh_X^{C^r})$. Applying $\bD_Y$ gives the statement.
\end{proof}

As a consequence, we deduce a compatibility between proper pushforwards and our duality functor.

\begin{cor}\label{compatibility between duality and pushforwards}
	Same notations as in \autoref{definition upper shriek functor}. Let $(A, B) \in \{(F, C), (C, F)\}$. If $f$ is proper, then there is a natural isomorphism \[ Rf_* \circ \bD_X \cong \bD_Y \circ Rf_* \] as functors $D^b(\Coh_X^{A^r})^{op} \to D^b(\Coh_X^{B^r})$.
	
	The same result also holds as functors $ D^b(\Crys_X^{A^r})^{op} \to D^b(\Crys_X^{B^r})$.
\end{cor}
Note that since both $X$ and $Y$ are semi-separated, the morphism $f$ is semi-separated too. Hence, by \autoref{all derived pushforwards are the same for F-modules} and \autoref{rem_ok_if_semisep}, there is no ambiguity in what we mean by $Rf_*$.
\begin{proof}
	Let $\cM \in D^b(\Coh_Y^{A^r})$. Note that it is enough to show the result for $A = F$ (indeed, for $\cM \in D^b(\Coh_Y^{C^r})$, write $\cM = \bD(\bD(\cM))$ and apply the result for $\bD(\cM) \in D^b(\Coh_Y^{F^r})$), so let us assume that we are in this case. We want to show that \[ Rf_*\cR\HHom(\cM, \omega_X^{\bullet}) \cong \cR\HHom(Rf_*\cM, \omega_Y^{\bullet}) \] 
	First, we claim that for any injective object $\cI \in \IndCoh_X^{C^r, \unit}$ and any $\cN \in \IndCoh_X^{F^r}$, the sheaf $\HHom(\cN, \cI)$ is acyclic for $f_*$ in $\IndCoh_X^{C^r}$. To see this, recall that by \autoref{inj qalg is inj O_X} and \autoref{all derived pushforwards are the same for F-modules}, it suffices to show the claim for the functor $Rf_*$ taken in $\Mod(\cO_X)$. Then, it is a consequence of \autoref{injective unit is injective O_X} and the same argument as in \cite[Corollary 2.3.(b)]{Milne_Etale_Cohomology}).
	
	Thanks to this claim, we have a natural morphism \[ Rf_*\cR\HHom(\cM, \omega_X^{\bullet}) \to \cR\HHom(Rf_*\cM, Rf_*\omega_X^{\bullet}). \] 
	
	Now, consider the composition \[ Rf_*\cR\HHom(\cM, \omega_X^{\bullet}) \to \cR\HHom(Rf_*\cM, Rf_*\omega_X^{\bullet}) \expl{\to}{\autoref{counit of adjuntion for upper shriek}} \cR\HHom(Rf_*\cM, \omega_Y^{\bullet}). \]
	To show that this is an isomorphism, we only have to show this at the level of quasi-coherent sheaves. Hence, this is a consequence of \cite[{\href{https://stacks.math.columbia.edu/tag/0AUE}{Tag 0AUE}}]{Stacks_Project}.
\end{proof}

\begin{prop}\label{f^! does not depend on the unit dc}
	Same notations as in \autoref{definition upper shriek functor}. Then the following holds:
	\begin{itemize}
		\item if $f$ is an open immersion, then $f^! = (\cdot)|_X$;
		\item if $f$ is proper, then $f^!$ is the right adjoint to $Rf_*$;
		\item if $f$ is separated, then $f^!$ does not depend on the chosen unit dualizing complex.
	\end{itemize}
\end{prop}

\begin{proof}
	\begin{itemize}		
		\item By \autoref{naturality of induced unit dc}, this is equivalent to showing that the duality functor commutes with restriction to open subsets. Hence, we have to show that the restriction of an injective (unit) Cartier module is $\HHom(\cM, -)$-acyclic. 
		
		Let $U \inc X$ be an open immersion, let $\iota \colon \IndCoh_U^{C^r, \unit} \to \QCoh_U$ denote the inclusion, and let $\cI \in \IndCoh_X^{C^r, \unit}$ be an injective object. By \autoref{injective unit is injective O_X} and the fact that restriction of an injective $\cO_X$-module is again injective, we know that $\cI|_U$ is both $\iota$-acyclic (see \autoref{F^flat acyclic implies iota-acyclic}) and $\HHom(\cM, -)$-acyclic. Hence, we win.
		
		\item Given $\cM \in D^b(\Crys_X^{C^r})$ and $\cN \in D^b(\Crys_Y^{C^r})$, we have 
		\begin{align*} 
			\Hom(\cM, f^!\cN) & \cong \Hom(f^*\bD_Y(\cN), \bD_X(\cM)) \\
			& \cong \Hom(\bD_Y(\cN), Rf_*\bD_X(\cM) \\
			& \expl{\cong}{\autoref{compatibility between duality and pushforwards}} \Hom(\bD_Y(\cN), \bD_X(Rf_*\cM)) \\
			& \cong \Hom(Rf_*\cM, \cN). 
		\end{align*}
		
		\item This is an immediate consequence of the previous points and \cite[{\href{https://stacks.math.columbia.edu/tag/0F41}{Tag 0F41}}]{Stacks_Project}.
	\end{itemize}
\end{proof}

\begin{rem}
	One can show that the separatedness hypothesis in \autoref{f^! does not depend on the unit dc} is not necessary, by mimicking the proof of \cite[{\href{https://stacks.math.columbia.edu/tag/0BV3}{Tag 0BV3}}]{Stacks_Project}.
\end{rem}

\begin{lemma}\label{duality preserves supports}
	Let $X$ be a Noetherian $F$-finite $\bF_p$-scheme with a unit dualizing complex $\omega_X^{\bullet}$, and associated duality functor $\bD$. Then for all $\cM \in D^b(\Crys_X^{C^r})$, \[ \Supp_{\crys}(\bD(\cM)) = \Supp_{\crys}(\cM). \]
\end{lemma}
\begin{proof}
	Since duality commutes with restrictions to opens (see the proof of \autoref{f^! does not depend on the unit dc}), the statement is local, and hence we may assume that $X$ is semi-separated. The same argument as in \emph{loc. cit.} shows that the duality functor commutes with stalks, so we are done by \autoref{main thm duality for crystals}.
\end{proof}

\begin{defn}
	Let $f \colon X \to Y$ a separated morphism between Noetherian $\bF_q$-schemes. Define $f_! \colon \Crys_X^{F^r} \to \Crys_Y^{F^r}$ as the unique functor such that the diagram \[ \begin{tikzcd}
		\Crys_X^{F^r} \arrow[rr, "f_!"] \arrow[d, "\Sol"'] &  & \Crys_Y^{F^r} \arrow[d, "\Sol"] \\
		{\Sh_c(X_{\et}, \bF_q)} \arrow[rr, "f_!"']           &  & {\Sh_c(X_{\et}, \bF_q)},          
	\end{tikzcd} \] commutes (on étale $\bF_q$-sheaves, we define $f_!$ as $R^0f_!$, where $Rf_!$ is defined as in \cite[{\href{https://stacks.math.columbia.edu/tag/0F7H}{Tag 0F7H}}]{Stacks_Project}).
\end{defn}

To understand the following result, it makes sense to have in mind the statement of \autoref{pushforward preserves crystals}.\autoref{push_crystals_derived}.

\begin{prop}\label{duality exchanges pushforward and compactly supported pushforward}
	Same notations as in \autoref{definition upper shriek functor}. Then there is a natural isomorphism \[ Rf_* \cong \bD_Y \circ Rf_! \circ \bD_X \] of functors $D^b(\Crys_X^{C^r})^{op} \to D^b(\Crys_Y^{F^r})$. 
\end{prop}
\begin{proof}
	By \cite[{\href{https://stacks.math.columbia.edu/tag/0F41}{Tag 0F41}}]{Stacks_Project} and \autoref{compatibility between duality and pushforwards}, it is enough to consider the case of an open immersion $j \colon U \inj T$. 
	
	By the proof of \autoref{f^! does not depend on the unit dc}, we know that \[ \Dd_U \circ j^* \circ \Dd_T \cong j^*. \] as functors $D^b(\Crys_T^{C^r}) \to D^b(\Crys_U^{C^r})$. Hence, the functor $\Dd_U \circ j_! \circ \Dd_T$ is a right adjoint of $j^*$, so we conclude by \autoref{pushforward preserves crystals}.\autoref{push_crystals_derived} and uniqueness of right adjoints.
\end{proof}

Summarizing, we obtained the following.
\begin{thm}\label{thm_upper_shriek_functor}
	Let $f \colon X \to Y$ be a separated morphism of finite type between Noetherian, $F$-finite and semi-separated $\bF_q$-schemes. Assume that $Y$ has a unit dualizing complex $\omega_Y^{\bullet}$. Then there is a unique unit dualizing complex $\omega_X^{\bullet}$ on $X$ and a functor $f^! \colon D^b(\Crys_Y^{C^r}) \to D^b(\Crys_X^{C^r})$ with the following properties:
	\begin{enumerate}
		\item as complexes of quasi-coherent sheaves, $\omega_X^{\bullet} \cong f^!\omega_Y^{\bullet}$, and the square \[ \begin{tikzcd}
			\omega_X^{\bullet} \arrow[r] \arrow[d] & F^!{\omega_X^{\bullet}} \arrow[d] \\
			f^!\omega_Y^{\bullet} \arrow[r]        & F^!f^!\omega_Y^{\bullet}         
		\end{tikzcd} \] commutes;
		\item we have $Rf_! \circ \bD_X \cong \bD_Y \circ Rf_*$ as functors $D^b(\Crys_X^{C^r})^{op} \to D^b(\Crys_Y^{F^r})$;
		\item we have $f^* \circ \bD_X \cong \bD_Y \circ f^!$ as functors $D^b(\Crys_X^{C^r})^{op} \to D^b(\Crys_Y^{F^r})$;
		\item the functor $f^!$ does not depend on the chosen unit dualizing complex;
		\item if $f$ is an open immersion, $f^! = (\cdot)|_X$;
		\item if $f$ is proper, then $f^!$ is right adjoint to $Rf_*$. In particular, if $f$ is finite, then $f^! = Rf^{\flat}$.
	\end{enumerate}
\end{thm}

We finish this subsection by showing that many schemes have a unit dualizing complex.

\begin{lem}\label{nice regular ring have a unit dualizing complex}
	Let $S$ be a Noetherian $F$-finite $\bF_p$-algebra with the following property: there exists $z_1, \dots, z_n \in S$ such that $F_*S$ is freely generated by the elements $z_1^{i_1}\dots z_n^{i_n}$ with $0 \leq i_j < p$ for all $j$. Then $F^{r, !}S \cong S$. In particular, $S$ has a unit dualizing complex.
\end{lem}
\begin{proof}
	We may assume that $r = 1$ for the first part of the statement. Let $\kappa_S \colon F_*S \ra S$ be the Cartier structure defined by sending $z_1^{i_1}\dots z_n^{i_n}$ to $1$ if $i_s = p - 1$ for all $s$ and $0$ otherwise. We want to show that the induced morphism $\kappa_S^{\flat} \colon S \ra F^\flat S$ is an isomorphism, or equivalently that the map $F_*\kappa_S^{\flat} \colon F_*S \ra \Hom_S(F_*S, S)$ is an isomorphism. By definition, this morphism sends $z_1^{i_1}\dots z_n^{i_n}$ to the function which sends $z_1^{j_1}\dots z_n^{j_n}$ to $1$ if $j_s = p - 1 - i_s$ for all $s$ and $0$ otherwise. Thus, the basis of $F_*S$ is sent to the dual basis of \[\Bcal = \{ z_1^{p - 1 - i_1}\dots z_n^{p - 1 - i_n} \}_{1 \leq i_s \leq p - 1 \esp \forall s},\] so it is an isomorphism.
	
	To show the part after ``In particular'', note that $S$ is regular by Kunz' theorem (see \cite[{\href{https://stacks.math.columbia.edu/tag/0EC0}{Tag 0EC0}}]{Stacks_Project}), and hence Gorenstein (i.e. $S$ is a dualizing complex). Thus, $(S, \kappa_S)$ is in fact a unit dualizing complex.
\end{proof}

\begin{cor}\label{finite type over a unit dc has a unit dc}
	Any semi-separated scheme of finite type over the spectrum of a Noetherian $F$-finite $\bF_p$-algebra has a unit dualizing complex.
\end{cor}
\begin{proof}
	Let $f \colon X \to \Spec R$ be a finite type morphism, where $R$ is Noetherian and $F$-finite. By \autoref{induced unit dualizing complex by a morphism}, it is enough to show that $\Spec R$ has a unit dualizing complex.
	
	By \cite[Remark 13.6]{Gabber_notes_on_some_t_structures}, $\Spec R$ admits a closed immersion to some $\Spec S$, where $S$ satisfies the hypotheses of \autoref{nice regular ring have a unit dualizing complex}. Thus, we are done by \emph{loc.cit}.
\end{proof}

\subsection{Equivalence between Cartier crystals and perverse $\Ff_p$-sheaves}\label{section eq cat}
Our main goal here is to show that there is an equivalence of categories between the category of Cartier crystals and the category of perverse $\Ff_q$-sheaves, with respect to the middle perversity.

From now on, fix a Noetherian, $F$-finite, semi-separated $\bF_p$-scheme $X$ with a unit dualizing complex $\omega_X^\bullet$. Furthemore, let $(A, B) \in \{(F, C), (C, F)\}$. Throughout, we will implicitly use that $X$ has finite dimension (see \cite[Proposition 1.1]{Kunz_Noetherian_Rings_of_char_p}).

\begin{defn}
	Define the \emph{perverse t-structure} associated to the middle perversity on $D^b(\Coh_X^{A^r})$ by the full subcategories \[ ^pD^{\leq 0}_{\coh} \coloneqq \bigset{\Mcal \in D^b(\Coh_X^{A^r})}{\cH^j(\Mcal_x) = 0 \esp \forall j > -\dim\overline{\{x\}} \esp \forall x \in X} \] and \[ ^pD^{\geq 0}_{\coh} \coloneqq \bigset{\Mcal \in D^b(\Coh_X^{A^r})}{H^j_{\mfr_{X, x}}(\Mcal_x) = 0 \esp \forall j < -\dim\overline{\{x\}} \esp \forall x \in X}, \] where $H^j_{\mfr_{X, x}}$ denotes usual local cohomology in the local ring $\Ocal_{X, x}$ (taken in the category of $\Ocal_{X, x}$-modules), or equivalently the functor defined in \autoref{def_local_cohomology}. 
	
	The heart $^pD_{\coh}^{\leq 0} \cap {}^pD_{\coh}^{\geq 0}$ of this t-structure is denoted $\Perv(\Coh_X^{A^r})$.
\end{defn}

\begin{lem}\label{usual perverse t-struct is a t-struct for F/Cartier modules}
	The above data defines a t-structure on $D^b(\Coh_X^{A^r})$. Furthermore, up to shifting $\omega_X^\bullet$ on each connected component, the functor $\Dd$ exchanges this perverse t-structure and the canonical t-structure on $D^b(\Coh_X^{B^r})$.
\end{lem}
\begin{proof}
	Let $\delta$ be the dimension function associated to $\omega_X^\bullet$ (see \cite[{\href{https://stacks.math.columbia.edu/tag/0AWF}{Tag 0AWF}}]{Stacks_Project}). We first show that we may shift $\omega_X^\bullet$ on each connected component, so that $\delta(x) = -\dim\overline{\{x\}}$ for all $x \in X$. 
	
	Since two dimension functions on a connected Noetherian scheme only differ by a shift, it is enough to show that $x \mapsto -\dim\overline{\{x\}}$ is a dimension function. Unwinding the definitions, this is exactly saying that $X$ is catenary, which holds by the existence of a dualizing complex, see \cite[{\href{https://stacks.math.columbia.edu/tag/0AWF}{Tag 0AWF}}]{Stacks_Project}. 
	
	Thus we may assume $\omega_X^\bullet$ is, on each connected component, shifted so that its dimension function is exactly $x \mapsto -\dim\overline{\{x\}}$. Consider the duality $\Dd$ associated to this new unit dualizing complex. We will show that under the equivalence $\Dd$, the canonical t-structure $(D^{\leq 0}, D^{\geq 0})$ corresponds to $({}^pD^{\geq 0}_{\coh}, {}^pD^{\leq 0}_{\coh})$ (this will in particular prove that the perverse t-structure is actually a t-structure).
	Since the definition of the perverse t-structure does not take into account the fact that we work with $A$-modules, it is a statement which is about usual quasi-coherent modules. The functor $\cR\HHom(-, \omega_X^{\bullet})$ exchanges $D^{\leq 0}$ and ${}^pD^{\geq 0}_{\coh}$ by local duality (see \cite[{\href{https://stacks.math.columbia.edu/tag/0A84}{Tag 0A84}}]{Stacks_Project}), and it exchanges $D^{\geq 0}$ and ${}^pD^{\leq 0}_{\coh}$ by \cite[{\href{https://stacks.math.columbia.edu/tag/0A7U}{Tag 0A7U}}]{Stacks_Project} and a standard hypercohomology spectral sequence argument.
\end{proof}

From now on, we assume that $\omega_X^\bullet$ is locally shifted so that it respects the statement of \autoref{usual perverse t-struct is a t-struct for F/Cartier modules}. 

\begin{defn}
	Define the \emph{perverse t-structure} on $D^b(\Crys_X^{A^r})$ associated to $\mathsf{p}$ by \[ {}^pD_{\crys}^{\leq 0} \coloneqq \bigset{\Mcal \in D^b(\Crys_X^{A^r})}{\exists \esp \Ncal \in D^b(\Coh_X^{A^r}) \mbox{ s.t. } \Ncal \sim_A \Mcal \mbox{ and } \Ncal \in {}^pD^{\leq 0}}, \] and define ${}^pD_{\crys}^{\geq 0}$ similarly. Its heart is denoted $\Perv(\Crys_X^{A^r})$.
\end{defn}

\begin{lem}\label{usual perverse t-struct is a t-struct for F/Cartier crystals}
	Assume that $X$ is defined over $\bF_q$. The above data defines a t-structure on $D^b(\Crys_X^{A^r})$, and the duality functor $\Dd$ exchanges this perverse t-structure and the canonical t-structure on $D^b(\Crys_X^{B^r})$. Furthermore, we also have \[ {}^pD_{\crys}^{\leq 0} = \bigset{\Mcal \in D^b(\Crys_X^{A^r})}{\cH^j(\cM_x) \sim_A 0 \esp \forall j > -\dim\overline{\{x\}} \esp \forall x \in X} \] and \[ {}^pD_{\crys}^{\geq 0} = \bigset{\Mcal \in D^b(\Crys_X^{A^r})}{H^j_{\fm_{X, x}}(\cM_x) \sim_A 0 \esp \forall j < -\dim\overline{\{x\}} \esp \forall x \in X}. \]
\end{lem}
\begin{proof}
	Let $(D_{\crys}^{\leq 0}, D_{\crys}^{\geq 0})$ denote the canonical t-structure on $D^b(\Crys_X^{B^r})$, and let $(D^{\leq 0}, D^{\geq 0})$ denote the canonical t-structure on $D^b(\Coh_X^{B^r})$. 
	
	Using the fact that $D^b(\Coh_X^{A^r}) \ra D^b(\Crys_X^{A^r})$ is essentially surjective (see \cite[Theorem 2.6.2]{Bockle_Pink_Cohomological_Theory_of_crystals_over_function_fields}), we see that \[ D_{\crys}^{\leq 0} = \bigset{\Mcal \in D^b(\Coh_X^{A^r})}{\exists \esp \Ncal \in D^b(\Coh_X^{A^r}) \mbox{ s.t. } \Ncal \sim_A \cM \mbox{ and } \Ncal \in D^{\leq 0}}, \] and similarly for $D_{\crys}^{\geq 0}$. Since the diagram \[ \begin{tikzcd}
		D^b(\Crys_X^{A^r}) \arrow[rr, "\Dd"]         &  & D^b(\Crys_X^{B^r})         \\
		D^b(\Coh_X^{A^r}) \arrow[rr, "\Dd"] \arrow[u] &  & D^b(\Coh_X^{B^r}) \arrow[u]
	\end{tikzcd} \] commutes, we deduce that the image of $(D_{\crys}^{\leq 0}, D_{\crys}^{\geq 0})$ under $\Dd$ is exactly $({}^pD_{\crys}^{\leq 0}, \: {}^pD_{\crys}^{\geq 0})$. Since $\Dd$ is an equivalence of categories, the statement before ``Furthermore'' is proven. 

	To obtain the statement after ``Furthermore'', let $H^{\leq 0}$ (resp. $H^{\geq 0}$) denote the right-hand side of the first equality (resp. second equality) that we want to show. The inclusions ${}^pD_{\crys}^{\leq 0} \inc H^{\leq 0}$ and ${}^pD_{\crys}^{\geq 0} \inc H^{\geq 0}$ are immediate. Since $\bD(H^{\geq 0}) \inc D_{\crys}^{\leq 0}$ by local duality (see \autoref{local duality Frobenius version}) and \autoref{duality preserves supports}, we deduce that ${}^pD_{\crys}^{\geq 0} = H^{\geq 0}$.
	
	Now, let $\cM \in H^{\leq 0}$. By a hypercohomology spectral sequence argument, it is enough to show that \[ \cH^i\left(\bD\left(\cH^j(\cM)[-j]\right)\right) = 0 \: \: \forall i < 0, \: \forall j \in \bZ. \]
	Let $Z \coloneqq \Supp_{\crys}(\cM)$ (see \autoref{support F-crystals} and \autoref{support Cartier crystals}), and let $i_Z \colon Z \to X$ denote the corresponding closed immersion. There exists $\cN \in D^b(\Crys_Z^{A^r})$ such that $i_{Z, *}\cN \sim_A \cM$. Indeed, for Frobenius crystals it follows from \autoref{support F-crystals}.\autoref{restricts_and_comes_back}, and for Cartier crystals it follows from \emph{loc. cit.} and duality. Since $\dim(Z) \leq -i$, we are done thanks to \cite[{\href{https://stacks.math.columbia.edu/tag/0A7U}{Tag 0A7U}}]{Stacks_Project}.
\end{proof}
 
\begin{defn}[{\cite{Gabber_notes_on_some_t_structures}}]\label{def:perverse F_q-sheaves}
	Define the \emph{perverse t-structure} on $D^b_c(X_{\et}, \bF_q)$ associated to $\mathsf{p}$ to be given by \[ {}^pD^{\leq 0} \coloneqq \bigset{\Fcal \in D^b_c(X_{\et}, \Ff_q)}{\cH^j(\Fcal)_{\overline{x}} = 0 \esp \forall j > -\dim\overline{\{x\}} \esp \forall x \in X} \] and \[ {}^pD^{\geq 0} \coloneqq \bigset{\Fcal \in D^b_c(X_{\et}, \Ff_q)}{\cH^j(i_{\overline{\{x\}}}^!\Fcal)_{\overline{x}} = 0 \esp \forall j < -\dim\overline{\{x\}} \esp \forall x \in X}. \]
\end{defn}

If $X$ is defined over $\bF_q$, then by \cite[Theorem 2.9.1]{Bockle_Pink_Cohomological_Theory_of_crystals_over_function_fields} and \autoref{main thm Bhatt Lurie}, the functor $\Sol \colon D^b(\Crys_X^{F^r}) \to D^b_c(X_{\et}, \bF_q)$ is an equivalence of categories. We want to show that it is perverse t-exact.

\begin{prop}\label{RSol sends perverse to perverse}
	Assume that $X$ is defined over $\Ff_q$. The above data defines a t-structure on $D^b_c(X_{\et}, \Ff_q)$. Furthermore, the equivalence of categories $\Sol \colon D^b(\Crys_X^{F^r}) \to D^b_c(X_{\et}, \Ff_q)$ is perverse $t$-exact, and hence induces an equivalence of categories \[ \Perv(\Crys_X^{F^r}) \cong \Perv_c(X, \Ff_q). \]
\end{prop}
\begin{proof}
	This defines a t-structure by \cite[Theorem 10.3]{Gabber_notes_on_some_t_structures} (it is only shown for $r = 1$, but the same proof shows it for any $r \geq 1$).
	
	Recall (\cite[1.3.4]{Beilinson_Bernstein_Deligne_Faisceaux_Pervers}) that one of the subcategories defining a t-structure completely determines the other one. Since $\Sol$ is an equivalence of categories, we then only have to show that \[ \Sol({}^pD^{\leq 0}_{\crys}) = {}^pD^{\leq 0}. \]
	
	\begin{itemize}
		\item We have $\Sol({}^pD^{\leq 0}_{\crys}) \inc {}^pD^{\leq 0}$ by \autoref{support F-crystals}.\autoref{restricts_and_comes_back}.
		\item We have $\Sol({}^pD^{\leq 0}_{\crys}) \cni {}^pD^{\leq 0}$ by \autoref{support F-crystals} and the part after ``Furthermore'' in \autoref{usual perverse t-struct is a t-struct for F/Cartier crystals}.
	\end{itemize} 
\end{proof}

\begin{theorem}\label{eq cat Cartier crystals and perverse Fp sheaves}
	Let $X$ be a Noetherian F-finite $\bF_q$-scheme with a unit dualizing complex $\omega_X^\bullet$. Then up to shifting $\omega_X^\bullet$ on each connected component, the composition $\Sol \circ \: \Dd$ induces an equivalence of categories \[ (\Crys_X^{C^r})^{op} \cong \Perv_c(X, \Ff_q), \] where $\Perv_c(X, \Ff_q)$ denotes the heart of the perverse t-structure on $D^b_c(X_{\et}, \Ff_q)$ associated to the middle perversity.
	
	Furthermore, this equivalence preserves derived proper pushforwards.
\end{theorem}
\begin{proof}
	The first statement follows from \autoref{RSol sends perverse to perverse} and \autoref{usual perverse t-struct is a t-struct for F/Cartier crystals}. The second one follows from \autoref{compatibility between duality and pushforwards} and \autoref{main thm Bhatt Lurie} (see also \autoref{all derived pushforwards are the same for F-modules}).
\end{proof}

\begin{cor}\label{universal_homeo_induce_equivalence_of_Cartier_crystals}
	Let $f \colon X \to Y$ be a finite morphism between Noetherian $F$-finite $\bF_q$-schemes, which is a universal homeomorphism. Then the functor $f_* \colon \IndCrys_X^{C^r} \to \IndCrys_Y^{C^r}$ is an equivalence of categories., with inverse given by $f^{\flat}$.
	
	In particular, if $Y$ is semi-separated, has a unit dualizing complex $\omega_Y^{\bullet}$ and if we set $\omega_X^{\bullet} \coloneqq f^!\omega_Y^{\bullet}$, then the natural morphism $f_*\omega_X^{\bullet} \to \omega_Y^{\bullet}$ is an isomorphism in $D^b(\Crys_Y^{C^r})$.
\end{cor}
\begin{proof}
	Note that under these assumptions, $f$ is automatically finite by \cite[{\href{https://stacks.math.columbia.edu/tag/04DF}{Tag 04DF}}]{Stacks_Project}. Hence, the part after ``In particular'' follows from the first statement since $X$ will automatically be semi-separated.
	
	To prove the first statement, we have to show that the unit and counit of the adjunction $f_* \dashv f^{\flat}$ are isomorphisms. This is a local statement on $Y$ so by \autoref{finite type over a unit dc has a unit dc}, we may assume that $Y$ is semi-separated and has a unit dualizing complex $\omega_Y^{\bullet}$. Set $\omega_X^{\bullet} \coloneqq f^!\omega_Y^{\bullet}$.
	
	Since $f$ is a universal homeomorphism, the natural map $\bF_{q, Y} \to f_*\bF_{q, X}$ is an isomorphism. By \cite[{\href{https://stacks.math.columbia.edu/tag/03SI}{Tag 03SI}}]{Stacks_Project}, $f_* : \Sh_c(X_{\et}, \bF_q) \to \Sh_c(Y_{\et}, \bF_q)$ is an equivalence of categories. Since it preserves supports, both $f_*$ and its inverse preserve ${}^pD^{\leq 0}$. Thus, $f_*$ also induces an equivalence on perverse $\bF_q$-sheaves, so by \autoref{eq cat Cartier crystals and perverse Fp sheaves}, we deduce the statement for Cartier crystals. We formally deduce the case of ind-crystals.
\end{proof}

We end this subsection by computing explicitly the composition $\Sol 
\circ \: \bD$. Recall that if $f \colon U \to X$ is an étale morphism, then the natural transformation $\theta \colon f^*F_* \to F_*f^*$ is an isomorphism on quasi-coherent sheaves (see \cite[XIV=XV §1 $n^{\circ}2$, Prop. 2(c)]{SGA5})). In particular, if $\cM \in \QCoh_X^{C^r}$, then $f^*\cM$ naturally acquires the structure of a Cartier module. 

Recall also that there exists a natural isomorphism $\psi \colon f^*F^{r, \flat} \cong F^{r, \flat}f^*$. Locally, if $f$ corresponds to a ring morphism $A \to B$, then for any $A$-module $M$, this natural transformation is given by \[ \begin{tikzcd}
	{\Hom_A(F^r_*A, M) \otimes_A B} \arrow[r]  & {\Hom_B(F^r_*A \otimes_A B, M \otimes_A B)} \arrow[r, "{- \circ \theta_A^{-1}}"] &  {\Hom_B(F^r_*B, M \otimes_A B).}
\end{tikzcd}  \]
	
\begin{lemma}\label{pullback Cartier module by etale maps}
	Let $f \colon U \to X$ be an étale morphism.
	\begin{enumerate}
		\item For any $\cM \in \QCoh_X^{C^r}$, the diagram \[ \begin{tikzcd}
			& f^*\cM \arrow[rd, "\kappa_{f^*\cM}^{\flat}"] \arrow[ld, "f^*\kappa_{\cM}^{\flat}"'] &                      \\
			{f^*F^{r, \flat}\cM} \arrow[rr, "\psi_{\cM}"] &                                                                                     & {F^{r, \flat}f^*\cM}
		\end{tikzcd} \] commutes. In particular, the functor $f^*$ preserves unit Cartier modules.
	\item\label{itm:pullback_commutes_with_Hom} For any $\cM \in \Coh_X^{F^r}$ (resp. $\Coh_X^{C^r}$) and $\cN \in \QCoh_X^{C^r}$ (resp. $\QCoh_X^{C^r, \unit}$), the natural isomorphism \[ f^*\HHom(\cM, \cN) \to \HHom(f^*\cM, f^*\cN) \] is an isomorphism in $\QCoh_X^{C^r}$ (resp. in $\QCoh_X^{F^r}$). In particular, the functors $\bD$ and $(\cdot)_{\et}$ commute. 
	\end{enumerate}
\end{lemma}

\begin{proof}
	\begin{enumerate}
		\item This is a local computation, so we may assume that $f$ is a ring morphism $\phi \colon A \to B$. Let $M$ correspond to $\cM$, and let $m \otimes \lambda \in M \otimes_A B$. We want to show that two morphisms $F^r_*B \to M \otimes_A B$ agree, so let $F^r_*b \in F^r_*B$. 
		
		Since $F^r_*A \otimes_A B \to F^r_*B$ is an isomorphism, we can write $b = \sum_i \phi(a_i)b_i^q$ for some $a_i \in A$ and $b_i \in B$. Then, by definition, \[ \psi_M\left(f^*\kappa_M^{\flat}(m \otimes \lambda)\right)(F^r_*b) = \sum_i \kappa_M(F^r_*(a_im)) \otimes b_i\lambda. \] 
		On the other hand, we have \[ \kappa^{\flat}_{M \otimes_A B}(m \otimes \lambda)(F^r_*b) = \kappa_{M \otimes_A B}(F^r_*(m \otimes b\lambda)) = f^*\kappa_M(\theta_M^{-1}(F^r_*(m \otimes b\lambda))). \] Since $\sum_i \phi(a_i)b_i^q = b$, we deduce that \[ \theta_M\left(\sum_i F^r_*(a_im) \otimes \lambda b_i\right) = F^r_*(m \otimes b\lambda),\] so we are done.

		\item This follows from the definitions and the previous point.
	\end{enumerate}
\end{proof}

\begin{notation}
	Let $A \in \{F, C\}$, and let $\cM$, $\cN \in \QCoh_X^{A^r}$. We denote by $\Hom_{A^r}(\cM, \cN)$ the group of morphisms $\cM \to \cN$ of $A$-modules. The associated $\Hom$-sheaf is denoted $\HHom_{A^r}(\cM, \cN)$. We use the same notation in the étale site.
\end{notation}

\begin{prop}
	For all $\cM \in D^b(\Coh_X^{C^r})$, there is a natural isomorphism \[ \Sol(\bD(\cM)) \cong \cR\HHom_{C^r}(\cM_{\et}, \omega_{X, \et}^{\bullet}). \]
\end{prop}
\begin{proof}
	Let $\cM \in D^b(\Coh_X^{C^r})$. Note that $\Sol \colon \QCoh_X^{F^r} \to \Sh(X_{\et}, \bF_p)$ is given by $\cN \mapsto \HHom_{F^r}(\cO_{X_{\et}}, \cN_{\et})$, where $\cO_X$ is given its canonical Frobenius module structure. We then see from \autoref{pullback Cartier module by etale maps}.\autoref{itm:pullback_commutes_with_Hom} that \[ \Sol(\bD(\cM)) = \cR\HHom_{F^r}(\cO_{X_{\et}}, \cR\HHom(\cM_{\et}, \omega_{X, \et}^{\bullet})). \] Thus, we have to show that \[ \cR\HHom_{F^r}(\cO_{X_{\et}}, \cR\HHom^{\bullet}(\cM_{\et}, \omega_{X, \et}^{\bullet})) \cong \cR\HHom^{\bullet}_{C^r}(\cM_{\et}, \omega_{X, \et}^{\bullet}). \]
	Let $\cI \in D^+(\IndCoh_X^{C^r, \unit})$ be an injective resolution of $\omega_X^{\bullet}$, and let $\cH^{i, j} \coloneqq \HHom(\cM^i, \cI^j)$. By \cite[Construction 3.1.7]{Bhatt_Lurie_RH_corr_pos_char}, there is an exact triangle \[ \cR\HHom_{F^r}(\cO_{X_{\et}}, \HHom^{\bullet}(\cM_{\et}, \cI_{\et})) \to \HHom^{\bullet}(\cM_{\et}, \cI_{\et}) \xrightarrow{\alpha} F^r_*\HHom^{\bullet}(\cM_{\et}, \cI_{\et}) \xrightarrow{+1} \] where $\alpha$ sends a morphism $\phi \colon \cM^i \to \cI^j$ (defined on some étale open) to $\tau_{\cH^{i, j}}(\phi) - F^r_*\phi$.
	
	On the other hand, by applying $\HHom^{\bullet}_{C^r}(-, \cI)$ to the exact sequence in \cite[Lemma 3.1]{Ma_Category_of_F_modules_has_finite_global_dimension} (their regularity hypothesis is not used), we obtain an exact triangle \[ \HHom_{C^r}^{\bullet}(\cM_{\et}, \cI_{\et}) \to \HHom^{\bullet}(\cM_{\et}, \cI_{\et}) \xrightarrow{\beta} \HHom^{\bullet}(F^r_*\cM_{\et}, \cI_{\et}) \xrightarrow{+1} \] where given $\phi \colon \cM^i \to \cI^j$, $\beta(\phi) = \phi \circ \kappa_{\cM^i} - \kappa_{\cI^j} \circ F^r_*f$.
	
	Let $\gamma$ denote the isomorphism \[ \HHom^{\bullet}(F^r_*\cM_{\et}, \cI_{\et}) \cong F^r_*\HHom^{\bullet}(\cM_{\et}, F^{r, \flat}\cI_{\et}). \] Since $\cI$ is unit, we are left to show that \[ \gamma \circ \beta = F^r_*(\kappa_{\cI}^{\flat} \circ -) \circ \alpha \] as morphisms of complexes $\HHom^{\bullet}(\cM_{\et}, \cI_{\et}) \to F^r_*\HHom^{\bullet}(\cM_{\et}, F^{r, \flat}\cI_{\et})$. By definition of $\HHom^{\bullet}$, we may assume that both $\cM$ and $\cI$ are sheaves. This is a local computation, so we only need to show the equality when evaluated at some affine étale open $U = \Spec R$. 
	
	Let $M$ correspond to $\cM|_U$, and $I$ correspond to $\cI|_U$. For all $\lambda \in R$, $\phi \in \Hom(M, I)$ and $m \in M$, we have \[ \gamma(\beta(\phi))(m)(F^r_*\lambda) = \phi(\kappa_M(F^r_*(\lambda m))) - \kappa_I(F^r_*(\lambda \phi(m))). \] On the other hand, by \autoref{construction Hom from Cartier to F}, we have \[ (F^r_*(\kappa_{\cI}^{\flat} \circ -) \circ \alpha)(\phi) = F^r_*\left(F^{r, \flat}\phi \circ \kappa_M^{\flat} - \kappa_I^{\flat} \circ \phi\right). \] Thus, we are done by \autoref{explicit_adjoint_structural_morphism}.
\end{proof}

\begin{rem}\label{Schedlmeier's construction and ours agree}
	As a result, our duality between Cartier crystals and perverse $\bF_p$-sheaves is a generalization of the equivalence from \cite{Schedlmeier_Cartier_crystals_and_perverse_sheaves}.
\end{rem}

\subsection{Finiteness and exactness results}\label{section_finiteness_and_exactness_results}
As a consequence of our results, we obtain new proofs of already well-known results (at least to the experts).

\begin{cor}\cite[Theorem 4.6]{Blickle_Bockle_Cartier_crystals}
	Let $X$ be a Noetherian, $F$-finite $\bF_q$-scheme. Then any object in $\Crys_X^{C^r}$ has finite length.
\end{cor}
\begin{proof}
	This statement is local, so we may assume that $X$ is affine. In particular, we may assume that it has a unit dualizing complex by \autoref{finite type over a unit dc has a unit dc}. Now, the statement is immediate from \autoref{eq cat Cartier crystals and perverse Fp sheaves} and the analogous statement for perverse sheaves, see \cite[Corollary 12.4]{Gabber_notes_on_some_t_structures}.
\end{proof}

\begin{cor}\label{Hom between finite perverse sheaves is finite}
	Let $X$ be a Noetherian, $F$-finite, semi-separated $\bF_q$-scheme that admits a unit dualizing complex. Then for all $\cF_1, \cF_2 \in \Perv_c(X_{\et}, \bF_q)$, the $\Hom$-set in the category of perverse $\bF_p$-sheaves \[ \Hom_{\bF_q}(\cF_1, \cF_2) \] is a finite-dimensional $\bF_q$-vector space.
\end{cor}
\begin{proof}
	This follows from the analogous result for Cartier crystals, see \cite[Theorem 4.17]{Blickle_Bockle_Cartier_modules_finiteness_results}.
\end{proof}

\begin{rem}\label{ex:RHom_does_not_preserve_constructible_sheaves}
	This fact is rather interesting, because it is not true in general that given two $\bF_q$-sheaves $\cF_1$ and $\cF_2$, the groups $\Ext^i(\cF_1, \cF_2)$ are finite (even in the proper case). For example, let $k = \overline{\bF}_q$, let $X = \bP^1_k$, let $x \in X$ be any closed point, and let $U \coloneqq X \setminus \{x\}$ with corresponding open immersion $j \colon U \inj X$. Then \[ \Rcal\HHom(j_!\bF_{q, U}, \bF_{q, X}) \cong Rj_*\bF_{q, U}. \] so \[ \Ext^1(j_!\bF_{q, U}, \bF_{q, X}) \cong H^1(U_{\et}, \bF_{q, U}) = H^1(\bA^1_{k, \et}, \bF_q). \] The Artin-Schreier sequence gives us that this group is exactly the cokernel of $F - 1 \colon k[t] \to k[t]$. This is easily seen to be infinite-dimensional.
\end{rem}

We also have many interesting exactness properties, both at the level of perverse sheaves and of Cartier crystals. After the first draft of this paper was published, some of the following results were reproved in \cite{Bhatt_and_co_Applications_of_perv_sheaves_in_commutative_algebra}.
\begin{cor}\label{Artin_vanishing_and_co}
	Let $f \colon X \to Y$ be a finite type morphism between Noetherian, $F$-finite, semi-separated $\bF_q$-schemes, and assume that $Y$ has a unit dualizing complex.
	\begin{enumerate}
		\item\label{one} the functor $f^! \colon D^b(\Crys_Y^{C^r}) \to D^b(\Crys_X^{C^r})$ is perverse t-exact;
		\item\label{two} if $f$ is étale, then $Rf_* \colon D^b(\Crys_X^{C^r}) \to D^b(\Crys_Y^{C^r})$ is perverse t-exact;
		\item\label{three} if $f$ is separated, then $Rf_! \colon D^b_c(X_{\et}, \bF_q) \to D^b_c(Y_{\et}, \bF_q)$ is right perverse t-exact;
		\item\label{four} if $f$ is affine, then $Rf_! \colon D^b_c(X_{\et}, \bF_q) \to D^b_c(Y_{\et}, \bF_q)$ is perverse t-exact.
	\end{enumerate}
\end{cor}
\begin{proof}
	Throughout, we will be using \autoref{usual perverse t-struct is a t-struct for F/Cartier crystals} and \autoref{RSol sends perverse to perverse} without further mention.
	\begin{enumerate}
		\item This follows from the definition of $f^!$ and exactness of $f^* \colon \Crys_Y^{F^r} \to \Crys_X^{F^r}$ (this is exact by \autoref{main thm Bhatt Lurie}).
		\item This follows from the fact that $f_! \colon \Sh_c(X_{\et}, \bF_q) \to \Sh_c(Y_{\et}, \bF_q)$ is exact.
		\item This follows from the fact that $Rf_* \colon D^b(\Crys_X^{C^r}) \to D^b(\Crys_Y^{C^r})$ is left t-exact and \autoref{duality exchanges pushforward and compactly supported pushforward}.
		\item This follows from the fact that $f_* \colon \Crys_X^{C^r} \to \Crys_Y^{C^r}$ is exact and and \autoref{duality exchanges pushforward and compactly supported pushforward}.
	\end{enumerate}
\end{proof}

We also obtain a version of the proper base change theorem for Cartier crystals.

\begin{cor}\label{proper base change Cartier crystals}
	Let 
	\[ \begin{tikzcd}
		X' \arrow[r, "v"] \arrow[d, "g"'] & X \arrow[d, "f"] \\
		Y' \arrow[r, "u"']                & Y               
	\end{tikzcd} 
	\] be a Cartesian square of Noetherian, $F$-finite and semi-separated $\bF_q$-schemes, with both $f$ and $u$ separated and of finite type. Assume that $Y$ has a unit dualizing complex. Then there is an isomorphism \[ u^!Rf_* \cong Rg_*v^! \] as functors $D^b(\Crys_X^{C^r}) \to D^b(\Crys_{Y'}^{C^r})$. 
	In particular, for any inclusion of a closed subscheme $i \colon Z \inj X$, then \[ i^!i_*\cM \sim_C \cM \] for all $\cM \in \Crys_Z^{C^r}$.
\end{cor}
\begin{proof}
	The first statement follows from the proper base change theorem for étale sheaves (\cite[{\href{https://stacks.math.columbia.edu/tag/0F7L}{Tag 0F7L}}]{Stacks_Project}), \autoref{main thm Bhatt Lurie}, \autoref{eq cat Cartier crystals and perverse Fp sheaves} and the definition of $i^!$.
	
	The second statement then follows by applying the above result to the Cartesian square \[ \begin{tikzcd}
		Z \arrow[r, "id"] \arrow[d, "id"'] & Z \arrow[d, "i"] \\
		Z \arrow[r, "i"']                & X.               
	\end{tikzcd} 
	\] 
\end{proof}
\begin{rem}
	\begin{itemize}
		\item The semi--separatedness condition in the last statement of \autoref{proper base change Cartier crystals} is not required, since it is a local statement (and the construction of $i^!$ does not require semi--separability). Similarly, the semi--separatedness conditions is not required in statements \autoref{two}, \autoref{three} and \autoref{four} of  \autoref{Artin_vanishing_and_co}.
		\item In the same spirit, we automatically obtain a new proof of the existence of the exact triangle \[ i_*i^! \to id \to Rj_*j^* \xrightarrow{+1} \] ($i$ is the inclusion of a closed subscheme, and $j$ the open immersion corresponding to the complement) for ind-coherent Cartier crystals on any Noetherian $F$-finite $\bF_p$-scheme. This was already shown to exist in \cite[Theorem 4.1.1]{Blickle_Bockle_Cartier_crystals} for arbitrary Cartier quasi-crystals. In particular, we have a new proof of Kashiwara's equivalence for Cartier crystals (see \cite[Theorem 4.1.2]{Blickle_Bockle_Cartier_crystals}).
	\end{itemize}
\end{rem}

\section{Generic vanishing for perverse $\overline{\Ff}_p$-sheaves over abelian varieties}\label{section generic vanishing on abelian varieties}

As the title suggests, our final goal is to show a generic vanishing statement for perverse $\overline{\Ff}_p$-sheaves on abelian varieties (see the introduction). Until now, we worked with $\bF_{p^r}$-sheaves, $r$-Frobenius modules and $r$-Cartier modules for some fixed $r \geq 1$. Our first goal is to see how our functors behave when we vary $r$.

Thanks to this, we will be able to pass to $\overline{\bF}_p$-sheaves, which is the right setup four our generic vanishing result.

\subsection{Behaviour of our functors under base change}\label{subsection functors behave well uner base change}
Fix $r$, $s \geq 1$.

\begin{defn}\label{def base change}
	Let $X$ be an $\Ff_p$-scheme, and let $A \in \{F, C\}$. Iterating the action $s$ times defines a functor $D(\cO_X[A^{r}]) \to D(\cO_X[A^{rs}])$. The image of $\cM \in D(\cO_X[A^{r}])$ under this functor is denoted $\cM_s$. Similarly, for $\Fcal \in D(X_{\et}, \Ff_{p^r})$, we define \[ \cF_s \coloneqq \Fcal \otimes_{\Ff_{p^r}} \Ff_{p^{rs}} \in D(X_{\et}, \Ff_{p^{rs}}). \]
\end{defn}

\noindent These functors are what we mean by \emph{base change}.

\begin{lem}\label{Sol preserves base change}
	Let $X$ be a Noetherian scheme over $\Ff_{p^{rs}}$ and let $\Mcal \in D(\IndCrys_X^{F^r})$. There is a natural isomorphism
	\[ \Sol(\Mcal)_s \cong \Sol(\Mcal_s). \] 
\end{lem}
\begin{proof}
	Since a fixed point of $\tau_{\cM}$ is also a fixed point of $\tau_{\cM}^s$, we have a natural morphism of complexes of $\Ff_{p^r}$-sheaves \[ \Sol(\Mcal) \ra \Sol(\Mcal_s). \] We then deduce a natural morphism of $\Ff_{p^{rs}}$-sheaves \[ \Sol(\Mcal) \otimes_{\Ff_{p^r}} \Ff_{p^{rs}} \ra \Sol(\Mcal_s). \] Our goal is to show this is an isomorphism. Since $\Sol$ is exact and preserves colimits, we may assume that $\cM \in \Crys_X^{F^r}$.
	
	It is enough to verify the result at geometric stalks, so we may assume that we work over a strictly Henselian ring $(R, \mfr)$. Let $M$ be the finitely generated Frobenius module corresponding to $\Mcal$. 
	
	We now show that we can reduce to the case of a separably closed field. We have an exact sequence \[ 0 \ra \mfr M \ra M \ra M/\mfr M \ra 0 \] of $R$-modules, and it is also an exact sequence of Frobenius modules (and hence Frobenius crystals). By exactness of $\Sol$, we see that to be able to reduce to the separably closed field case, it is enough to show that the module $\mfr M$ has no non-zero fixed point. 
	
	Let $x \in \mfr M$ be a fixed point, and write $x = rm$ with $r \in \mfr$ and $m \in M$. Then for all $l \geq 1$, \[ x = \tau_M^l(x) = \tau_M^l(rm) = r^{p^{rl}}\tau_M^{rl}(m) \in \mfr^{p^{rl}} M. \] Thus, \[ x \in \bigcap_{n \geq 0} \mfr^n M = 0. \] 
	
	Hence we reduce to the case of a separably closed field $k$. By \cite[Proposition 3.2.9]{Bhatt_Lurie_RH_corr_pos_char}, we have \[ \Sol(M) \cong \Sol(M^{1/p^\infty}) \: \: \mbox{ and } \: \: \Sol(M_s) \cong \Sol\left((M_s)^{1/p^\infty}\right) \cong \Sol\left((M^{1/p^\infty})_s\right) \] (see \cite[Notation 3.2.3]{Bhatt_Lurie_RH_corr_pos_char} for the functor $(\cdot)^{1/p^{\infty}}$). Since $M^{1/p^\infty}$ is a finitely generated Frobenius module over $k^{1/p^\infty}$, we reduce to the case of an algebraically closed field $k$. It then follows from \cite[Corollary p. 143]{Mumford_Abelian_Varieties}.
\end{proof}

\begin{lem}\label{Hom behaves well under base change}
	Let $X$ an $\Ff_p$-scheme (resp. Noetherian and $F$-finite), $\Mcal \in D(\Ocal_X[F^r])$ (resp. $D(\QCoh_X^{C^r})$) and $\Ncal \in D(\Ocal_X[C^r])$ (resp. $D(\QCoh_X^{C^r, \unit})$ or $D(\IndCoh_X^{C^r, \unit})$). Then we have a natural isomorphism \[ \Rcal\HHom(\Mcal, \Ncal)_s \cong \Rcal\HHom(\Mcal_s, \Ncal_s). \] 
\end{lem}
\begin{proof}
	Note that by the explicit definition of the pairings, we immediately have \[ \HHom(\Mcal, \Ncal)_s \cong \HHom(\Mcal_s, \Ncal_s) \] when $\cM$ and $\cN$ are sheaves.
	
	In the case where $\Mcal$ and $\Ncal$ are general complexes, this would then follow from the fact that if $\Ical$ is injective in $\Mod(\Ocal_X[C^r])$ (resp. $\IndCoh_X^{C^r, \unit}$ or $\QCoh_X^{C^r, \unit}$), then $\Ical_s$ is $\HHom(\Mcal, -)$-acyclic. 
	
	When $\Mcal \in D(\cO_X[F^r])$, this follows from \autoref{Hom from F to Cartier}. When $\Mcal \in D(\QCoh_X^{C^r})$, we would be done (see \autoref{Hom from Cartier to F}), if we knew that $\Ical_s$ was $\iota$-acyclic, where $\iota$ denotes either $\IndCoh_X^{C^{rs}, \unit} \to \QCoh_X$ or $\QCoh_X^{C^{rs}, \unit} \to \QCoh_X$. This holds because of \autoref{injective unit is injective O_X} and \autoref{F^flat acyclic implies iota-acyclic}.
\end{proof}
\begin{cor}\label{duality behaves well under base change}
	Let $X$ be a Noetherian, $F$-finite $\bF_p$-scheme with a unit dualizing complex, and induced duality functor $\bD$. Then for all $\Mcal \in D^b(\Coh_X^{F^r}) \cup D^b(\Coh_X^{C^r})$, we have \[ \Dd(\Mcal)_s \cong \Dd(\Mcal_s) \]
\end{cor}
\begin{proof}
	By definition of $\Dd$ and by \autoref{Hom behaves well under base change}, it is enough to show that the inverses of the functors $G \colon D^b(\Coh) \to D^b_{\coh}(\QCoh)$ preserve base change. This is immediate from the fact that $G$ does.
\end{proof}

\subsection{Passing to $\overline{\bF}_p$-sheaves}
Fix a Noetherian, $F$-finite and semi-separated scheme $X$ over $\overline{\bF}_p$, with a unit dualizing complex. Thanks to the results in \autoref{subsection functors behave well uner base change}, we deduce from \autoref{eq cat Cartier crystals and perverse Fp sheaves} that there is an equivalence between the two following categories: 
	\begin{equation}\label{equation:eq cat after colimit}
	 	\left(\colim_r \Perv_c(X, \bF_{p^r})\right)^{op} \cong \colim_r \Crys_X^{C^r}, 
	\end{equation} where the transition functors are the base change functors defined in \autoref{def base change}. 

\begin{defn}
	We define \[ \Crys_X^{C^\infty} \coloneqq \colim_r \Crys_X^{C^r}. \]
	Concretely, its objects are \[ \bigsqcup_r \Crys_X^{C^r}, \] and a morphism between $\cM \in \Crys_X^{C^r}$ and $\cN \in \Crys_X^{C^s}$ is an element of \[ \colim_t \Hom_{\Crys_X^{C^{rst}}}(\cM_{st}, \cN_{rt}), \] (i.e. a morphism between some common base changes of $\cM$ and $\cN$, up to base change). Similarly, we define $\Crys_X^{F^\infty}$.
\end{defn}

Let us understand the left-hand side of \autoref{equation:eq cat after colimit}. 

\begin{defn}\label{def:perverse_F_p-bar_sheaves}
	Define the \emph{perverse t-structure} on $D^b_c(X_{\et}, \overline{\bF}_p)$ associated to $\mathsf{p}$ by the same formulas as in \autoref{def:perverse F_q-sheaves}. Its heart is denoted $\Perv_c(X_{\et}, \overline{\bF}_p)$.
\end{defn}

\begin{prop}\label{perverse Fp-bar sheaves are base change of finite perverse sheaves}
	The above defines a t-structure on $D^b_c(X_{\et}, \overline{\bF}_p)$, and we have \[ \Perv_c(X_{\et}, \overline{\bF}_p) \cong \colim_r \Perv_c(X_{\et}, \bF_{p^r}).\] In particular, there is an equivalence of categories \[ \Perv_c(X_{\et}, \overline{\bF}_p)^{op} \cong \Crys_X^{C^\infty}. \]
\end{prop}
\begin{proof}
	By \cite[Theorem 2.9.1]{Bockle_Pink_Cohomological_Theory_of_crystals_over_function_fields} and \cite[{\href{https://stacks.math.columbia.edu/tag/09YV}{Tag 09YV}} and {\href{https://stacks.math.columbia.edu/tag/03SA}{Tag 03SA}}]{Stacks_Project}, we obtain that $D^b_c(X_{\et}, \overline{\bF}_p) \cong D^b(\Sh_c(X_{\et}, \overline{\bF}_p))$ (and similarly for each $\bF_{p^r}$). Let us show that the natural functor \[ \colim_r \Sh_c(X_{\et}, \bF_{p^r}) \to \Sh_c(X_{\et}, \overline{\bF}_p) \] is an equivalence. First, note that for any $r \geq 1$ and étale $\bF_{p^r}$-sheaves $\cF$, $\cG$ with $\cF$ constructible, then \[ \Hom_{\overline{\bF}_p}(\cF \otimes_{\bF_{p^r}} \overline{\bF}_p, \cG \otimes_{\bF_{p^r}} \overline{\bF}_p) \cong \Hom_{\bF_{p^r}}(\cF, \cG) \otimes_{\bF_{p^r}} \overline{\bF}_p. \] Indeed, by \cite[{\href{https://stacks.math.columbia.edu/tag/095N}{Tag 095N}}]{Stacks_Project} and flatness of $\bF_{p^r} \to \overline{\bF}_p$, it is enough to show the result when $\cF$ is of the form $j_!\bF_{p^r}$, where $j : V \to X$ is étale. This is immediate.
	
	With this observation, faithful flatness is immediate. Essential surjectivity follows from \cite[{\href{https://stacks.math.columbia.edu/tag/095N}{Tag 095N}}]{Stacks_Project}.
	
	In particular, since our base change functors are exact, we formally obtain an induced equivalence of categories \[ D^b_c(X_{\et}, \overline{\bF}_p) \cong \colim D^b_c(X_{\et}, \bF_{p^r}). \] Given that the functors appearing in the definitions of the middle perverse t-structure on étale sheaves preserve base change, the result is proven.
\end{proof}

\begin{cor}
	Let $\cF$, $\cG \in \Perv_c(X_{\et}, \overline{\bF}_p)$. Then the $\Hom$-set $\Hom(\cF, \cG)$ in the category $\Perv_c(X_{\et}, \overline{\bF}_p)$ is a finite-dimensional $\overline{\bF}_p$-vector space.
\end{cor}
\begin{proof}
	By \autoref{perverse Fp-bar sheaves are base change of finite perverse sheaves} and its proof, there exists $r \geq 0$ and $\cF'$, $\cG' \in \Perv_c(X_{\et}, \bF_{p^r})$ such that 
	\[ \begin{cases*}
		\cF \cong \cF' \otimes_{\bF_{p^r}} \overline{\bF}_p; \\
		\cF \cong \cF' \otimes_{\bF_{p^r}} \overline{\bF}_p; \\
		\Hom_{\overline{\bF}_p}(\cF, \cG) \cong \Hom_{\bF_{p^r}}(\cF', \cG') \otimes_{\bF_{p^r}} \overline{\bF}_p.
	\end{cases*} \]
	The result now follows from \autoref{Hom between finite perverse sheaves is finite}.
\end{proof}

\begin{notation}
	Given a Noetherian ring $\Gamma$, we let $\Loc^1(X_{\et}, \Gamma)$ denote the group of rank one étale $\Gamma$-local systems. 
	
	Moreover, given $n \geq 1$, $\Pic(X)[n]$ denotes the subgroup of $n$-torsion line bundles, and $\Pic(X)^{(p)}$ dentoes the subgroup of prime-to-$p$-torsion line bundles.
\end{notation}

\begin{lemma}\label{local systems and torsion line bundles}
	Assume that $X$ is proper over a separably closed field $k$, and let $r \geq 1$. The group morphism $\theta \colon \Loc^1(X_{\et}, \bF_{p^r}) \to \Pic(Y)$ given by \[ \cL \mapsto \left(\cL \otimes_{\bF_{p^r}} \cO_{X_{\et}}\right)\Big|_{X_{\Zar}} \] induces an isomorphism \[ \Loc^1(X_{\et}, \bF_{p^r}) \cong \Pic(Y)[p^r - 1]. \] Furthermore, for all $\cL \in \Loc^1(X_{\et}, \bF_{p^r})$, $\Sol(\theta(\cL)) = \cL$, where the Frobenius module structure on $\theta(\cL)$ is induced by that on $\cO_X$. In particular, \[ \Pic(X)^{(p)} \cong \Loc^1(X_{\et}, \overline{\bF}_p). \]
\end{lemma}
\begin{proof}
	Since $X$ is proper over a separably closed field, we know by the Kummer sequence that \[ \Loc^1(X_{\et}, \bF_{p^r}) \cong H^1(X_{\et}, \bF_{p^r}) \cong \Pic(X)[p^r - 1]. \] Tracking down these isomorphisms explicitly, we see that this composition is exactly $\theta$. In addition, there is a natural morphism $\cL \to \theta(\cL)$, which by definition maps into $\Sol(\theta(\cL))$. To show that it is an isomorphism, we can work étale locally, and hence assume that $\cL$ is trivial. In this case, this boils down to the fact that \[ \Sol\left((\cO_X, (\cdot)^{p^r})\right) = \bF_{p^r}, \] which is immediate.
	
	The statement after ``In particular'' follows, since \[ \Loc^1(X_{\et}, \overline{\bF}_p) \cong H^1(X_{\et}, \GL_1(\overline{\bF}_p)) \expl{\cong}{see \cite[{\href{https://stacks.math.columbia.edu/tag/073E}{Tag 073E}}]{Stacks_Project}} \colim_rH^1(X_{\et}, \GL_1(\bF_{p^r})) \cong \Pic(X)^{(p)}. \]
\end{proof}

\subsection{The generic vanishing}
Let $A$ be an abelian variety of dimension $g$ over an algebraically closed field $k$ of characteristic $p > 0$, with dual abelian variety $\bighat{A} = \Pic^0(A)$. Since $k$ is in particular perfect, it has a canonical unit dualizing complex given by $\Ocal_{\Spec k}$, together with the inverse of the Frobenius. If $f \colon A \ra \Spec k$ denotes the structural map, then by \autoref{induced unit dualizing complex by a morphism}, $f^!\Ocal_{\Spec k}$ is also naturally a unit dualizing complex on $A$. By \cite[{\href{https://stacks.math.columbia.edu/tag/0BRT}{Tag 0BRT}}]{Stacks_Project}, $p^!\Ocal_{\Spec k} \cong \omega_{A/k}[g]$, so we obtain the structure of a unit Cartier module on $\omega_{A/k}$. By \cite[Lemmas 4.1 and 4.4]{Stabler_Test_Module_filtrations_for_unit_F_modules}, this Cartier structure agrees with the usual Cartier operator on top forms. \\

Recall that $\bighat{A}^{(p)}$ denotes the subset of prime-to-$p$ torsion points of $\bighat{A}$.

\begin{defn}
	Let $\cF \in \Perv_c(X_{\et}, \overline{\bF}_p)$. For $i \in \bZ$, We define \[ S^i(\cF) \coloneqq \bigset{\cL \in \Loc^1(X_{\et}, \overline{\bF}_p)}{H^i(X_{\et}, \cF \otimes \cL) \neq 0}. \] By \autoref{local systems and torsion line bundles}, we see $S^i(\cF)$ as a subset of $\bighat{A}^{(p)}$. In particular, we can consider its closure $\overline{S^i(\cF)} \inc \bighat{A}$.
\end{defn}

\begin{theorem}\label{main thm generic vanishing}
	Let $\cF \in \Perv_c(A_{\et}, \overline{\bF}_p)$. Then the following holds:
	\begin{enumerate}
		\item\label{itm:vanishing_for_some_degrees} $S^i(\Fcal) = \emptyset$ for all $i \notin \{-g, \dots, 0\}$;
		\item\label{itm:bound_on_codimension} for all $0 \leq i \leq g$, $\codim \overline{S^{-i}(\Fcal)} \geq i$. In particular, for $\cL \in \Loc^1(A, \overline{\bF}_p)$ general, \[ H^i(A_{\et}, \cF \otimes \cL) = 0 \: \: \: \forall i \neq 0 \] and hence \[ \chi(A_{\et}, \cF \otimes \cL) \coloneqq \sum_{i \in \bZ} (-1)^i\dim H^i(A_{\et}, \cF \otimes \cL) \geq 0. \]  
		\item\label{itm:vanishing_support_loci_are_p_stable} There exists $r \geq 0$ such that for all $i \in \bZ$, $[p^r]S^i(\Fcal) = S^i(\Fcal)$,
	\end{enumerate}
	where $[p^j]$ means multiplication by $p^j$ in the group law of $\bighat{A}$. If $A$ is ordinary, then in addition we have:
	\begin{enumerate}[start=4]
		\item\label{itm:inclusion_of_vanishing_support_loci} $S^{-g}(\Fcal) \inc \dots \inc S^0(\Fcal)$;
		\item\label{itm:vanishing_support_loci_are_linear} each $\overline{S^i(\Fcal)}$ is a finite union of torsion translates of abelian subvarieties;
		\item\label{itm:Euler_char_non_negative} $\chi(A_{\et}, \cF) \geq 0$.
	\end{enumerate}
\end{theorem}

\begin{proof}
	Let $r \geq 0$ be such that $\cF = \cF_0 \otimes_{\bF_{p^r}} \overline{\bF}_p$, for some $\cF_0 \in \Perv_c(A_{\et}, \bF_{p^r})$ (see \autoref{perverse Fp-bar sheaves are base change of finite perverse sheaves}). Let $\cL \in \Loc^1(A_{\et}, \overline{\bF}_p)$, and assume that $\cL = \cL_0 \otimes_{\bF_{p^r}} \overline{\bF}_p$ for some $\cL_0 \in \Loc(A_{\et}, \bF_{p^r})$ (see \autoref{local systems and torsion line bundles}). Note that we do not lose any generality, since we can replace $r$ by any of its multiples. 
	
	Let $\RH \colon \Sh_c(A_{\et}, \bF_{p^r}) \to \Crys_X^{F^r}$ denote the inverse of $\Sol$ (see \autoref{main thm Bhatt Lurie}), let $\cM_0 \coloneqq \bD(\RH(\cF_0))$ and let $\cE_0 \coloneqq \RH(\cL_0)$. By \autoref{compatibility between duality and pushforwards}, \autoref{main thm Bhatt Lurie} and \autoref{behaviour Hom and tensor product}, we have 
	\[ H^{-i}(A_{\et}, \cF_0 \otimes \cL_0) \neq 0 \iff \lim H^i(A, F^{er}_*(\cM_0 \otimes \cE_0^{-1})). \] 
	To simply notations, let $\kappa$ denote the Cartier module $\cM_0 \otimes \cE_0^{-1}$ (see \autoref{rem def Cartier module}.\autoref{tensor product of Cartier mod and unit F-mod}). Using the projection formula and the isomorphism $F^{r, *}\cE_0^{-1} \cong \cE_0^{-1}$, the projective system 
	\[ \begin{tikzcd}
		\dots \arrow[rr] && F^{(e + 1)r}_*(\cM_0 \otimes \cE_0^{-1}) \arrow[rr, "F^{er}_*\kappa"] && F^{er}_*(\cM_0 \otimes \cE_0^{-1}) \arrow[rr] && \dots
	\end{tikzcd} \] is isomorphic to the projective system \[ \begin{tikzcd}
		\dots \arrow[rr] && F^{(e + 1)r}_*\cM_0 \otimes \cE_0^{-1} \arrow[rr, "F^{er}_*\kappa_{\cM_0} \otimes id"] && F^{er}_*\cM_0 \otimes \cE_0^{-1} \arrow[rr] && \dots.
	\end{tikzcd} \]
	Thus,
	\begin{equation}\label{equation_generic_vanishing}
		H^{-i}(A_{\et}, \cF \otimes \cL) \neq 0 \iff \lim H^i(A, F^{er}_*\cM_0 \otimes \cE_0^{-1}) \neq 0.
	\end{equation} 
	By \autoref{Hom behaves well under base change} and \autoref{Sol preserves base change}, we deduce that \[ S^{-i}(\cF)^{-1} = \bigset{\cE \in \bighat{A}(k)}{\lim H^i(A, F^{er}_*\cM_0 \otimes \cE) \neq 0} \cap \bighat{A}^{(p)} \] where for $B \inc \bighat{A}$, we denote by $B^{-1}$ the image of $B$ under the inverse map $\bighat{A} \to \bighat{A}$.
	
	In particular, we automatically obtain \autoref{itm:vanishing_for_some_degrees}. By \cite[Theorem 3.3.4]{Baudin_Generic_vanishing_theory_in_positive_characteristic} and \autoref{equation_generic_vanishing}, there exist closed subsets $W^i \inc \bighat{A}$ of codimension $\geq i$ such that \[ [p^r]W^i = W^i \: \: \mbox{   and   } \: \: S^{-i}(\Fcal)^{-1} \inc \bigcup_{s \geq 0} [p^{rs}]^{-1}W^i. \]
	Since $S^{-i}(\Fcal)^{-1} \inc \bighat{A}^{(p)}$, we deduce that \[ S^{-i}(\Fcal)^{-1} \inc \bigcup_{s \geq 0}\left([p^{rs}]^{-1}W^i \cap \bighat{A}^{(p)}\right) = \bighat{A}^{(p)} \cap \bigcup_{s \geq 0}[p^{rs}]^{-1}\left(W^i \cap \bighat{A}^{(p)}\right).  \] 
	Note that for any $x$, $y \in \bighat{A}^{(p)}$, \[ \exists e \geq 0, \mbox{ s.t. } [p^{re}](x) = y \: \iff \: \exists f \geq 0, \mbox{ s.t. } [p^{rf}](y) = x. \] Thus, \[ S^{-i}(\Fcal)^{-1} \inc \bigcup_{s \geq 0}[p^{rs}]\left(W^i \cap \bighat{A}^{(p)}\right) \inc W^i \cap \bighat{A}^{(p)} \inc W^i.\]
	
	This proves \autoref{itm:bound_on_codimension}. Point \autoref{itm:vanishing_support_loci_are_p_stable} follows from \cite[Lemma 4.1]{Baudin_Generic_vanishing_theory_in_positive_characteristic}, and \autoref{itm:Euler_char_non_negative} follows from \cite[Proposition E]{Baudin_Euler_characteristic_of_weakly_ordinary_irregular_varieties}. Finally, \autoref{itm:vanishing_support_loci_are_linear} follows from \autoref{itm:vanishing_support_loci_are_p_stable} and \autoref{rem_structure_of_p-stable_subsets}.
\end{proof}

\begin{rem}\label{rem_structure_of_p-stable_subsets}
	The condition $[p^r]Z = Z$ for a closed subset $Z \inc \bighat{A}$ is a rather strong condition. Indeed, first note that there exists $s > r$ such that $[p^s]Z_i = Z_i$ for each irreducible component $Z_i$ or $Z$. 
	
	Then we can apply \cite[Theorem 3.1]{Pink_Roessler_Manin_Mumford_conjecture} on each $Z_i$. Let $B_i \coloneqq \Stab_A(Z_i)^{\red}$. Then there exists finitely many morphisms of abelian varieties $h_{\alpha} \colon A_{\alpha} \to \bighat{A}/B_i$, where each $A_{\alpha}$ is supersingular and defined over a finite field, and there exist irreducible closed subsets $X_{\alpha} \inc A_{\alpha}$ such that \[ h_i \coloneqq \sum_{\alpha}h_{\alpha} \colon  \prod_{\alpha}A_{\alpha} \to \bighat{A}/B_i \] has finite kernel, and for some $\overline{a}_i \in A/B$, \[ Z_i/B = \overline{a} + h\left(\prod_{\alpha}X_{\alpha}\right). \]
	In particular, if $A$ has no supersingular factor, then $Z$ is a finite union of torsion translates of abelian subvarieties.
\end{rem}

\bibliographystyle{alpha}
\bibliography{Bibliography}

\Addresses
 
\end{document}